\documentclass[preprint]{imsart}

\usepackage{xr}

\RequirePackage{amsthm,amsmath,amsfonts,amssymb}
\RequirePackage[numbers,sort&compress]{natbib}
\RequirePackage[colorlinks,citecolor=blue,urlcolor=blue]{hyperref}
\RequirePackage{graphicx}

\usepackage{verbatim}
\usepackage{caption}
\usepackage{subcaption}
\usepackage{multirow}
\usepackage{tabularx}

\usepackage[margin=1.5in]{geometry} 

\newtheoremstyle{exampstyle}
{6pt} 
{6pt} 
{\it} 
{} 
{\bfseries} 
{.} 
{.5em} 
{} 

\theoremstyle{exampstyle}

\newtheorem{theorem}{Theorem}[section]
\newtheorem{proposition}[theorem]{Proposition}
\newtheorem{lemma}[theorem]{Lemma}

\newcommand{\E}{{\mathbb E}}
\newcommand{\Q}{{\mathbb Q}}

\newcommand{\R}{{\mathbb R}}
\renewcommand{\P}{{\mathbb P}}

\newcommand{\A}{{\mathcal{A}}}
\newcommand{\B}{{\mathcal{B}}}
\newcommand{\C}{{\mathcal{C}}}

\newcommand{\G}{{\mathcal{G}}}

\newcommand{\F}{{\cal F}}
\newcommand{\samp}{{\mathcal{X}}}

\newcommand{\finiteset}{\mathbf{V}}

\newcounter{rcnt}[section]

\newcommand{\set}[2]{\left\{ #1  \;\middle\vert\; #2 \right\}}

\newcommand{\cN}{{\mathcal N}}
\newcommand{\cG}{{\mathcal G}}
\newcommand{\cK}{{\mathcal K}}

\newcommand{\eps}{{\varepsilon}}

\newcommand{\empdes}{{\mathbb{P}_n}}
\newcommand{\disdes}{{\mathbb{P}}}

\newcommand{\boucon}{{\mathcal{C}^B(\Omega)}}

\newcommand{\metemp}{{\ell_{\empdes}}}
\newcommand{\metdis}{{\ell_{\disdes}}}
\newcommand{\ball}{{\mathfrak{B}}}
\newcommand{\balld}{{\mathfrak{B}_d}}
\newcommand{\one}{{\mathbf{1}}}
\newcommand{\lincon}{{\mathfrak{C}}}
\newcommand{\lin}{{\mathfrak{L}}}
\newcommand{\minmax}{{\mathfrak{R}}}
\newcommand{\grid}{\mathcal{S}}
\newcommand{\twopiece}[4]{\left\{\begin{array}{ll}{#1}&{#2}\\{#3}&{#4}\end{array}\right.}

\def\qt#1{\qquad\text{#1}}

\def\argmin{\mathop{\rm argmin}}
\def\argmax{\mathop{\rm argmax}}

\begin{document}
\begin{frontmatter}
\title{Convex Regression in Multidimensions: Suboptimality of Least
  Squares Estimators}
\runtitle{Convex Regression in multidimensions}

\begin{aug}
\author[A]{\fnms{Gil}~\snm{Kur}\ead[label=e3]{gil.kur@inf.ethz.ch}},
\author[B]{\fnms{Fuchang}~\snm{Gao}\ead[label=e1]{fuchang@uidaho.edu}},
\author[C]{\fnms{Adityanand}~\snm{Guntuboyina}\ead[label=e2]{aditya@stat.berkeley.edu}},
\and
\author[D]{\fnms{Bodhisattva}~\snm{Sen}\ead[label=e4]{bodhi@stat.columbia.edu}}


\runauthor{Kur, G., Gao, F., Guntuboyina, A. and Sen, B.}

\address[A]{Institute for Machine Learning, ETH Z\"{u}rich\printead[presep={,\ }]{e3}}

\address[B]{Department of Mathematics, University of Idaho\printead[presep={,\ }]{e1}} 

\address[C]{Department of Statistics, University of California Berkeley\printead[presep={,\ }]{e2}}

\address[D]{Department of Statistics, Columbia University\printead[presep={,\ }]{e4}}
\end{aug}

\begin{abstract}
Under the usual nonparametric regression model with Gaussian errors,
Least Squares Estimators (LSEs) over natural subclasses of convex
functions are shown to be suboptimal for estimating a $d$-dimensional
convex function in squared error loss when the dimension $d$ is 5 or 
larger. The specific function classes considered include: (i) bounded
convex functions supported on a polytope (in random design), (ii)
Lipschitz convex functions supported on any convex domain (in random
design),  (iii) convex functions supported on a polytope (in fixed
design). For each of these classes, the risk of the LSE is proved to 
be of the order $n^{-2/d}$ (up to logarithmic factors) while the minimax
risk is $n^{-4/(d+4)}$, when $d \ge 5$. In addition, the first rate of
convergence results (worst case and adaptive) for the unrestricted
convex LSE are established in fixed-design for polytopal domains for
all $d \geq 1$. Some new metric entropy results for convex functions
are also proved which are of independent interest.    
\end{abstract}


\end{frontmatter}

\section{Introduction}\label{intro}
The main goal of this paper is to show that nonparametric
Least Squares Estimators (LSEs) associated with the constraint of
convexity can be minimax 
suboptimal when the underlying dimension is 5 or larger. Specifically,
we consider regression over a variety of different classes of convex
functions (and covariate designs) and find lower bounds on the rates
of convergence of the corresponding least squares estimator over each
class, and show that, when $d \geq 5$, these rates of convergence do
not match the minimax rate for the given class.  

We work in
the standard nonparametric regression setting for estimating
an unknown \textit{convex} function $f_0 :  \Omega \rightarrow \R$
defined on a known full-dimensional compact 
convex domain $\Omega \subseteq \R^d$ ($d \geq 1$) from observations 
$(X_1, Y_1), \dots, (X_n, Y_n)$ generated via the model: 
\begin{equation}\label{eq:RegMdl}
Y_i = f_0(X_i) + \xi_i, \qquad \mbox{for } i = 1,\ldots, n,
\end{equation}
where $\xi_1,\ldots, \xi_n$ are i.i.d. errors having the
$N(0,\sigma^2)$  distribution, and design points $X_1, \dots, X_n \in
\Omega$ that may be fixed or random. This problem is known as {\it convex 
  regression} and has a long history 
in statistics and related fields. Standard references are
\cite{Hildreth54, HanPled76, GJW01, groeneboom2014nonparametric, SS11,
  K08, LM12, balazs2016convex}, and applications of convex
regression can be found in \cite{Varian82, Varian84, AllonEtAl07,
  Matzkin91, AD03, K11, T10}.

The Least Squares Estimator (LSE) over a class $\F$ of functions on
$\Omega$ is defined as any minimizer of the least squares criterion
over $\F$: 
\begin{equation}\label{genlse}
  \hat{f}_n(\F) \in \argmin_{f \in \F} \sum_{i=1}^n \left(Y_i - f(X_i)
  \right)^2. 
\end{equation}
In order to introduce the specific
LSEs studied in this paper, consider the following four function
classes: 
\begin{enumerate}
\item $\C_L^B(\Omega)$: class of all convex functions on
  $\Omega$ which are uniformly Lipschitz with Lipschitz constant $L$
  and uniformly bounded by $B$. 
\item $\C_L(\Omega)$: class of all convex functions on $\Omega$ that
  are uniformly Lipschitz with Lipschitz constant $L$. There is no
  uniform boundedness assumption on functions in this class. 
\item $\C^B(\Omega)$: class of all convex functions on $\Omega$ that
  are uniformly bounded by $B$. There is no Lipschitz assumption on
  functions in this class. 
\item $\C(\Omega)$: class of all convex functions on $\Omega$. 
\end{enumerate}
Throughout the paper, we assume that $L$ and $B$ are positive constants not
depending on the sample size $n$. We study LSEs over each of the above
function classes and we use the following terminology for these
estimators:
\begin{enumerate}
\item $\hat{f}_n(\C_L^B(\Omega))$: Bounded Lipschitz 
  Convex LSE,
\item  $\hat{f}_n(\C_L(\Omega))$: Lipschitz Convex LSE,
\item   $\hat{f}_n(\C^B(\Omega))$: Bounded Convex LSE,
\item   $\hat{f}_n(\C(\Omega))$: Unrestricted Convex LSE.  
\end{enumerate}
Strictly speaking, $\hat{f}_n (\C_L^B(\Omega))$ should be called the
``Uniformly Bounded Uniformly Lipschitz Convex LSE'' but we omit the
word ``uniformly'' throughout for brevity. Table \ref{nomenclature}
gives references to some previous papers that studied each of these
estimators. For the unrestricted convex LSE, Table \ref{nomenclature} lists
only the papers focusing on the multivariate case (univariate
investigations are in \cite{HanPled76, Dumbgen04, GJW01, CW16, GS16,
  GSvex, CGS15, B16,   C16}). 

\begin{table}
    \begin{tabular}{ | l | l | l | l |}
    \hline
    Function class & LSE & Name & References \\ \hline
    $\C_L^B(\Omega)$ & $\hat{f}_n(\C_L^B(\Omega))$ & Bounded Lipschitz
                                Convex LSE&
                                            \cite{nemirovskij1984signal,
                                            nemirovskij1985rate,
                                            Mammen91, Balazs15} \\ \hline
$\C_L(\Omega)$     & $\hat{f}_n(\C_L(\Omega))$ & Lipschitz Convex LSE
                                & \cite{Lim14, mazumder2019computational} \\ \hline
$\C^B(\Omega)$ & $\hat{f}_n(\C^B(\Omega))$ & Bounded Convex LSE & \cite{HW16}\\
      \hline 
      $\C(\Omega)$ & $\hat{f}_n(\C(\Omega))$ & Unrestricted Convex LSE &
                                                                 \cite{SS11, K08, LM12, mazumder2019computational,
  chen2020multivariate}\\ 
      \hline 
    \end{tabular}
    \caption{A listing of the LSEs studied in this paper}
    \label{nomenclature}
  \end{table}

This paper studies the performance of these LSEs and compares their
rates of convergence with the corresponding minimax rates. For our
results, we assume throughout the paper that the convex body $\Omega$
(which is the domain of the unknown function $f_0$) is translated and
scaled so that 
\begin{equation}
  \label{OmegaAss}
r_d \balld \subseteq \Omega \subseteq  \balld  
\end{equation}
where $\balld$ is the unit ball in $\R^d$ and $r_d$ is a positive constant 
depending on $d$ alone. In particular, this assumption implies that
$\Omega$ has diameter at most 2 and also volume at most 1. Also the
diameter and volume of $\Omega$ are bounded from below by a constant
depending on $d$ alone. The classical John's theorem (see e.g.,
\cite[Lecture 3]{ball1997elementary}) shows that for every convex
body $\Omega$, there exists an affine transformation such that
\eqref{OmegaAss} holds for the transformed body with $r_d = 1/d$. In
the rest of the paper, whenever we say that a result holds 
for ``any convex body $\Omega$'' (such as in Tables
\ref{MinimaxRatesRandomDesign}, \ref{EntropyRates}, \ref{LSEexisting}, 
\ref{randomdesignsummary} and \ref{fixeddesignsummary}), we mean ``any
convex body $\Omega$ satisfying \eqref{OmegaAss}''.

We focus first on the random design setting  where the design
points 
$X_1, \dots, X_n$ are assumed to be independent having the uniform
distribution $\P$ on $\Omega$ (see Subsection \ref{gendensran} for a discussion of more general design assumptions in random design), and work with the loss function    
\begin{equation}\label{ranloss}
  \ell_{\P}^2 \left(f, g \right)  := \int_{\Omega} (f - g)^2 d\P. 
\end{equation}
The minimax rate for a function class $\F$ in the above setting is: 
\begin{equation*}
\minmax_n^{\mathrm{random}}(\F) :=  \inf_{\breve{f}_n} \sup_{f_0 \in \F}
\E_{f_0} \ell^2_{\P}(\breve{f}_n, f_0) 
\end{equation*}
where $\E_{f_0}$ denotes expectation with respect to the joint
distribution of all the observations $(X_1, Y_1), \dots,
(X_n,Y_n)$ when $f_0$ is the true regression function (see
\eqref{eq:RegMdl}), and the infimum is over all estimators
$\breve{f}_n$.

Minimax rates for the above function classes 1--4  are presented
in  Table \ref{MinimaxRatesRandomDesign}. These results are known
except  perhaps for $\C_L(\Omega)$ which we state and prove in this paper as
Proposition \ref{MinimaxLips}. These minimax rates can also be derived
from the corresponding metric entropy rates which are given in Table
\ref{EntropyRates}. Recall that the $\epsilon$-metric entropy of a
function class $\F$ with respect to a metric $\ell$ is defined as the 
logarithm of the smallest number $N(\epsilon, \F, \ell)$ of
 closed balls of radius $\epsilon$ whose union contains $\F$. Classical
results from Yang and Barron \cite{YangBarron} imply that the minimax rate (for the
loss $\ell_{\P}^2$) for a nonparametric function class $\F$
satisfying some natural assumptions (see \cite[Subsection
3.2]{YangBarron} for a listing of these assumptions which are
satisfied for $\C_L^B(\Omega)$ and $\C^B(\Omega)$)  equals
$\epsilon^2_n$ where 
$\epsilon_n$ solves the equation: 
\begin{equation} \label{yb}
n \epsilon^2 \asymp \log N(\epsilon, \F, \ell_{\P}).  
\end{equation}

  \begin{table}
    \begin{tabular}{ | l | l | l | l |}
    \hline
    Class & Minimax Rate & Assumption on $\Omega$ & Reference \\ \hline
    $\C_L^B(\Omega)$ & $n^{-4/(d+4)}$ & any convex body &
                                                          \cite{nemirovskij1985rate}, \cite[Theorem 
4.1]{Balazs15} \\ \hline
$\C_L(\Omega)$     & $n^{-4/(d+4)}$ & any convex body & this paper
                                                        (Proposition \ref{MinimaxLips}) \\ \hline
$\C^B(\Omega)$ & $n^{-4/(d+4)}$ & polytope & \cite[Theorems 2.3, 2.4]{HW16}\\
      \hline 
      $\C^B(\Omega)$ & $n^{-2/(d+1)}$ & ball & \cite[Theorems 2.3, 2.4]{HW16}\\
      \hline 
    \end{tabular}
    \caption{Minimax rates in the loss \eqref{ranloss} under the random design
      setting where $X_1, \dots, X_n$ are i.i.d. from the uniform
      distribution on $\Omega$. For the rate in the last
      row, we assume $d \geq 2$.}
    \label{MinimaxRatesRandomDesign}
  \end{table}

  \begin{table}
    \begin{tabular}{ | l | l | l | l | l | }
    \hline
    Class & $\epsilon$-entropy & $\epsilon$-bracketing entropy & Assumption on $\Omega$ & Reference \\ \hline
    $\C_L^B(\Omega)$ & $\epsilon^{-d/2}$ & $\epsilon^{-d/2}$ & any convex body &
                                                          \cite{Bronshtein76} \\ \hline
$\C_L(\Omega)$     & infinite & infinite & any convex body & ---\\ \hline
$\C^B(\Omega)$ & $\epsilon^{-d/2}$ & $\epsilon^{-d/2}$ & polytope &
                                                                    \cite{gao2017entropy,
                                                                    doss2020bracketing}\\
      \hline 
      $\C^B(\Omega)$ & $\epsilon^{-(d-1)}$ & $\epsilon^{-(d-1)}$ & ball
                                                  &\cite{gao2017entropy}\\ 
      \hline 
    \end{tabular}
    \caption{$\epsilon$-entropy and $\epsilon$-bracketing entropy
      rates in the metric $\ell_{\P}$ for each of the convex function
      classes. For the last row, we assume $d \geq 2$.}   
    \label{EntropyRates}
  \end{table}

For $\C_L^B(\Omega)$, the minimax rate is $n^{-4/(d+4)}$ which is the
square of the solution for $\epsilon$ in \eqref{yb}:  $n\epsilon^2 =
\epsilon^{-d/2}$. For $\C_L(\Omega)$, the metric entropy is infinite
as this class contains every constant function and hence is not
totally bounded. But the minimax rate is still $n^{-4/(d+4)}$ and
can be easily derived as a consequence of the minimax rate result for
$\C_L^L(\Omega)$ (the argument is given in the proof of Proposition
\ref{MinimaxLips}).

For $\C^B(\Omega)$ with $d \geq 2$, as shown by Han and Wellner
\cite{HW16} and Gao and Wellner \cite{gao2017entropy} respectively, the minimax
rate and the metric entropy rate 
change depending on whether $\Omega$ is polytopal or has smooth
boundary. Here (and in the rest of the paper), when we say that a
specific set $S$ is a ``polytope'', we assume that the number of
facets or extreme points of $S$ is bounded by a constant not
depending on the sample size $n$. If the number of facets is allowed
to grow with $n$, then  polytopes can approximate general convex
bodies arbitrarily well in which case it is not meaningful to give
separate rates for polytopes and other convex bodies. 

When $\Omega$ is a polytope, the minimax rate for $\C^B(\Omega)$
equals $n^{-4/(d+4)}$ corresponding to the $\epsilon$-entropy of
$\epsilon^{-d/2}$. In contrast, when $\Omega$ is the unit ball, the
minimax rate is $n^{-2/(d+1)}$ corresponding to the $\epsilon$-entropy
of $\epsilon^{-(d-1)}$. The larger metric entropy for the ball
(compared to polytopes) is driven by the curvature of the boundary and
the lower bound is proved by considering indicator-like functions of
spherical caps (see \cite[Subsection 2.10]{gao2017entropy} for
the metric entropy proofs). More high level intuition on the
differences between the polytopal case and the ball case is provided
in Subsection \ref{ballcase}.    



We now ask whether the LSEs (as defined via \eqref{genlse}) achieve the corresponding minimax
rates. The supremum risk of an estimator $\breve{f}$ on a
function class $\F$ is defined as
\begin{equation*}
  R_n(\breve{f}_n, \F) := \sup_{f \in \F} \E_f
  \ell_{\P}^2(\breve{f}_n, f).
\end{equation*}
Existing results for the convex regression LSEs are described in Table
\ref{LSEexisting} where the rate $r_{n, d}$ is given by  
\begin{equation}\label{balazs}
r_{n, d} := \left\{
  \begin{array}{lr}
  n^{-4/(d+4)}   & : d \leq 4\\
    n^{-2/d} & : d \geq 5. 
  \end{array}
\right.
\end{equation}

  \begin{table}
    \begin{tabular}{ | l | l | l | l |}
    \hline
   Supremum Risk & Result & Assumption on $\Omega$ & Reference \\ \hline
$R_n(\hat{f}_n(\C_L^B(\Omega)),
                                  \C_L^B(\Omega))$ & $\lesssim^* r_{n, d}$ & any convex body & 
                                                          \cite{Balazs15}
      \\ \hline 
$R_n(\hat{f}_n(\C_L(\Omega)),
                                  \C_L(\Omega))$    &
                                                                $\lesssim^*
                                                                r_{n,
                                                                d}$  & any convex
                                                         body &
                                                                      \cite{Lim14,
                                                                      mazumder2019computational}
      \\ \hline 
$R_n(\hat{f}_n(\C^B(\Omega)),
                                  \C^B(\Omega))$ &
                                                                $\lesssim^*
                                                                r_{n,
                                                   d}$ & polytope & \cite{HW16}\\
      \hline 
$R_n(\hat{f}_n(\C^B(\Omega)),
                                  \C^B(\Omega))$ &
                                                          $n^{-2/(d+1)}$
                          & ball & \cite{kur2019optimality}\\ 
      \hline 
    \end{tabular}
    \caption{Existing LSE rates in the loss \eqref{ranloss}
      for random design (for the last
      row, we take $d \geq 2$). Here $\lesssim^*$ refers to an upper bound
      with multiplicative factors that are logarithmic in $n$. We
      prove, in this paper, that the rate $r_{n, d}$ is actually tight
      for each of the first three rows.} 
    \label{LSEexisting}
  \end{table}
The upper bound of $r_{n, d}$ in the first three rows of
Table \ref{LSEexisting} is derived from standard upper bounds on the
performance of LSEs \cite{BM93,  vdG00} which state that
the LSE risk is bounded from above by $\delta_n^2$ where $\delta_n$
solves the equation: 
\begin{equation}\label{LSEeq}
  \sqrt{n} \delta^2 \asymp \int_{\delta^2}^{\delta} \sqrt{\log N_{[\,]}(\epsilon,
  \F, \ell_{\P})} d\epsilon. 
\end{equation}
Here $N_{[\,]}(\epsilon, \F, \ell_{\P})$ denotes the
$\epsilon$-bracketing number of $\F$ under the $\ell_{\P}$ metric
which is defined as the smallest number of $\epsilon$-brackets:
\begin{equation*}
  [\underline{f}, \bar{f}] := \left\{g : \underline{f}(x) \leq g(x)
    \leq \bar{f}(x) \text{ for all } x \right\} ~~~~~\text{ with }
  \ell_{\P}(\underline{f}, \bar{f}) \leq \epsilon
\end{equation*}
needed to cover $\F$. The logarithm of $N_{[\,]}(\epsilon, \F,
\ell_{\P})$ denotes the $\epsilon$-bracketing entropy of
$\F$. Gao and Wellner \cite{gao2017entropy} (see also
Doss \cite{doss2020bracketing}) proved bracketing entropy 
numbers for the  
convex function classes; these are given in Table \ref{EntropyRates}
and they coincide with the metric entropy rates. Plugging in
$\epsilon^{-d/2}$ for the bracketing entropy in \eqref{LSEeq} and
solving the resulting equation in $\delta$ leads to
the rate $r_{n, d}$ (although $\C_L(\Omega)$ has
infinite entropy, results for this class can be derived by working
instead with $\C_L^L(\Omega)$ which has entropy
$\epsilon^{-d/2}$). The split in $r_{n, d}$ for $d \leq 4$ and $d \geq
5$ occurs 
because $\int_{\delta^2}^{\delta} \epsilon^{-d/4} d\epsilon$ depends
differently on $\delta$ for $d \leq 4$ compared to $d \geq 5$. The same
dimension-dependent split in the upper bounds for the rate can be seen
in  \cite{BM93} and \cite[Chapter 9]{vdG00} for  certain other LSEs on
function classes with smoothness restrictions.

It is interesting to note that the last row in Table
\ref{LSEexisting} (corresponding to the case when $\Omega$ is the unit
ball) does not have any 
dimension-dependent split in the rate (it equals $n^{-2/(d+1)}$ for
every $d \geq 2$). This result is 
due to Kur et al. \cite{kur2019optimality}. More details on this case
are given in Subsection \ref{ballcase}. 


We are now ready to describe the main results of this paper. Tables
\ref{MinimaxRatesRandomDesign} and \ref{LSEexisting} together indicate
that the minimax rate for the first three rows equals $n^{-4/(d+4)}$
which also coincides with $r_{n, d}$ for $d \leq 4$. This immediately
implies 
minimax optimality (up to log factors) of the LSE over each function
class in the first three rows of Table \ref{MinimaxRatesRandomDesign}
when $d \leq 4$. However for $d \geq 5$, there is a significant gap
between the minimax rate $n^{-4/(d+4)}$ and $r_{n, d} = n^{-2/d}$. The
question of whether this gap is real (implying suboptimality of the
LSEs) or merely an artifact of some loose argument in the derived
upper bounds was previously open. We settle this question in this
paper by proving that the LSE for each function class in the first
three rows of Table \ref{MinimaxRatesRandomDesign} is minimax
suboptimal for $d \geq 5$. More precisely, in Theorem
\ref{implb.blse}, we prove that $n^{-2/d} 
(\log n)^{-4(d+1)/d}$ is a lower bound (up to a multiplicative factor
independent of $n$) on each supremum risk in the first three rows of
Table \ref{LSEexisting} for $d \geq 5$.

Table \ref{randomdesignsummary} has a summary of the minimax rates and
rates exhibited by the LSE (along with minimax optimality and
suboptimality of the LSE) in our random design setting.  

\begin{table}[h]
\centering
\small
\begin{tabular}{|p{1.15cm}|p{2.5cm}|p{1.5cm}|p{3cm}|p{3cm}|}
\hline
Class $\mathcal{F}$ & Assumption on $\Omega$ & Minimax Rate & Supremum
                                                              LSE
                                                              \newline
                                                              risk
                                                              $R_n(\hat{f}_n(\mathcal{F}),
                                                              \mathcal{F})$
                                                              (upto logs)
  & Minimax Optimality of $\hat{f}_n(\F)$ (upto logs) \\
\hline
$\mathcal{C}_L^B(\Omega)$ & any convex body & $n^{-4/(d+4)}$ & 
$n^{-4/(d+4)}$, $d \leq 4$ 
\newline
$ n^{-2/d}$, $d \geq 5$
& 
Optimal for $d \leq 4$ 
\newline
Suboptimal for $d \geq 5$
\\
  \hline
$\mathcal{C}_L(\Omega)$ & any convex body & $n^{-4/(d+4)}$ & 
$n^{-4/(d+4)}$, $d \leq 4$ 
\newline
$n^{-2/d}$, $d \geq 5$
& 
Optimal for $d \leq 4$ 
\newline
Suboptimal for $d \geq 5$
\\
 \hline
$\mathcal{C}^B(\Omega)$ & polytope & $n^{-4/(d+4)}$ & 
$n^{-4/(d+4)}$, $d \leq 4$ 
\newline
$n^{-2/d}$, $d \geq 5$
& 
Optimal for $d \leq 4$ 
\newline
Suboptimal for $d \geq 5$
\\
  \hline
$\mathcal{C}^B(\Omega)$ & ball ($d \geq 2$) & $n^{-2/(d+1)}$ & 
$n^{-2/(d+1)}$
& 
Optimal for $d \geq 2$ 
\\
\hline  
\end{tabular}
\caption{Summary of the minimax and LSE rates, and
  optimality/suboptimality of the LSE in our random design setting}
\label{randomdesignsummary}
\end{table}

Our proof of minimax suboptimality for the LSE is based on the
following idea. Suppose $f_0(x) := \|x\|^2$ and let $\tilde{f}_k$ be a
piecewise affine convex approximation to $f_0$ with $k$ affine
pieces as described in the statment of Lemma \ref{ballapp}. It is then
well-known (see e.g., \cite[Lemma 4.1]{Balazs15} or 
\cite{BronIvan}) that the $\ell_{\P}$ distance  between $f_0$
and $\tilde{f}_k$ is of order $k^{-2/d}$. If we now set $k$ to be of
order $\sqrt{n}$, then $\ell_{\P}(\tilde{f}_k, f_0) \sim n^{-1/d}$, or
equivalently $\ell_{\P}^2(\tilde{f}_k, f_0) \sim n^{-2/d}$. Note that
$n^{-2/d}$ is much larger than the minimax rate $n^{-4/(d+4)}$ for $d
\geq 5$. We study the behavior of each convex LSE $\hat{f}_n$ (in the
first 
three rows of Table \ref{MinimaxRatesRandomDesign}) when the true
function is $\tilde{f}_k$ for $k \sim \sqrt{n}$, and show (in Theorem
\ref{implb.blse}) that $\E \ell_{\P}^2(\hat{f}_n, \tilde{f}_k)$ is
bounded below by a constant multiple of $n^{-2/d} (\log
n)^{-4(d+1)/d}$. In other words, we prove that the LSEs are at a
squared distance from the true function $\tilde{f}_k$ that is much
larger than the minimax rate when $d \geq 5$. Our proof techniques are
outlined in Section \ref{pfskeran}. We are unable to
definitively say why 
the LSEs are behaving this way when the true function is piecewise
affine with $\sqrt{n}$ pieces. We believe this is due to
overfitting, probably the LSEs have many more affine pieces than $k
\sim \sqrt{n}$ in this case and are actually closer to the quadratic
function $f_0$.  

It has previously been observed that LSEs and related Empirical Risk 
Minimization procedures 
can be  suboptimal for function classes with large entropy
integrals. For example, \cite[Section 4]{BM93} designed certain pathological large function
classes where the LSE is provably minimax suboptimal. The fact that
minimax suboptimality also holds for natural function classes such as
$\C_L^B(\Omega)$ and $\C_L(\Omega)$ (for $d \geq 5$) is
noteworthy.

In contrast to our suboptimality results, the last row of Table
\ref{LSEexisting} and Table \ref{randomdesignsummary} correspond to
the class $\C^B(\balld)$ (here 
$\balld$ is the unit ball) for which the LSE is minimax optimal for
all $d \geq 2$ (this result is due to Kur et
al. \cite{kur2019optimality}). As 
already mentioned, $\C^B(\balld)$ has much larger metric entropy
(compared to $\C^B(\Omega)$ for polytopal $\Omega$) that is driven by
the curvature of $\balld$; one can understand this from the proofs of
the metric entropy lower bounds on $\C^B(\balld)$ and  $\C^B(\Omega)$
for polytopal $\Omega$ in \cite{gao2017entropy} (specifically see
\cite[proof of Theorem 1(i) and proof of Theorem
4]{gao2017entropy}). High-level intuition on the differences
between the ball and polytopal cases is given in Subsection
\ref{ballcase}. Isotonic regression is another shape constrained
regression problem 
where the LSE is minimax optimal for all $d$ (see
Han et al. \cite{han2019isotonic}). The class of coordinatewise
monotone 
functions on $[0, 1]^d$ is similar to $\C^B(\balld)$ in that its
metric entropy is driven by well-separated subsets of the domain $[0, 1]^d$ (see
\cite[Proof of Proposition 2.1]{gao2007entropy}). Other examples of
such classes where the LSE is optimal for all dimensions can be found
in Han \cite{han2019global}.

\subsection{Fixed-design results for the unrestricted convex
  LSE} \label{fd.intro} 

So far we have not discussed any rates of convergence for the
unrestricted Convex LSE 
$\hat{f}_n(\C(\Omega))$. No such results exist in the literature; 
it appears difficult to prove them in the random design setting in
part because the
general random-design theory of Empirical Risk Minimization is largely
restricted to uniformly bounded function classes. In this paper, we
prove rates for 
$\hat{f}_n(\C(\Omega))$ for a specific \textit{fixed-design} setting. We now
assume that $X_1, \dots,  X_n$ 
form a fixed regular rectangular grid intersected with $\Omega$, and
the loss function is  
\begin{equation}\label{ldef}
  \ell_{\P_n}^2(f, g) := \int \left(f - g \right)^2 d\P_n =
  \frac{1}{n} \sum_{i=1}^n \left(f(X_i) - g(X_i) \right)^2
\end{equation}
with $\P_n$ being the (non-random) empirical distribution of $X_1,
\dots, X_n$. This setting is also quite standard in nonparametric
function estimation (see e.g.,  \cite{nemirovski2000}). We only
work with polytopal $\Omega$. Fixed-design
rates for the unrestricted convex LSE when $\Omega$ is nonpolytopal
are likely more complicated (as indicated in Subsection \ref{fdnp}) and
are not  addressed in this paper.   

In this setting, we are able to prove uniform rates of convergence for
$\hat{f}_n(\C(\Omega))$ over a function class that is larger than the
function classes considered so far. This function class is given by: 
\begin{equation}\label{FlinClass}
  \F^{\lin}(\Omega) := \left\{f_0 \text{ convex on } \Omega : \inf_{g \in
      \mathcal{A}(\Omega)} \ell_{\P_n}(f_0, g) \leq \lin \right\} 
\end{equation}
where $\A(\Omega)$ denotes the class of all \textit{affine} functions
on $\Omega$, and $\lin$ is a fixed positive constant. It can be easily
shown (see Section \ref{fdresults} for details) that, as a function
class, $\F^{\lin}(\Omega)$ is larger than both $\C^{\lin}(\Omega)$
and $\C_{\lin}(\Omega)$, which means that $f_0 \in \F^{\lin}(\Omega)$
is a weaker assumption compared to uniform boundedness ($f_0
\in \C^{\lin}(\Omega)$) and also  compared to uniform Lipschitzness
($f_0 \in \C_{\lin}(\Omega)$). Further $\F^{\lin}(\Omega)$
satisfies the natural invariance property: $f_0 \in
\F^{\lin}(\Omega)$ if and only if $f_0 - g_0 \in \F^{\lin}(\Omega)$
for every affine function $g_0$. This invariance is obviously
also satisfied by $\C(\Omega)$ but not by the other classes in Table
\ref{nomenclature}. 

In the aforementioned fixed-design setting, we prove, in Theorem
\ref{wcad}, that 
\begin{equation}\label{worstcaseratedef}
\sup_{f \in \F^{\lin}(\Omega)} \E_{f} \ell_{\P_n}^2
  \left(\hat{f}_n(\C(\Omega)), f
  \right)
\end{equation}
is bounded by the rate $r_{n, d}$ (see \eqref{balazs}) up to logarithmic factors. Thus the
unrestricted convex LSE achieves the rate $r_{n, d}$ under the assumption that
the true function $f_0 \in \F^{\lin}(\Omega)$. That this holds under
the weaker assumption $f_0 \in \F^{\lin}(\Omega)$ can be considered
an advantage of the unrestricted convex LSE over the Bounded or
Lipschitz convex LSE which require the stronger assumptions $f_0 \in
\C^{\lin}(\Omega)$ or $f_0 \in \C_{\lin}(\Omega)$.  

The minimax rate for $\F^{\lin}(\Omega)$ in the fixed-design setting
is defined by 
 \begin{equation}\label{minratedeffix}
   \minmax_n^{\mathrm{fixed}}(\F^{\lin}(\Omega)) :=
   \inf_{\breve{f}_n} \sup_{f \in \F^{\lin}(\Omega)} 
\E_f \ell^2_{\P_n}(\breve{f}_n, f) 
\end{equation}
where $\ell_{\P_n}$ is the loss function \eqref{ldef}. In Proposition \ref{minratelin}, we prove that
$\minmax_n^{\mathrm{fixed}}(\F^{\lin}(\Omega))$ equals 
$n^{-4/(d+4)}$ (up to log 
factors) for all $d \geq 1$. This implies 
that the unrestricted convex LSE is, up to log factors, minimax optimal for $d
\leq 4$ for the class $\F^{\lin}(\Omega)$ (and consequently also for
the smaller classes $\C_{\lin}(\Omega)$ and $\C^{\lin}(\Omega)$).

For $d \geq 5$, we prove, in Theorem \ref{minsub}, that there exists
$f_0 \in \C^L_L(\Omega)$ 
where the risk of $\hat{f}_n$ is bounded from below by $n^{-2/d} (\log
n)^{-4(d+1)/d}$. This proves minimax suboptimality of the unrestricted
convex LSE for $d \geq 5$ over the class $\C_L^L(\Omega)$, and
consequently also over the larger function class $\C_L(\Omega)$,
$\C^L(\Omega)$, $\F^L(\Omega)$ (it is helpful to note here all these
function classes have, up to a logarithmic factor, the same minimax rate
because of Proposition \ref{minratelin} and Proposition
\ref{fixlobo}).  

In addition to these results, we also prove that
the rate of convergence of the unrestricted convex LSE can be faster than
$r_{n, d}$ when $f_0$ is piecewise  
affine. Specifically, we prove that when $f_0$ is a piecewise affine 
convex function on $\Omega$ with $k$ affine pieces, the risk of
$\hat{f}_n$ is bounded from above by
\begin{equation}\label{aknd}
a_{k, n, d} := \left\{
  \begin{array}{lr}
  \frac{k}{n}   & : d \leq 4\\
\left(\frac{k}{n} \right)^{4/d} & : d \geq 5
  \end{array}
\right.
\end{equation}
up to logarithmic factors. When $k$ is not too large, $a_{k, n, d}$ is
smaller than $r_{n, d}$ (which bounds
\eqref{worstcaseratedef}). Specifically, when $d \leq 4$, we have
$a_{k, n, d} = k/n \leq r_{n, d} = n^{-4/(d+4)}$ for $k \leq
n^{d/(d+4)}$, and when $d \geq 5$, we have $a_{k, n, d} = (k/n)^{4/d}
  \leq r_{n, d} = n^{-2/d}$ for $k \leq \sqrt{n}$. For such $k$, we
  say that the unrestricted convex LSE \textit{adapts} to the
  piecewise affine convex function $f_0$. In the random design
  setting, adaptive rates for the bounded convex LSE can be found in
  \cite{HW16} with a suboptimal dependence on $k$. 

It is especially interesting that the adaptive rate $a_{k, n,
  d}$ switches from being parametric for $d \leq 4$ to a slower
nonparametric rate for $d \geq 5$. For $d \geq 5$, we prove a lower
bound showing that the adaptive rate $(k/n)^{4/d}$ cannot be improved
for $k \lesssim \sqrt{n}$. Specifically, in Theorem \ref{implb}, we
prove, assuming $d \geq 5$, that for the piecewise affine convex
function $\tilde{f}_k$ given by Lemma \ref{ballapp}, the rate of the
unrestricted convex LSE is bounded from below by $a_{k, n, d} (\log
n)^{-4(d+1)/d}$ for every $k \lesssim \sqrt{n}$. This lower bound is
related to the minimax suboptimality of $\hat{f}_n(\C(\Omega))$
because when $k \sim \sqrt{n}$, the rate $(k/n)^{4/d}$ becomes
$n^{-2/d}$.   

\subsection{Paper Outline}
The rest of this paper is organized as follows. Minimax rate results
for convex regression are summarized in Section 
\ref{minimax}. These minimax rates are useful for gauging minimax
optimality and suboptimality of LSEs in convex regression. Rates of
convergence of the LSEs are given in Section
\ref{LseRates}. Subsection \ref{sec.mores} contains the main
minimax suboptimality result for the LSEs over the three classes
$\C_L^B(\Omega)$, $\C^B(\Omega)$ and $\C_L(\Omega)$ in the random
design setting. Section \ref{fdresults} contains results for the
unrestricted convex LSE in fixed design. Section \ref{pfskeran}
provides a sketch of the main ideas and ingredients behind the proofs
of the LSE rate results in Section \ref{LseRates}. Proofs of all
results can be found in the Appendices: Appendix \ref{minimax.proofs}
contains proofs of all results 
in Section \ref{minimax}, Appendix \ref{rd.proofs} contains proofs of
all results in Section \ref{sec.mores} (and the auxiliary results
stated in Subsection \ref{pfskeblse}) and Appendix \ref{fd.proofs}
contains proofs of all results in Section \ref{fdresults} (and the
auxiliary results stated in Subsection \ref{pfskefdr}). The paper ends
with a discussion section (see Section \ref{discu}) where some issues
naturally connected to our results are discussed.

\subsection{A note on constants underlying our rate
  results}\label{consta}
Our rate results often involve various constants which are generically
denoted by $c$ or $C$. These constants never depend on the sample size
$n$ but they often depend on various other parameters involved in the
problem (such as the dimension $d$, Lipschitz constant $L$, uniform
bound $B$, number of facets of the polytopal domain $\Omega$, etc.). We
have tried to indicate the dependence of the constant explicity using
notation such as $C_{d, B, L}$. We have not attempted to explicitly
specify the dependence of these constants on these parameters. Please
note that the value of these constants may change from occurrence to
occurrence. 

\section{Minimax Rates for Convex Regression}\label{minimax}

In this section, we provide more details for the minimax rate results
 mentioned in the Introduction. 

\subsection{Random Design}
Minimax rates for convex regression in random design are
given in Table \ref{MinimaxRatesRandomDesign}. The first result
is for the class $\C_L^B(\Omega)$ and, as indicated in Table
\ref{MinimaxRatesRandomDesign}, the minimax rate of $n^{-4/(d+4)}$ for
$\C_L^B(\Omega)$ was stated and proved in \cite[Theorem
4.1]{Balazs15}.  However, they additionally assumed that $\Omega$ is a
rectangle which is unnecessary and the same proof applies to every
convex body $\Omega$ satisfying \eqref{OmegaAss}. This is because the
metric entropy of $\epsilon^{-d/2}$ holds for $\C_L^B(\Omega)$ for
every convex body $\Omega$ satisfying \eqref{OmegaAss} as can be
readily seen from the proof of \cite[Theorem 6]{Bronshtein76}.

The next minimax result is for $\C_L(\Omega)$. Here also the rate is
$n^{-4/(d+4)}$ and it holds for every convex body $\Omega$ satisfying
\eqref{OmegaAss}. The upper bound here is slightly nontrivial
 because the class $\C_L(\Omega)$ has infinite metric
entropy as it contains 
all constant functions. We could not find a reference for this result
in the literature so we state it below and include a proof in Appendix
\ref{minimax.proofs}.  

\begin{proposition}\label{MinimaxLips}
  For every $\Omega$ satisfying \eqref{OmegaAss}, we have
  \begin{equation}\label{MinimaxLips.eq}
    c_{d, L, \sigma} n^{-4/(d+4)} \leq
    \minmax_n^{\mathrm{random}}(\C_L(\Omega)) \leq C_{d, L, \sigma}
    n^{-4/(d+4)}
  \end{equation}
  where $c_{d, L, \sigma}$ and $C_{d, L, \sigma}$ are positive
  constants depending only on  $d, L, \sigma$. 
\end{proposition}
Next are the results for $\C^B(\Omega)$ where the minimax rates change
depending on whether $\Omega$  is polytopal or not. When $\Omega$ is a
polytope, the minimax rate is $n^{-4/(d+4)}$ for all $d \geq 1$ and
when $\Omega$ is a ball, the rate is $n^{-2/(d+1)}$ for $d \geq
2$. These results (whose proofs can be found in \cite{HW16}) can be
derived as a consequence of the metric entropy rates in Table
\ref{EntropyRates} ($\epsilon^{-d/2}$ and $\epsilon^{-(d-1)}$
respectively) by solving the equation \eqref{yb}. 

\subsection{Fixed Design}\label{minfd}
For fixed design, we are mostly only able to prove results when $\Omega$ is a polytope (see Subsection \ref{fdnp} for some remarks on fixed-design results for non-polytopal $\Omega$). Specifically, we assume that
\begin{equation}\label{parallel}
  \Omega := \left\{x \in \R^d : a_i \leq v_i^T x \leq b_i \text{ for }
    i = 1, \dots, F \right\}
\end{equation}
for $F \geq 1$, unit vectors $v_1, \dots, v_F$ and
real numbers $a_1, \dots, a_F, b_1, \dots, b_F$. The number $F$ is
assumed to be bounded from above by a constant depending on $d$ 
alone. As in the rest of the paper, we also assume \eqref{OmegaAss}.

The design points $X_1,
\dots, X_n$ are assumed to form a fixed regular rectangular grid in
$\Omega$ and 
$Y_1, \dots, Y_n$ are generated according to
\eqref{eq:RegMdl}. Specifically, for $\delta > 0$, let 
\begin{equation}\label{frg}
  \grid := \left\{(k_1 \delta, \dots, k_d \delta) : k_i \in \mathbb{Z}, 1
    \leq i \leq d \right\}
\end{equation}
denote the regular $d$-dimensional $\delta$-grid in $\R^d$. We assume
that $X_1, \dots, X_n$ are an enumeration of the points in 
$\grid \cap \Omega$ with $n$ denoting the cardinality
of $\grid \cap \Omega$. By the usual volumetric argument and 
assumption \eqref{OmegaAss} , there exists 
a small enough constant $\kappa_d > 0$ such that whenever $0 < \delta \leq
\kappa_d$, we have
\begin{equation}\label{nlob}
  2 \leq c_d \delta^{-d} \leq n \leq C_d \delta^{-d}
\end{equation}
for dimensional constants $c_d$ and $C_d$. We have included a proof of
the above claim in Lemma \ref{points}. Throughout, we assume 
 $\delta \leq \kappa_d$ so that the above inequality holds. The
 following result (proved in Appendix \ref{minimax.proofs}) gives an
 upper bound on the minimax risk 
 \eqref{minratedeffix} over the 
 class $\F^{\lin}(\Omega)$ defined in \eqref{FlinClass}. 
\begin{proposition}\label{minratelin}
  For every $\Omega$ of the form \eqref{parallel} and satisfying
  \eqref{OmegaAss}, we have
  \begin{equation}
    \label{minratelin.eq}
       \minmax_n^{\mathrm{fixed}}(\F^{\lin}(\Omega)) \leq C_{\lin, \sigma} 
    n^{-4/(d+4)} (c_d \log n)^{4F/(d+4)}. 
  \end{equation}
\end{proposition}
The class $\F^{\lin}(\Omega)$ is larger than both $\C^{\lin}(\Omega)$
and $\C_{\lin}(\Omega)$ which means that the bound in
\eqref{minratelin.eq} also holds for $\minmax_n^{\mathrm{fixed}}(\F)$
for $\F = \C^{\lin}(\Omega)$ or $\C_{\lin}(\Omega)$ or
$\C^{\lin}_{\lin}(\Omega)$. To see why $\F^{\lin}(\Omega)$ is larger
than $\C^{\lin}(\Omega)$, just note that (by taking $g \equiv 0$)
\begin{equation*}
  \inf_{g \in   \A(\Omega)}\ell_{\P_n}(f_0, g)
\leq \sup_{x \in \Omega} |f_0(x)|. 
\end{equation*}
To see why $\F^{\lin}(\Omega)$ is also larger than
$\C_{\lin}(\Omega)$, note that (taking $g(x) := f_0(0)$ for the origin
$0$; the origin belongs to $\Omega$ because of \eqref{OmegaAss})  
\begin{align*}
  \inf_{g \in   \A(\Omega)}\ell^2_{\P_n}(f_0, g) &\leq \frac{1}{n}
  \sum_{i=1}^n \left(f_0(X_i) - f_0(0) \right)^2 \\ &\leq \left[\sup_{x
      \neq y} \frac{|f_0(x) - f_0(y)|}{\|x - y\|} \right]^2
                                                        \frac{1}{n}
                                                      \sum_{i=1}^n
                                                      \|X_i\|^2
   \leq \left[\sup_{x
      \neq y} \frac{|f_0(x) - f_0(y)|}{\|x - y\|} \right]^2
\end{align*}
where the last inequality is true because each $\|X_i\| \leq 1$ as
$\Omega \subseteq \balld$. 

The following minimax lower bound (proved in Appendix
\ref{minimax.proofs}) complements the upper bound in Proposition
\ref{minratelin} and it is applicable to the smaller function class
$\C_L^L(\Omega)$. $\Omega$ need not be polytopal for this result.     
\begin{proposition}\label{fixlobo}
  For every $\Omega$ satisfying \eqref{OmegaAss}, we have
  \begin{equation}
    \label{fixlobo.eq}
       \minmax_n^{\mathrm{fixed}}(\C^{L}_L(\Omega))  \geq c_{d,
         \sigma} n^{-4/(d+4)}. 
     \end{equation}
     provided $L \geq C_d$. 
\end{proposition}
Combining the above two results, in
fixed-design for $\Omega$ of the form \eqref{parallel} and satisfying
  \eqref{OmegaAss}, the minimax rate for each class
$\C_L^L(\Omega),  \C_L(\Omega), \C^L(\Omega), \F^{L}(\Omega)$ equals
$n^{-4/(d+4)}$ up to logarithmic multiplicative factors (this is
because $\C_L^L(\Omega)$ is the smallest of these classes for which
the lower bound \eqref{fixlobo.eq} holds and $\F^L(\Omega)$ is the
largest of these classes for which the upper bound
\eqref{minratelin.eq} holds).

\section{LSE Rates for Convex Regression}\label{LseRates}

\subsection{Random Design}\label{sec.mores}
In this section, we state our main result that the supremum risk appearing in the first three rows of
Table \ref{LSEexisting} is bounded from below by $n^{-2/d}$
(up to logarithmic  factors) for $d \geq 5$. As
$n^{-2/d}$ is strictly larger compared to the minimax rate of
$n^{-4/(d+4)}$ for $d \geq 5$, these results prove minimax
suboptimality, for $d \geq 5$ of the LSE $\hat{f}_n(\F)$ over the
class $\F$ for $\F$ equaling $\C_L^B(\Omega)$ or $\C_L(\Omega)$ for any
$\Omega$, and also for $\F$ equaling
$\C^B(\Omega)$ for polytopal $\Omega$. 

\begin{theorem}\label{implb.blse}
  Fix $d \geq 5$ and let $\Omega$ be a convex body satisfying
  \eqref{OmegaAss}. The following inequality holds when $\F$ is either
    $\C_L^B(\Omega)$ for $B \geq L$, or $\C_L(\Omega)$: 
    \begin{equation}\label{implb.clse.eq}
      \sup_{f \in \F} \E_{f} \ell^2_{\P}(\hat{f}_n(\F), f) \geq
      c_d \sigma L n^{-2/d} (\log n)^{-4(d+1)/d}
    \end{equation}
    provided $n \geq N_{d, \sigma/L}$. Additionally when $\Omega$ is a
    polytope whose number of facets is bounded by a constant depending
    on $d$ alone, the same inequality holds for $\F = \C^B(\Omega)$
    i.e.,
    \begin{equation}\label{implb.blse.eq}
      \sup_{f \in \C^B(\Omega)} \E_{f} \ell^2_{\P}(\hat{f}_n(\C^B(\Omega)), f) \geq
      c_d \sigma B n^{-2/d} (\log n)^{-4(d+1)/d}
    \end{equation}
    for $n \geq N_{d, \sigma/B}$. 
\end{theorem}

We prove the above theorem via $\sup_{f \in \F} \E_{f}
\ell^2_{\P}(\hat{f}_n(\F), f) \geq 
\E_{\tilde{f}} \ell^2_{\P}(\hat{f}_n(\F), \tilde{f})$ for a specific
piecewise affine convex function $\tilde{f}$ with 
$\sqrt{n}$ affine pieces. This function $\tilde{f}$ will be a
piecewise affine approximation to the quadratic $f_0(x) := \|x\|^2$. 
The properties of $\tilde f$ that we need
along with the existence of $\tilde f$ satisfying those 
properties are given in the next result. While this result is
stated for arbitrary $k$, we only need it for $k \sim \sqrt{n}$ for
proving Theorem \ref{implb.blse}. Recall that a $d$-simplex $\Delta$
in $\R^d$ is the convex hull of $d+1$ affinely independent points. It
is well-known that $d$-simplices can be represented as the
intersection of $d+1$ halfspaces. 

\begin{lemma}\label{ballapp}
  Suppose $\Omega$ is a convex body satisfying \eqref{OmegaAss}. Let
  $f_0(x) := \|x\|^2$. There exists a positive 
  constant $C_d$ (depending on $d$ alone) such that the
  following is true. For every $k \geq 1$, there exist $m \leq C_d k$ $d$-simplices
  $\Delta_1, \dots, \Delta_m \subseteq \Omega$ and a convex function
  $\tilde{f}_k$ on $\Omega$ such that   
  \begin{enumerate}
  \item $(1 - C_d k^{-1/d}) \Omega \subseteq \cup_{i=1}^m \Delta_i
    \subseteq \Omega$,
  \item $\Delta_i \cap \Delta_j$ is contained in a facet of $\Delta_i$
    and a facet of $\Delta_j$ for each $i \neq j$, 
  \item $\tilde{f}_k$ is affine on each $\Delta_i, i = 1, \dots, m$,
  \item $\sup_{x \in \Omega} |f_0(x) - \tilde{f}_k(x) | \leq C_d
    k^{-2/d}$,
  \item $\tilde{f}_k \in \C_{C_d}^{C_d}(\Omega)$,
  \end{enumerate}
If, in addition, $\Omega$ is a polytope whose number of facets is
bounded by a constant depending on $d$ alone, then the first condition
above can be strengthened to $\Omega = \cup_{i=1}^m \Delta_i$. 
\end{lemma}

\subsection{Results for the unrestricted convex LSE in fixed
  design} \label{fdresults} 
Here we work with the same setting as in Subsection \ref{minfd} (in particular, note that $\Omega$ is polytopal of the form \eqref{parallel}). The
following result shows that the unrestricted convex LSE
$\hat{f}_n(\C(\Omega))$ achieves the rate $r_{n, d}$ (defined in 
\eqref{balazs}) under the assumption $f_0 \in \F^{\lin}(\Omega)$ for a
positive constant $\lin$. The number $F$ appearing in the bound
\eqref{swc} below is the number of parallel halfspaces or slabs
defining $\Omega$ (see \eqref{parallel}) and this number is assumed to
be bounded by a constant depending on $d$ alone.   

\begin{theorem}\label{wcad}
For every $\lin > 0$ and $\sigma > 0$, there exist
constants $C_d$ (depending on $d$ alone) and $N_{d, \sigma/\lin}$
(depending only on $d$ and $\sigma/\lin$) such that
\begin{equation}\label{swc}
\sup_{f \in \F^{\lin}(\Omega)} \E_{f} \ell_{\empdes}^2(\hat f_n(\C(\Omega)), f) \leq
  \begin{cases}
   C_d \lin^{\frac{2d}{4 + d}} \left(\frac{\sigma^2 }{n} (\log n)^{F}
     \right)^{\frac{4}{d+4}}  \text{for } d  \le 3 \\
   C_4 \frac{\sigma \lin }{\sqrt{n}} (\log n)^{1 + \frac{F}{2}}   ~~~ \text{for } d = 4 \\
   C_d \sigma \lin \left(\frac{(\log n)^{F}}{n} \right)^{\frac{2}{d}} ~
 \text{for } d \geq 5
  \end{cases}
\end{equation}
for $n \geq N_{d, \sigma/\lin}$.
\end{theorem}
The rates appearing on the right hand side of \eqref{swc} coincide
with $r_{n, d}$ if we ignore multiplicative factors that are at most
logarithmic in $n$. Theorem 
\ref{wcad} and the minimax lower bound given in
Proposition \ref{fixlobo} together imply that the unrestricted convex LSE
is minimax optimal (up to log factors) over each class
$\F^{L}(\Omega), \C_{L}^{L}(\Omega), \C_L(\Omega), \C^{L}(\Omega)$
when the dimension $d \leq 4$. However for $d \geq 5$, there is a gap
between the rate given by \eqref{swc} and the minimax upper bound in
Proposition \ref{minratelin}. The following result shows that, for $d
\geq 5$, the unrestricted
convex LSE is indeed minimax suboptimal over $\C_L^L(\Omega)$ (or over
the larger classes $\C_L(\Omega)$, $\C^L(\Omega)$, $\F^L(\Omega)$). 

\begin{theorem}\label{minsub}
Fix $d \geq 5$, $L > 0$ and $\sigma >
  0$. There exist constants
  $c_d$ and $C_{d, \sigma/L}$ such that
  \begin{equation}\label{minsub.eq}
  \sup_{f \in \C^L_L(\Omega)} \E_{f} \ell_{\empdes}^2(\hat{f}_n(\C(\Omega)), f)
  \geq c_d \sigma L  n^{-\frac{2}{d}} (\log
  n)^{-\frac{4(d+1)}{d}} 
\end{equation}
for $n \geq C_{d, \sigma/L}$.  
\end{theorem}
The main idea for the proof of Theorem \ref{minsub} comes from
analyzing the rates of convergence of the unrestricted convex LSE when
the true convex function is piecewise affine. Indeed, Theorem
\ref{minsub}  is derived from Theorem \ref{implb} (stated later in
this section) which provides a lower bound on the risk of
$\hat{f}_n(\C(\Omega))$ for certain piecewise affine convex
functions. We state Theorem \ref{implb} after the next result which
provides upper bounds on the rate of convergence of
$\hat{f}_n(\C(\Omega))$ for piecewise affine convex functions. This
result shows that the unrestricted convex LSE 
exhibits adaptive behaviour to piecewise
affine convex functions $f$ in the sense that the risk $\E_{f}
\ell_{\P_n}^2(\hat{f}_n(\C(\Omega)), f)$ is 
much smaller than the right hand side of \eqref{swc} for such $f$.
For $k \geq 1$ and $h \geq 1$, let $\lincon_{k, h}(\Omega)$ denote all
functions $f \in \C(\Omega)$ for which there exist $k$ convex
subsets $\Omega_1, \dots, \Omega_k$ satisfying the following
properties: 
\begin{enumerate}
\item $f$ is affine on each $\Omega_i$,
\item each $\Omega_i$ can be written as an intersection of at most $h$ slabs (i.e., as in \eqref{parallel} with $F = h$),
  and  
\item $\Omega_1 \cap \grid, \dots, \Omega_k \cap \grid$ are disjoint
  with $\cup_{i=1}^k (\Omega_i \cap \grid) = \Omega \cap \grid$
  (recall that $\grid$ is the regular rectangular grid \eqref{frg}). 
\end{enumerate}

\begin{theorem}\label{adrate}
For every $k \geq 1$ and $h \geq 1$, we have
\begin{equation}\label{adrate.eq}
\sup_{f \in \lincon_{k, h}(\Omega)} \E_{f} \ell_{\empdes}^2(\hat f_n(\C(\Omega)), f) \leq
  \begin{cases}
   C_d \sigma^2 \left(\frac{k}{n} \right) (\log n)^{h} &
   \text{for } d = 1, 2, 3 \\
   C_d \sigma^2\left(\frac{k}{n}
     \right) (\log n)^{h + 2} & \text{for } d = 4 \\
   C_d \sigma^2 \left(\frac{k (\log n)^{h}}{n}
     \right)^{4/d} &
 \text{for } d \geq 5
  \end{cases}
\end{equation}
for a constant $C_d$ depending on $d$ alone.
\end{theorem}
When $h$ is a constant (not depending on $n$) and $k$ is not too
large, the rates in
\eqref{adrate.eq} are of strictly smaller order compared to
\eqref{swc}. More precisely, ignoring log factors, \eqref{adrate.eq}
is strictly smaller than \eqref{swc} as long as $k$ is
of smaller order than $n^{d/(d+4)}$ for $d \leq 4$, and as long as $k$
is of smaller order than $\sqrt{n}$ for $d \geq 5$. The next result
gives a lower bound which proves that the upper bound in
\eqref{adrate.eq} cannot be improved (up to log factors) for $d \geq
5$ for all $k$ up to $\sqrt{n}$. More specifically, we prove the
$(k/n)^{4/d}$ lower bound for the piecewise affine convex function
$\tilde{f}_k$ described in Lemma \ref{ballapp}. 

\begin{theorem}\label{implb}
Fix $d \geq 5$. There exist positive constants $c_d$ and $N_d$ such
that for  $n \geq N_d$ and
\begin{equation}\label{implb.con}
  1 \leq k \leq \min \left(\sqrt{n} \sigma^{-d/4}, c_d n \right),
\end{equation}
we have
\begin{equation}\label{implb.eq}
    \E_{\tilde{f}_k} \ell_{\empdes}^2(\hat{f}_n(\C(\Omega)),
    \tilde{f}_k)\geq c_d \sigma^2 \left(\frac{k}{n} 
   \right)^{4/d} (\log n)^{-4(d+1)/d}
 \end{equation}
 where $\tilde{f}_k$ is the function from Lemma \ref{ballapp}.
\end{theorem}
The lower bound given by \eqref{implb.eq} for $k = \sqrt{n}
\sigma^{-d/4}$ is of the same order as that given by Theorem
\ref{minsub}. In other words, Theorem \ref{minsub} is a corollary of
Theorem \ref{implb}.  

We reiterate that the results of this subsection (fixed-design risk
bounds for the unrestricted convex LSE) hold when $\Omega$ is 
a polytope. A natural question is to extend these to the case where
$\Omega$ is a smooth convex body such as the unit ball. Based on the
results of Kur et al. \cite{kur2019optimality} who analyzed the LSE over
$\C^B(\Omega)$ under random-design, it is reasonable to conjecture
that the unrestricted convex LSE will be minimax optimal in fixed
design when the domain is the unit ball. However it appears nontrivial
to prove this as the main ideas from \cite{kur2019optimality} such as
the reduction to expected suprema bounds over level sets cannot be
used in the absence of uniform boundedness.

\section{Proof Sketches for Results in Section
  \ref{LseRates}}\label{pfskeran}
In this section, we provide the main ideas and ingredients behind the
proofs of the LSE rate results in Section \ref{LseRates}. Further
details and full proofs of the auxiliary results in this section can
be found in Appendix \ref{rd.proofs}.

\subsection{General LSE Accuracy result}\label{genchasec}

The first step to prove all our LSE results is the following theorem
due to Chatterjee \cite{Chat14} which provides a course of action for
bounding (from 
above as well as below) the accuracy of abstract LSEs on convex classes of
functions. The original result applies to the fixed design case for
every set of design points $X_1, \dots, X_n$. In the random design
case, we apply this result by conditioning on $X_1, \dots, X_n$ (see
\eqref{charan} and \eqref{tfran} below). 

\begin{theorem}[Chatterjee]\label{chat}
  Consider data generated according to the model:
  \begin{equation*}
    Y_i = f(X_i) + \xi_i \qt{for $i = 1, \dots, n$}
  \end{equation*}
  where $X_1, \dots, X_n$ are fixed deterministic design points in a
  convex body $\samp \subseteq \R^d$, $f$ belongs to a convex class of
  functions $\F$ and $\xi_1, \dots, \xi_n \overset{\text{i.i.d}}{\sim}
  N(0, \sigma^2)$. Consider the LSE $\hat{f}_n(\F)$ defined in
  \eqref{genlse}. Define
  \begin{equation*}
t_f(\F) := \argmax_{t \geq 0} H_f(t, \F)    
  \end{equation*}
where
    \begin{equation*}
      H_f(t, \F) := \E \sup_{g \in \F : \ell_{\P_n}(f, g) \leq t}
      \frac{1}{n} \sum_{i=1}^n \xi_i   \left(g(X_i) - f(X_i) \right) -
      \frac{t^2}{2}.  
    \end{equation*}
 Then
$H_f(\cdot, \F)$ is a concave function on $[0, \infty)$, $t_f(\F)$ is
unique and the following pair of inequalities hold for positive constants $c$
and $C$:
    \begin{equation}\label{problb}
      \P \left\{0.5 t^2_f(\F) \leq \ell_{\empdes}^2(\hat{f}_n(\F), f)
        \leq 2 t^2_f(\F)  \right\} \geq 1 - 6 \exp
      \left(-\frac{cnt^2_f(\F)}{\sigma^2} \right) 
    \end{equation}
   and
    \begin{equation}\label{expecb}
    0.5 t^2_f(\F) - \frac{C \sigma^2}{n} \leq   \E \ell_{\empdes}^2(\hat{f}_n(\F), f)
    \leq 2 t^2_f(\F) + \frac{C \sigma^2}{n}.
  \end{equation}
  Upper bounds for $t_f(\F)$ can be obtained via:
  \begin{equation}\label{tfub}
    t_f(\F) \leq \inf \left\{t > 0 : H_f(t, \F) \leq 0 \right\}
  \end{equation}
  and lower bounds for $t_f(\F)$ can be obtained via:
  \begin{equation}\label{tflb}
    t_f(\F) \geq t_1 \qt{if $0 \leq t_1 < t_0$ are such that $H_f(t_1,
      \F) \le
      H_f(t_0, \F)$}.
  \end{equation}
\end{theorem}
Intuitively, Theorem \ref{chat} states that the loss
$\ell_{\P_n}(\hat{f}_n(\F), f)$ of the LSE is controlled by $t_f(\F)$
for which bounds can be obtained using \eqref{tfub} and \eqref{tflb}.

Theorem \ref{chat} holds for the
fixed design setting with no restriction on the design
points which means that it also applies to the random
design setting provided we condition on the design points
$X_1, \dots, X_n$. In particular, for random design, inequality
\eqref{problb} becomes: 
\begin{equation}\label{charan}
  \begin{split}
  &\P \left\{0.5 t^2_f(\F) \leq \ell_{\P_n}^2(\hat{f}_n(\F), f) \leq 2
    t^2_f(\F) \bigg| X_1, 
    \dots, X_n \right\} \\ &\geq 1 - 6 \exp \left(- \frac{cn
    t^2_f(\F)}{\sigma^2} \right)
    \end{split}
\end{equation}
where
\begin{equation}
  \label{tfran}
t_f(\F) := \argmax_{t \geq 0} H_f(t, \F)
\end{equation}
with
\begin{equation}\label{hfdef}
  H_f(t, \F) := \E \left[ \sup_{g \in \F: \ell_{\P_n}(f, g) \leq t} \frac{1}{n}
  \sum_{i=1}^n \xi_i \left(g(X_i) - f(X_i) \right) \bigg| X_1, \dots,
  X_n \right] - \frac{t^2}{2}.
\end{equation}
Here $t_f(\F)$ is random as it depends on the random design points
$X_1, \dots, X_n$. Note that \eqref{charan} applies to the loss
$\ell_{\P_n}$ and not to $\ell_{\P}$.

\subsection{High-level intuition for the proof of Theorem
  \ref{implb.blse}}\label{intuit}
Here we provide the main ideas behind the proof of Theorem
\ref{implb.blse} in a somewhat non-rigorous fashion. More details are  
provided in the remainder of this section. According to Theorem
\ref{chat}, the rate of convergence of the LSE over $\F$ (when the true
function is $f$) is controlled by $t_f^2(\F)$ where $t_f(\F)$ is the
maximizer of $H_f(t, \F)$ defined in \eqref{hfdef}. Note that $H_f(t,
\F)$ is given by an expected supremum term minus $t^2/2$. The expected
supremum term can be seen as a measure of the complexity of the local
region $\{g \in \F : \ell_{\P_n}(f, g) \leq t\}$ around the true
function $f$. 

For the proof of Theorem \ref{implb.blse}, we shall take the true
function to be $\tilde{f}_k$ (given by Lemma \ref{ballapp}) with $k
\sim \sqrt{n}$. The challenge is to show that the maximizer of
$H_{\tilde{f}_k}(t, \F)$ is at least $n^{-1/d} (\log
n)^{-2(d+1)/d}$ for each function class $\F$ in the statement of
Theorem \ref{implb.blse}. For this, the first step is to prove that 
\begin{align}\label{intuitive_1}
   H_{\tilde{f}_k}(t, \F) \leq C_d t n^{-1/d} (\log n)^{2(d+1)/d}. 
\end{align}
for a large enough range of $t$ (see Lemma \ref{GrUB} and the writing after it). This inequality, which grows linearly in $t$, can be seen as a bound
on the complexity of the local region in $\F$ around
$\tilde{f}_k$, and is proved by bounding the bracketing entropy
numbers of local regions around $\tilde{f}_k$ (Theorem
\ref{Thm:Bddcube}). 

The second step is to prove the following lower bound on
$H_{\tilde{f}_k}(t, \F)$:
\begin{equation}\label{intuitive_2}
 \sup_{t}  H_{\tilde{f}_k}(t, \F) \geq c_d n^{-2/d}
\end{equation}
This is proved by lower bounding the metric entropy of a local region
around $f_0(x) := \|x\|^2$ (Lemma \ref{lbf0.ran})  and the fact that
$\tilde{f}_k$ and $f_0$ are within $C_d k^{-2/d} = n^{-1/d}$ of each
other (this is guaranteed by Lemma \ref{ballapp}). 

Combining \eqref{intuitive_1} and \eqref{intuitive_2}, it is
straightforward to see that
\begin{align*}
   H_{\tilde{f}_k}(t, \F) \leq C_d t n^{-1/d} (\log n)^{2(d+1)/d} <
 c_d n^{-2/d} \leq  \sup_{t}  H_{\tilde{f}_k}(t, \F) 
\end{align*}
for
\begin{align*}
  t < (c_d/C_d) n^{-1/d} (\log n)^{-2(d+1)/d}, 
\end{align*}
so that the maximizer of $H_{\tilde{f}_k}(t, \F)$ is $\geq
(c_d/C_d) n^{-1/d} (\log n)^{-2(d+1)/d}$.    

It might be interesting to note that this proof will not work if we
take the true function to be the smooth function $f_0(x) := \|x\|^2$ instead
of the piecewise affine approximation $\tilde{f}_k$. For $f_0$, the
upper bound \eqref{intuitive_1} is no longer true; in fact $H_{f_0}(t,
\F)$ will be as large as $c_d n^{-2/d}$ for $t$ of the order
$n^{-2/d}$ which clearly violates the upper bound
\eqref{intuitive_1}. Intuitively, this happens because local
regions around $f_0$ have more complexity compared to local regions
around $\tilde{f}_k$. This is because the piecewise affine convex
function $\tilde{f}_k$ has flat Hessians so it is not easy to find
many perturbations which keep the resulting function convex. But this
is much easier to do with the smooth function $f_0$.  We are unable to determine the rate of convergence of the LSEs when
the true function is $f_0$. It can be anywhere between $n^{-4/d}$ and
$n^{-2/d}$, a range which includes the minimax rate $n^{-4/(d+4)}$.

\subsection{Proof Sketch for Theorem \ref{implb.blse}}\label{pfskeblse}
It is enough to prove \eqref{implb.clse.eq} and \eqref{implb.blse.eq}
when $L$ and $B$ are fixed constants depending on the dimension $d$ alone, respectively. From
here, these inequalities for arbitrary $L$ and $B$ can be deduced by
elementary scaling arguments.  Let $f_0(x) :=
\|x\|^2$ and let $\tilde{f}_k$ be as given by Lemma \ref{ballapp} for a
fixed $1 \leq k \leq n$. Below 
we assume that $L$ is a large enough constant depending on $d$ alone so
that $\tilde{f}_{k} \in \C^L_L(\Omega)$ and also that $B \geq L$.  We
assume that the true function is $\tilde{f}_k$ and show that, when $k
\sim \sqrt{n}$, the risk of $\hat{f}_n(\F)$ at $\tilde{f}_k$ is
bounded from below by $c_d \sigma n^{-2/d} (\log n)^{-4(d+1)/d}$ for
each of the choices of the function class $\F$ in the statement of
Theorem \ref{implb.blse}.  

Our strategy is to bound 
$t_{\tilde{f}_k}(\F)$ from below and then use Theorem \ref{chat} to
obtain the LSE risk lower bound. Recall, from Theorem \ref{chat}, that
$t_{\tilde{f}_k}(\F)$ maximizes
 \begin{equation}\label{Hkdef}
 H_{\tilde{f}_k}(t, \F) :=  G_{\tilde{f}_k}(t, \F) - \frac{t^2}{2}
 \end{equation}
 over all $t \geq 0$ where
 \begin{equation}\label{gfkt}
   G_{\tilde{f}_k}(t, \F) := \E \left[ \sup_{g \in \F :
       \ell_{\P_n}(g, \tilde{f}_k) \leq t} \frac{1}{n}
  \sum_{i=1}^n \xi_i \left(g(X_i) - \tilde{f}_k(X_i) \right) \bigg| X_1, \dots,
  X_n \right].
\end{equation}
The main ingredients necessary to bound $t_{\tilde{f}_k}(\F)$ are the
next pair of results, which provide upper and lower 
bounds for $G_{\tilde{f}_k}(t, \F)$ respectively.
\begin{lemma}\label{GrUB}
  Fix $d \geq 5$ and let $\Omega$ be a convex body satisfying
  \eqref{OmegaAss}. There exists $C_d > 0$ such that the following
  holds 
    \begin{equation}\label{GrUB.eq1}
    \begin{split}
    & \P \left\{G_{\tilde{f}_k}(t, \F) > C_d t \sigma
      \left(\frac{k}{n} \right)^{2/d} (\log  (C_d \sigma
      \sqrt{n}))^{2(d+1)/d} + \frac{C_d}{\sqrt{n}} +  \frac{C_d \sigma
      k^{2/d^2}}{k^{1/d}n^{2/d}} \right\} \\ &\leq  C_d
                                                                 \exp
                                                                 \left(-\frac{n^{d/(d+4)}}{C_d} 
                                                                 \right)
                                               + C_d \exp 
      \left(-\frac{n^{(d-4)/d}}{C_d^2} t^2 k^{4/d} \right) + \exp
                                       \left(-\frac{C_dn}{k^{2/d}} \right)   
    \end{split}
  \end{equation}
  for every fixed $t$ satisfying $C_d n^{-2/(d+4)} \leq t \leq
  L$, and for $\F$ equal to either of the two classes $\C_L^L(\Omega)$
  and $\C_L(\Omega)$.

  Additionally, when $\Omega$ is a polytope whose number of facets is
  bounded by a constant depending on $d$ alone, we have the following
  inequality for $\F = \C^B(\Omega)$: 
  \begin{equation}\label{GrUB.eq2}
    \begin{split}
    & \P \left\{G_{\tilde{f}_k}(t, \C^B(\Omega)) > C_d t \sigma
      \left(\frac{k}{n} \right)^{2/d} (\log (C_d \sigma
      \sqrt{n}))^{2(d+1)/d} + \frac{C_d}{\sqrt{n}} \right\} \\ &\leq 
                                                                 C_d
                                                                 \exp
                                                                 \left(-\frac{n^{d/(d+4)}}{C_d} 
                                                                 \right)
                                                                 + C_d \exp
      \left(-\frac{n^{(d-4)/d}}{C_d^2} t^2 k^{4/d} \right) 
    \end{split}
  \end{equation}
  for every fixed $t$ satisfying $t \geq C_d
  n^{-2/(d+4)}$. 
\end{lemma}

\begin{lemma}\label{GrLB}
  Fix $d \geq 1$ and let $\Omega$ be a convex body satisfying
  \eqref{OmegaAss}. There exists $c_d, C_d > 0$ such that the
  following holds for every fixed  
  $1 \leq k \leq n$ and $t \geq C_d k^{-2/d}$: 
  \begin{equation*}
    \P \left\{G_{\tilde{f}_k}(t, \F) \geq c_d \sigma n^{-2/d}
  \right\} \geq 1 - \exp \left(-c_d n  \right)
\end{equation*}
for $\F$ equal to $\C^B_L(\Omega)$, and also for the larger classes
$\C_L(\Omega)$ and $\C^B(\Omega)$.   
\end{lemma}

Before providing details behind the proofs of Lemma \ref{GrUB} and
Lemma \ref{GrLB}, let us outline how they lead to the proof of Theorem
\ref{implb.blse} (full details are in Appendices
\ref{tfloboran} and \ref{comple.implb.blse}). The leading
term in the upper bound on  $G_{\tilde{f}_k}(t, \F)$ given by Lemma
\ref{GrUB} is
\begin{equation}\label{ublea}
  C_d t \sigma \left(\frac{k}{n} \right)^{2/d} (\log (C_d \sigma
  \sqrt{n}))^{2(d+1)/d}.  
\end{equation}
This upper bound also applies to $H_{\tilde{f}_k}(t, \F)$ because
$H_{\tilde{f}_k}(t, \F) \leq
G_{\tilde{f}_k}(t, \F)$ as is clear from the definition
\eqref{Hkdef}. This bound was previously stated as \eqref{intuitive_1}
in Subsection \ref{intuit}.  

On the other hand, the lower bound on $G_{\tilde{f}_k}(t, \F)$ from
Lemma  \ref{GrLB} leads to the following lower bound for
$H_{\tilde{f}_k}(t, \F)$:
\begin{equation*}
  H_{\tilde{f}_k}(t, \F) \geq c_d \sigma n^{-2/d} - \frac{t^2}{2}. 
\end{equation*}
The special choice $t_0 := \sqrt{c_d} n^{-1/d}
\sqrt{\sigma}$ in the above bound leads to
\begin{equation}\label{hlo}
    H_{\tilde{f}_k}(t_0, \F) \geq \frac{c_d}{2} \sigma n^{-2/d}. 
\end{equation}
This lower bound is the basis for the inequality \eqref{intuitive_2}
in the intuition subsection \ref{intuit}. The requirement $t_0 =
\sqrt{c_d} n^{-1/d} \sqrt{\sigma} \geq C_d 
k^{-2/d}$ in Lemma \ref{GrLB} would hold if $k$ satisfies $k \geq
\gamma_d \sqrt{n} \sigma^{-d/4}$ for some $\gamma_d$.

We now compare the upper bound \eqref{ublea} with the lower bound
\eqref{hlo}. It can be seen that \eqref{ublea} is less than or equal
to the right hand side of  \eqref{hlo} when $t$ is of the order $k^{-2/d}
(\log (C_d \sigma \sqrt{n}))^{-2(d+1)/d}$. Let us denote by $t_1$ this
particular choice of $t$ for $k = \gamma_d \sqrt{n} \sigma^{-d/4}$. We
will argue, in Lemma \ref{tfktildelb} below, that $t_1 < t_0$ and that 
\begin{equation*}
  H_{\tilde{f}_k}(t_1, \F) \leq H_{\tilde{f}_k}(t_0, \F). 
\end{equation*}
This allows application of inequality \eqref{tflb} to yield the
following lower bound on $t_{\tilde{f}_k}(\F)$ (this lower bound gives
the necessary risk lower bounds for $\hat{f}_n(\F)$ via Theorem 
\ref{chat}). 
\begin{lemma}\label{tfktildelb}
  There exist constants $\gamma_d, c_d, C_d, N_{d, \sigma}$ such that
  \begin{equation}\label{fbtf}
  \P \left\{t_{\tilde{f}_k}(\F) \geq c_d n^{-1/d} \sqrt{\sigma} (\log
    n)^{-2(d+1)/d}  \right\} \geq 1 - C_d \exp
  \left(\frac{-n^{(d-4)/d}}{C_d^2} \right)
\end{equation}
for $k = \gamma_d \sqrt{n} \sigma^{-d/4}$ and $n \geq N_{d,
  \sigma}$. Here $\F$ can be taken to be any of the three choices
in Theorem \ref{implb.blse}.  
\end{lemma}

\subsubsection{Proof ideas for Lemma \ref{GrUB} and Lemma
  \ref{GrLB}} \label{hightech} 

We now explain the main proof ideas behind Lemma \ref{GrUB}
and Lemma \ref{GrLB}. For Lemma \ref{GrUB}, we use available bounds on
suprema of empirical processes via bracketing numbers. The key
bracketing result that is needed for Lemma \ref{GrUB} is given
below. It provides an upper bound on the bracketing entropy of bounded
convex functions with an additional $L_p$ norm constraint. The metric
employed is the $L_p$ metric on a bunch of simplices. This result is
stated for arbitrary $p \in [1, \infty)$  although we only use it for
$p = 2$. Lemma \ref{GrUB} is proved by applying this theorem to
the simplices given in Lemma \ref{ballapp}. 

\begin{theorem}\label{Thm:Bddcube}
	Suppose $\Omega$ is a convex body contained in the unit
        ball. Let $\tilde{f}$ be a convex function on $\Omega$ that is
        bounded by $\Gamma$. For a fixed $1 \leq p < \infty$  and $t >
        0$, let 
        \begin{equation}\label{bpgam}
	B_p^\Gamma(\tilde{f}, t, \Omega)=\left\{f \in \C^{\Gamma}(\Omega):
          \int_{\Omega}|f(x)-\tilde{f}(x)|^pdx\le t^p \right\}.
      \end{equation}
      Suppose $\Delta_1, \dots, \Delta_k \subseteq \Omega$ are
      $d$-simplices with disjoint interiors such that $\tilde{f}$ is affine
      on each $\Delta_i$. Then for every $0 < \epsilon < \Gamma$ and $t > 0$,
      we have
       \begin{equation}\label{Thm:Bddcube.eq}
	\log N_{[\,]}(\eps, B_p^\Gamma(\tilde{f}, t, \Omega), \|\cdot\|_{p,
          \cup_{i=1}^k \Delta_i})\le C_{d,p}
        k\left(\log \frac{\Gamma}{\epsilon} \right)^{d+1}
        \left(\frac{t}{\epsilon} \right)^{d/2}
      \end{equation}
      for a constant $C_{d, p}$ that depends on $p$ and $d$ alone. The
      left hand side above denotes bracketing entropy with respect to
      $L_p$ metric on $\Delta_1 \cup \dots \cup
      \Delta_k$.
    \end{theorem}

The above bracketing entropy result (whose proof is given in Appendix \ref{Gao:subsection}) is nontrivial and novel. The
function class considered in the above theorem has both an
$L_{\infty}$ 
constraint (uniform boundedness) as well as an $L_p$
constraint. If the $L_p$  constraint is dropped, then the bracketing
entropy is of the order $(\Gamma/\epsilon)^{d/2}$ as proved by
Gao and Wellner \cite{gao2017entropy} (see also
Doss \cite{doss2020bracketing}). In 
contrast to $(\Gamma/\epsilon)^{d/2}$, \eqref{Thm:Bddcube.eq} only has
a logarithmic dependence on $\Gamma$ and is much smaller when $t$ is
small. Also, Theorem \ref{Thm:Bddcube} is comparable to but stronger
than \cite[Lemma 3.3]{HW16} which gives a weaker bound for the left
hand side of \eqref{Thm:Bddcube.eq} having additional multiplicative
factors involving $k$ (these factors cannot be neglected since we care
about the regime $k \sim \sqrt{n}$).

For Lemma \ref{GrLB}, the key is the following result which proves a
lower bound on the metric entropy of balls of convex functions around
the quadratic function $f_0(x) := \|x\|^2$ (recall, from Lemma
\ref{ballapp}, that $\tilde{f}_k$ are chosen to be piecewise affine
approximations of $f_0(x) = \|x\|^2$). 

\begin{lemma}\label{lbf0.ran}
 Let $\Omega$ be a convex body satisfying \eqref{OmegaAss}. Let
 $f_0(x) := \|x\|^2$. Then there 
 exist positive constants $c_1, c_2, c_3, c_4, C$ depending on
 $d$ alone such that
 \begin{equation*}
   \P \left\{\log N(\epsilon, \{f \in \C_L^L(\Omega): \ell_{\P_n}(f,
     f_0) \leq t \}, \metemp)
     \geq c_1 \epsilon^{-d/2} \right\} \geq 1 - \exp(-c_2 n)
 \end{equation*}
 for each fixed $\epsilon, t, L$ satisfying $L \geq C$ and $c_3
 n^{-2/d} \leq \epsilon  \leq \min(c_4, t/4)$. The probability above
 is with respect to the randomness in the design points $X_1, \dots,
 X_n \overset{\text{i.i.d}}{\sim} \P$. 
\end{lemma}
The main consequence of the above lemma is that when $\epsilon \sim
n^{-2/d}$ and $t \gtrsim n^{-2/d}$, the metric entropy of $\{f \in
\C_L^L(\Omega): \ell_{\P_n}(f, f_0) \leq t \}$ is of order $n$. This
also holds for $t \sim k^{-2/d}$ for $k \leq n$ as $k^{-2/d} \geq
n^{-2/d}$. Because the distance 
between $\tilde{f}_k$ and $f_0$ is bounded by $k^{-2/d}$, the same
order-$n$ lower bound also holds for the metric entropy of $\{f \in
\C_L^L(\Omega): \ell_{\P_n}(f, \tilde{f}_k) \leq t \}$. Sudakov
minoration can then be used to prove Lemma \ref{GrLB}. See Appendix \ref{rd.proofs} for full details.

\subsection{Proof sketches for fixed-design results (Subsection
  \ref{fdresults})}\label{pfskefdr}
Here, we provide sketches for the proofs of the fixed-design results
stated in Subsection \ref{fdresults}. Further details and full proofs
can be found in Appendix \ref{fd.proofs}. 

The starting point for these proofs is Theorem \ref{chat}. For 
the risk upper bound results (Theorem \ref{wcad} and Theorem
\ref{adrate}), we need to bound the quantity $H_{f_0}(t, \C(\Omega))$
(appearing in the statement of Theorem \ref{chat}) from above. Let
\begin{equation}\label{gtdef}
  G_{f_0}(t, \C(\Omega)) := \E \sup_{g
        \in \C(\Omega): \ell_{\empdes}(f_0, g) \leq t} \frac{1}{n} \sum_{i=1}^n \xi_i
      \left(g(X_i) - f_0(X_i) \right)
\end{equation}
so that $H_{f_0}(t, \C(\Omega)) = G_{f_0}(t, \C(\Omega)) - t^2/2$ and
upper bounds on $G_{f_0}(t, \C(\Omega))$ imply upper bounds on
$H_{f_0}(t, \C(\Omega))$. In contrast to the random design setting of
the previous subsection, we do not need to explicitly indicate
conditioning on $X_1, \dots, X_n$ in this fixed design setting.  

Theorem \ref{wcad} is a consequence of the following upper bound on
$G_{f_0}(t, \C(\Omega))$: 
\begin{lemma}\label{GtUBwcFD}
  Fix $f_0 \in \C(\Omega)$ and suppose $\inf_{g \in \A(\Omega)}
  \ell_{\P_n}(f_0, g) \leq \lin$. There exists $C_d$ such that for
  every $t > 0$, 
  \begin{equation}\label{GtUBwcFD.eq}
G_{f_0}(t, \C(\Omega)) \leq
  \begin{cases}
  \frac{C_d \sigma}{\sqrt{n}} (\log n)^{F/2} \left(t + \lin^{d/4} t^{1-d/4} \right) ~~\text{for } d  \le 3 \\
\frac{C_d \sigma}{\sqrt{n}} (t + \lin) (\log n)^{1 + (F/2)}  ~~~~~~~~~ \text{for } d = 4 \\
C_d \sigma \left(\frac{(\log n)^{F/2}}{\sqrt{n}} \right)^{4/d} (t +
    \lin) ~~~~~~ 
 \text{for } d \geq 5.
  \end{cases}
\end{equation}
\end{lemma}

Lemma \ref{GtUBwcFD} is a consequence of the following bound which
holds for functions $f_0
\in \C(\Omega)$ satisfying $\inf_{f \in \A(\Omega)}
\ell_{\P_n}(f_0, f) \leq \lin$: 
\begin{equation}\label{simmet}
 \log N(\epsilon, \left\{f \in
  \C(\Omega) : \ell_{\P_n}(f_0, f) \leq t \right\}, \ell_{\empdes}) \leq
 C_d (\log n)^F  \left(\frac{t + \lin}{\epsilon} \right)^{d/2}
\end{equation}
for every $t > 0$ and $\epsilon > 0$. This metric entropy bound, which is
novel and nontrivial, will be derived based on a more fundamental 
metric entropy result that is stated later in this section. Dudley's
entropy bound \cite{Dudley67} will be applied in
conjunction with \eqref{simmet} to 
derive Lemma \ref{GtUBwcFD}. Dudley's
bound involves an integral of the square root of the metric entropy
which will lead to an integral of $\epsilon^{-d/4}$. Since this 
integral converges for $d \leq 3$, and blows up at zero for $d = 4$ (logarithmically) and for $d
\geq 5$ (exponentially), the upper bound for $G_{f_0}(t, \C(\Omega))$ behaves
differently for the three regimes $d \leq 3$, $d = 4$ and $d \geq 5$.  

Theorem \ref{adrate} is a consequence of the following upper bound on $G_{f_0}(t, \C(\Omega))$:
\begin{lemma}\label{GtUBadFD}
  Fix $f_0 \in \lincon_{k, h}(\Omega)$. There exists $c_d$ such that
  for every $t > 0$,
  \begin{equation}\label{GtUBadFD.eq}
G_{f_0}(t, \C(\Omega)) \leq
  \begin{cases}
 t \sigma \sqrt{\frac{k}{n}} (c_d\log n)^{h/2} ~~~~~~~\text{for } d  \le 3 \\
t \sigma \sqrt{\frac{k}{n}} (c_d\log n)^{1 + (h/2)}   ~~~~~~~~~ \text{for } d = 4 \\
 t \sigma  \left((c_d\log n)^{h/2} \sqrt{\frac{k}{n}} \right)^{4/d}
    ~~~~~~  
 \text{for } d \geq 5. 
  \end{cases}
\end{equation}  
\end{lemma}

Lemma \ref{GtUBadFD} is a consequence of the following bound which
holds for functions $f_0
\in \mathfrak{C}_{k, h}(\Omega)$:
\begin{equation}\label{simmet1}
 \log N(\epsilon, \left\{f \in
  \C(\Omega) : \ell_{\P_n}(f_0, f) \leq t \right\}, \ell_{\empdes}) \leq
k \left(\frac{t}{\epsilon} \right)^{d/2} \left(c_d \log n \right)^h.
\end{equation}
The bound \eqref{simmet1}, which also seems novel and nontrivial, is
an improvement over \eqref{simmet} provided $t \lesssim L$ and $k$ is
not too large. Both the entropy bounds \eqref{simmet} and
\eqref{simmet1} will be derived based on a more general bound that is
stated later in this section.

We now move to the proof sketches for the lower bound results, Theorem
\ref{minsub} and Theorem \ref{implb}, which apply to the case $d \geq
5$. Because Theorem \ref{minsub} is a special case of Theorem \ref{implb}
corresponding to $k 
= \sqrt{n} \sigma^{-d/4}$ (see Appendix \ref{minsubproof}), we focus
on the proof  sketch of Theorem \ref{implb}. This proof is similar to
that of Theorem \ref{implb.blse} but simpler; for example, we can
simply work here with the single loss  $\ell_{\P_n}$ (and not worry
about its discrepancy with any continuous loss $\ell_{\P}$ as in the
proof of Theorem \ref{implb.blse}).  

As in the proof of Theorem
\ref{implb.blse}, we need to prove upper and lower bounds for
$G_{\tilde{f}_k}(t, \C(\Omega))$. The upper bound follows from Lemma
\ref{GtUBadFD}  because, as guaranteed by Lemma
\ref{ballapp}, the function $\tilde{f}_k$ belongs to $\mathfrak{C}_{m,
  d+1}(\Omega)$ for some $m \leq C_d k$ (we are using 
$h = d+1$ here because every $d$-simplex can be written as the
intersection of atmost $d+1$ slabs). This gives the following bound
which coincides with the leading term in random design bound given by
\eqref{GrUB.eq1}:
\begin{equation}\label{upbcl}
  G_{\tilde{f}_k}(t, \C(\Omega)) \leq C_d t \sigma (\log
  n)^{2(d+1)/d} \left(\frac{k}{n} \right)^{2/d} \qt{for every $t > 0$}
\end{equation}
The lower bound on $G_{\tilde{f}_k}(t, \C(\Omega))$ is given below.
\begin{lemma}\label{allkt.lem}
  There exist a positive constant $c_d$ such that
    \begin{equation}\label{allkt}
    G_{\tilde{f}_k}(t, \C(\Omega)) \geq c_d \sigma t \left(\frac{k}{n} \right)^{2/d} \qt{for
      all $0 < t \leq c_d k^{-2/d}$}
  \end{equation}
  provided $k \leq c_d n$. 
\end{lemma}
The lower bound \eqref{allkt} and the upper bound \eqref{upbcl}
differ only by the logarithmic factor $(\log n)^{2(d+1)/d}$. Further
\eqref{allkt} coincides with corresponding lower bound in Lemma \ref{GrLB}
for the random design setting when $k = \sqrt{n} \sigma^{-d/4}$ and
$t$ is of order $k^{-2/d}$. Lemma \ref{allkt.lem} is proved via the
following metric entropy lower bound which applies to the discrete
metric $\ell_{\P_n}$ and is analogous to Lemma \ref{lbf0.ran}.
  \begin{lemma}\label{lbf0}
    Let $\Omega$ be a convex body satisying \eqref{OmegaAss}. Let $f_0(x) :=
    \|x\|^2$. There exist two positive constants $c_1$ and $c_2$
    depending on $d$ alone such that
    \begin{equation*}
      \log N(c_1 n^{-2/d}, \{g \in \C(\Omega): \ell_{\P_n}(f_0, g)
      \leq t\},
      \ell_{\empdes}) \geq \frac{n}{8} \qt{for $t \geq c_2
        n^{-2/d}$}.
    \end{equation*}
\end{lemma}
Theorem \ref{implb} is proved by combining \eqref{upbcl} and
\eqref{allkt} via Theorem \ref{chat}.

The metric entropy upper bounds \eqref{simmet} and \eqref{simmet1}
play a very important role in these proofs. Both these bounds will be
derived as a consequence of the following entropy result. This result 
involves the resolution $\delta$ of the grid $\grid$ (defined in
\eqref{frg}) which, by \eqref{nlob}, is of order 
$n^{-1/d}$. Even though we only need the entropy bound for $p = 2$, we
state the next result for the discrete $L_p$ metric $(f, g) \mapsto
\ell_{\grid}(f - g, \Omega, p)$ for  every $1 \leq p < \infty$ where   
\begin{equation}\label{ellSdef}
  \ell_{\grid}(f,\Omega, p)=\left(\frac{1}{n}\sum_{s\in \Omega\cap
      \grid}|f(s)|^p\right)^{1/p}.  
\end{equation}
The $\epsilon$-covering number of a space $\F$ of functions
on $\Omega$ under the metric $(f, g) \mapsto \ell_{\grid}(f - g,
\Omega, p)$ will be denoted by $N(\epsilon, \F, \ell_{\grid}(\cdot,
\Omega, p))$.

\begin{theorem}\label{polytope}
  Suppose $\Omega$ is of the form \eqref{parallel} and satisfies
  \eqref{OmegaAss}. There exists $c_{d,p}$ depending only on $d$ and
  $p$ such that for every $\epsilon > 0$ and $t > 0$,  
  \begin{equation}\label{polytope.eq}
\log N(\epsilon, \{f \in \C(\Omega) : \ell_{\grid}(f, \Omega,
        p) \leq t\},
        \ell_{\grid}(\cdot,\Omega, p)) \le
        [c_{d,p}\log(1/\delta)]^{F}\left(\frac{t}{\epsilon}
      \right)^{d/2}.
        \end{equation}
      \end{theorem}
Theorem \ref{polytope} is the first entropy result for convex
functions that deals with the discrete metric $\ell_{\grid}(\cdot,
\Omega, p)$ (all previous results
hold for continuous $L_p$ metrics). Its proof is similar to but longer
than the proof of the related result Theorem \ref{Thm:Bddcube}. The
two bounds \eqref{simmet} and \eqref{simmet1} follow easily from
Theorem \ref{polytope} as proved in Appendix \ref{fdsimproofs}.  

\section{Discussion}\label{discu}
We conclude by addressing a few natural questions and extensions that
arise from our main results.

\subsection{Optimality vs Suboptimality}\label{ballcase} 
As already indicated, the statement ``the LSE in convex regression is
rate optimal'' can be both true and false in random-design depending
on details and context. In this paper, we proved that, if the domain
is polytopal and the dimension $d \geq 5$, this statement is false
meaning that the LSE is minimax suboptimal. On the other hand, if the domain
is the unit ball, then the statement is true for all $d \geq 2$ meaning that
the LSE is minimax optimal (this optimality has been proved in
\cite{kur2019optimality, han2019global}).

Here we briefly outline the two key differences (relevant to LSE
suboptimality/optimality) in the structures of $\C^B(\Omega)$
corresponding to the two cases where $\Omega$ is a polytope and
$\Omega$ is the unit ball. We assume $d \ge 2$ below. 
\begin{enumerate}
\item The main difference is that the $L_2(\P)$ metric entropy of
  $\C^B(\Omega)$ is $\epsilon^{-d/2}$ in the polytopal case and
  $\epsilon^{-(d-1)}$ in the ball case (these entropy results are due
  to \cite{gao2017entropy}). The larger metric entropy
  directly leads to the slower minimax rate $n^{-2/(d+1)}$ in the ball
  case compared to $n^{-4/(d+4)}$  in the polytopal case. In the
  polytopal case, constructions for the proof of the metric entropy
  lower bound are based on perturbations of a smooth convex
  function. This construction also leads to the same $\epsilon^{-d/2}$
  lower bound in the ball case. However the improved lower bound of
  $\epsilon^{-(d-1)}$ in the 
  ball case is obtained by taking indicator-like convex functions which are zero
  in the interior of the ball but rapidly rise to 1 near a subset of
  the boundary. This construction is
  feasible because of the high 
  complexity of the boundary of the ball, as well as because of the
  nature of the $L_2(\P)$ metric. If the metric is changed to
  $L_1(\P)$, the metric entropy remains $\epsilon^{-d/2}$ for both the
  ball and the polytopal case. It may be helpful to note here that if
  $f = I_A$ is the indicator function of a set $A$, then
  $\|f\|_{L_2(\P)} = \sqrt{\P(A)}$ which can be much larger than
  $\|f\|_{L_1(\P)} = \P(A)$.
\item Now let us discuss the behavior of the LSE. In the polytopal
  case, when $d \geq 5$, the supremum risk of the LSE over
  $\C^B(\Omega)$ is $n^{-2/d}$ (upto log factors). This exceeds the
  minimax rate $n^{-4/(d+4)}$ implying that the LSE is minimax
  suboptimal. On the other hand, note that $n^{-2/d}$ is actually
  smaller than the minimax rate of $n^{-2/(d+1)}$ in the ball
  case. This might suggest that the LSE might be minimax optimal in
  the ball case. Kur et al. \cite{kur2019optimality} proved that the LSE achieves
  the rate $n^{-2/(d+1)}$ in the ball case implying optimality. At a
  very high level, their argument is as follows. From Theorem
  \ref{chat}, it is clear that the risk of the LSE can be upper bounded
  via upper bounds on the expected supremum
  \begin{equation*}
    \E \sup_{f \in \C^B(\Omega): \ell_{\P_n}(f, f_0)  \leq t}
    \frac{1}{n} \sum_{i=1}^n \xi_i \left(f(X_i) - f_0(X_i) \right)
  \end{equation*}
  where $f_0 \in \C^B(\Omega)$ is the true
  function. Kur et al. \cite{kur2019optimality} upper bound the above by first
  removing the localization constraint $\ell_{\P_n}(f, f_0) \leq t$
  leading to
  \begin{equation*}
    \E \sup_{f \in \C^B(\Omega)}
    \frac{1}{n} \sum_{i=1}^n \xi_i \left(f(X_i) - f_0(X_i) \right)
    \leq \E \sup_{f \in \C^{2B}(\Omega)}
    \frac{1}{n} \sum_{i=1}^n \xi_if(X_i). 
  \end{equation*}
  The right hand side above is the Gaussian complexity of the class
  $\C^{2B}(\Omega)$. Kur et al. \cite{kur2019optimality} then show that the
  Gaussian complexity is upper bounded by the discrepancy between
  $\P_n$   and $\P$ over all compact convex subsets of $\Omega$:
  \begin{equation}\label{setdisc}
    \E \sup_{C \subseteq \Omega: C \text{ is compact, convex}}
    |\P_n(C) -\P(C)|. 
  \end{equation}
  They then use chaining with respect to the $L_1$ norm to prove that
  \eqref{setdisc} is bounded by $n^{-2/(d+1)}$. This argument reveals
  that, like the minimax rate which is driven 
  by indicator-like functions over convex sets, the LSE accuracy is
  also driven by the discrepancy over convex
  sets. Han \cite{han2019global} gives examples of other problems where
  LSE accuracy is also driven by discrepancy over sets and the LSE
  turns out to be optimal in these problems as well. 
\end{enumerate}

\subsection{Additional remarks on the fixed-design results}\label{fdnp} 
In the fixed-design setting (where the design
points are given by grid points \eqref{frg} intersected with $\Omega$
and the loss function is \eqref{ldef}), we proved that when $\Omega$
is a polytope of the form \eqref{parallel}, the minimax rate over each
class $\C_L^B(\Omega)$, $\C_L(\Omega)$, $\C^B(\Omega)$, $\F^L(\Omega)$
equals $n^{-4/(d+4)}$ (up to logarithmic factors). We also proved
(Theorem \ref{minsub}) that the supremum risk of the unrestricted LSE
$\hat{f}_n(\C(\Omega))$ over each of these classes is at least
$\Tilde{\Omega}(n^{-2/d})$ for $d \geq 5$. The proof of Theorem
\ref{minsub} can be easily 
modified to prove that the restricted LSE over $\F$ also has
supremum risk of at least $\Tilde{\Omega}(n^{-2/d})$ and hence is minimax
suboptimal over $\F$,  where $\F$ is any of the four classes
$\C_L^B(\Omega)$, $\C_L(\Omega)$, $\C^B(\Omega)$, $\F^L(\Omega)$. 

Now, let us briefly comment on the case when $\Omega$ is
not a polytope. Under the Lipschitz constraint, the story is the
same as in the case of random design. Specifically, the minimax rate
for $\C_L^B(\Omega)$ and $\C_L(\Omega)$ equals $n^{-4/(d+4)}$ for
every $\Omega$ satisfying \eqref{OmegaAss} regardless of whether
$\Omega$ is polytopal or not (this is proved by essentially the same
argument as in the proof of Proposition
\ref{MinimaxLips}; the main idea is that the metric entropy of
$\C_L^B(\Omega)$ is not influenced by the boundary of
$\Omega$). Furthermore, the restricted LSE over   
each of these classes achieves supremum risk of at least $\Tilde{\Omega}(n^{-2/d})$. The proof of these results can be obtained by
suitably modifying the proof of Theorem \ref{minsub} following ideas
in the proof of Theorem \ref{implb.blse}.

Without the Lipschitz constraint, things are more complicated. In
Section \ref{intro}, we discussed that when $\Omega$ is the unit 
ball for $d \geq 2$ and we are in the random-design setting, the minimax rate over
$\C^B(\Omega)$ equals $n^{-2/(d+1)}$ as proved by Han and Wellner
\cite{HW16}. Additionally, we 
noted the minimax \textit{optimality} of the LSE over
$\C^B(\Omega)$ as proved by Kur et al. \cite{kur2019optimality}. The
analogues of these results for the fixed design setting are  unknown.

We summarize this discussion on the fixed-design setting in Table
\ref{fixeddesignsummary}.

\begin{table}[h]
\centering
\small
\begin{tabular}{|p{1.15cm}|p{1.5cm}|p{2cm}|p{3cm}|p{3cm}|}
\hline
Class $\mathcal{F}$ & Assumption on $\Omega$ & Minimax Rate (up to logs) & $\underset{f
                                                              \in \F}{\sup}
                                                              \E_f
                                                              \ell_{\P_n}^2(\hat{f}_n(\F),
                                                                           f
                                                                           )$
                                                                           (upto logs)
  & Minimax Optimality of $\hat{f}_n(\F)$ (upto logs) \\
  \hline
  $\F^L(\Omega)$ & polytope & $n^{-4/(d+4)}$ & 
$n^{-4/(d+4)}$, $d \leq 4$ 
\newline
$n^{-2/d}$, $d \geq 5$
  &
    Optimal for $d \leq 4$ 
\newline
Suboptimal for $d \geq 5$
\\
\hline  
$\mathcal{C}_L^B(\Omega)$ & any convex body & $n^{-4/(d+4)}$ & 
$n^{-4/(d+4)}$, $d \leq 4$ 
\newline
$n^{-2/d}$, $d \geq 5$
& 
Optimal for $d \leq 4$ 
\newline
Suboptimal for $d \geq 5$
\\
  \hline
$\mathcal{C}_L(\Omega)$ & any convex body & $n^{-4/(d+4)}$ & 
$n^{-4/(d+4)}$, $d \leq 4$ 
\newline
$n^{-2/d}$, $d \geq 5$
& 
Optimal for $d \leq 4$ 
\newline
Suboptimal for $d \geq 5$
\\
 \hline
$\mathcal{C}^B(\Omega)$ & polytope & $n^{-4/(d+4)}$ & 
$n^{-4/(d+4)}$, $d \leq 4$ 
\newline
$n^{-2/d}$, $d \geq 5$
& 
Optimal for $d \leq 4$ 
\newline
Suboptimal for $d \geq 5$
\\
  \hline
$\mathcal{C}^B(\Omega)$ & ball, $d \geq 2$ & open question & 
open question
& 
open question
\\
\hline  
\end{tabular}
\caption{Analogue of Table \ref{randomdesignsummary} for the
  fixed-design setting} 
\label{fixeddesignsummary}
\end{table}

The larger
minimax rate of $n^{-2/(d+1)}$ (compared to $n^{-4/(d+4)}$ for
polytopal $\Omega$) in the random design case can be attributed to the
increased metric entropy (in the $L_2$ metric with respect to the continuous uniform probability measure on $\Omega$) of $\C^B(\Omega)$ due to the curvature of the
boundary of $\Omega$ (see \cite[Subsection 2.10]{gao2017entropy}). In the fixed design case, the boundary of
$\Omega$ is also expected to elevate the minimax rate above
$n^{-4/(d+4)}$. However, the precise minimax rate in this context seems
hard to determine (note that there are no existing metric entropy results for $\C^B(\Omega)$ in the discrete $L_2$ metric with respect to the grid points of the fixed design) and it could range between  $n^{-4/(d+4)}$ and
$n^{-2/(d+1)}$. Furthermore, the question of whether the LSE over 
$\C^B(\Omega)$ and the unrestricted convex LSE are optimal or
suboptimal in the fixed design setting for non-polytopal $\Omega$ also
remains open.

\subsection{Additional discussion on the random-design setting}\label{gendensran}
In our random design setting with $X_1, \dots, X_n
\overset{\text{i.i.d}}{\sim} \P$, we assumed that $\P$ is the uniform
distribution on the convex body $\Omega$. From an examination
of our proofs, it should be clear that all our random design results
would continue to work under the more general assumption that $\P$ has a density on
$\Omega$ that is bounded from above and below by 
positive constants. With this assumption, our bounds will involve
additional multiplicative constants depending on the density
constants. We worked with the simpler uniform distribution
assumption because, essentially, the proof in the more general case is
reduced to the uniform case.

One can also consider the case where $\P$ has full support over $\R^d$
such as the case when $\P$ is the standard multivariate normal
vdistribution on $\R^d$ (here $\Omega = \R^d$). In this case,
boundedness would not make sense as non-constant convex functions
cannot be bounded over the whole of $\R^d$. As the Lipschitz assumption
is still reasonable, one may study the minimax rate over
$\C_L(\R^d)$, and the optimality of the 
Lipschitz convex LSE $\hat{f}_n(\C_L(\R^d))$ over $\C_L(\R^d)$ (note
that the loss function is still 
given by \eqref{ranloss} but now $\P$ has full support over
$\R^d$). When $\P$ is the standard multivariate normal distribution,
we believe that the minimax rate over $\C_L(\R^d)$ should be
$n^{-4/(d+4)}$ and that  the supremum 
risk of $\hat{f}_n(\C_L(\R^d))$ over $\C_L(\R^d)$ is bounded below by
$n^{-2/d}$ (up to logarithmic factors) for $d \geq 5$. The intuition
is that $\P\{x \in \R^d : \|x\| > r\}$ decreases exponentially in $r$,
while the metric entropy (of bounded Lipschitz convex functions)
over $\{x \in \R^d : \|x\| \leq r\}$ only grows polynomially in
$r$. We leave a principled study of such unbounded design
distributions to future work. 

\subsection{Beyond Gaussian noise}\label{subgau}
Throughout, we assumed that the errors in the regression model
\eqref{eq:RegMdl} 
are Gaussian, i.e.  $\xi_1, \dots, \xi_n \overset{\text{i.i.d}}{\sim} N(0,
\sigma^2)$. It is natural to ask if 
the results continue to hold if the errors have mean zero and variance
$\sigma^2$ but a non-Gaussian distribution. For sub-Gaussian errors, we
believe that certain variants of our results can be proved with
additional work. One important bottleneck in the extension of our
main results (Theorem \ref{implb.blse} and Theorem \ref{minsub}) to
sub-Gaussian errors is the result of 
\cite{Chat14} stated as Theorem \ref{chat}. This result was
originally stated and proved in \cite{Chat14} for Gaussian errors and
we do not know if an extension to sub-Gaussian errors exists in the
literature. The proof is mainly based on the concentration of the
random variable (for fixed $t > 0$):
\begin{equation*}
  F(\vec \xi,t): = \sup_{g \in \F: \ell_{\P_n}(f, g) \leq t} \frac{1}{n} \sum_{i=1}^n
  \xi_i \left(f(X_i) - g(X_i) \right)
\end{equation*}
which is proved in \cite{Chat14}, by using the fact that the above is a
$t$-Lipschitz function of i.i.d Gaussian errors $\vec \xi:= (\xi_1, \dots,
\xi_n)$. 

For sub-Gaussian $\xi_1, \dots, \xi_n$, one can first argue boundedness of
$\xi_1, \dots, \xi_n$ by $C\sqrt{\log n}$ with high probability (say
probability of $1-O(n^{-1})$), and then
invoke concentration of convex Lipschitz functions of bounded random
variables (see e.g., \cite[Theorem 6.6]{Talagrand96newlook}); note that $\vec \xi \mapsto F(\vec \xi, t)$ is convex in $\vec \xi$. This
should lead to a version of Theorem \ref{chat}, for sub-Gaussian errors
(although with weaker control on the probabilities involved). 

Note that the proofs of Theorems \ref{implb.blse} and \ref{minsub}, also use certain other tools such as the Sudakov minoration inequality which are also only valid for Gaussian errors. Their use should also be replaced by
appropriate multiplier inequality for the suprema of an empirical
process (such results can be found in e.g., \cite[Chapter
2.9]{vdVaartWellner96} and \cite{han2019convergence}). We leave, to
future work, a rigorous extension of our results 
to sub-Gaussian errors.

\subsection{Rates of convergence in the interior of the domain
  $\Omega$}\label{domainint}
It is natural to wonder if our suboptimality results are caused by
boundary effects, and if the LSEs are optimal or suboptimal in the
interior of the domain $\Omega$. Concretely, let $\Omega_0$ denote a
compact, convex region that is contained in the interior of
$\Omega$, and that the loss function is given by
\begin{equation}\label{modilo}
  \int_{\Omega_0} \left(f(x) - g(x) \right)^2 d\P(x). 
\end{equation}
The question then is whether the LSEs considered in Theorem
\ref{implb.blse} are still suboptimal under the above modified loss
function. This is a difficult question that is hard to resolve with
the methods used for the proof of Theorem \ref{implb.blse}. However,
we believe that the result should be true because of the following
heuristic arguments. Consider, for concreteness, the bounded convex
LSE $\hat{f}_n(\C^B(\Omega))$. Every function in $\C^B(\Omega)$ is
actually both $B$-bounded and $O(B)$-Lipschitz inside the smaller domain
$\Omega_0$ -- \emph{regardless} of the shape of
$\Omega_0$). Because the entropy of $O(B)$-bounded and Lipschitz convex functions
on $\Omega_0$ equals $C(B) \cdot \epsilon^{-d/2}$ regardless of the shape of
$\Omega_0$, it gives us reason to believe that
$\hat{f}_n(\C^B(\Omega))$ would be minimax suboptimal (for $d \geq 5$)
with respect to the modified loss \eqref{modilo}. Proving this is
challenging, mainly because the behavior of the LSE
$\hat{f}_n(\C^B(\Omega))$ inside $\Omega_0$ is also influenced
(in a complicated way) by the observations lying outside
$\Omega_0$. Therefore, we would like to highlight this question as an open problem -- as we do not know how to ``decouple'' the performance of the LSE from the domain $\Omega_0$ from $\Omega\setminus \Omega_0$ under $\C^B(\Omega)$.

\subsection{Results for LSEs with smoothness constraints (without
  convexity)}\label{gensublse}
All the results in this paper apply to LSEs over classes of convex
functions (with additional constraints such as boundedness and/or
being Lipschitz). It is natural to ask if similar results of
suboptimality holds for LSEs with purely smoothness constraints
(without any shape constraints such as convexity). The methods of this
paper are closely tied to convexity, and we leave a study of purely
smoothness constrained LSEs for future work. For a concrete problem in
this direction, consider the class of $L$-Lipschitz functions on
$\Omega$:
\begin{equation*}
 \mathcal{L} := \left\{f: \Omega \rightarrow \R \text{ such that } f \text{ is }
    L\text{-Lipschitz on } \Omega \right\}. 
\end{equation*}
For a natural $\Omega$ (e.g., $\Omega = [-1, 1]^d$), the question is
whether the LSE over $\mathcal{L}$ is minimax optimal over
$\mathcal{L}$ for all $d \geq 1$. One can also ask the same question
for function classes with constraints on second (and higher) order
derivatives. We would like to highlight these as open problems -- as our techniques do not apply on these classes.

\begin{acks}
We are truly thankful to the Associate Editor and three anonymous 
referees for their comprehensive reviews of our earlier
manuscript. Their insightful feedback significantly enhanced both the
content and organization of the paper.   
\end{acks}

\begin{funding}
  The first author was funded by the Center for Minds, Brains and
  Machines, funded by NSF award CCF-1231216. The second author was
  funded by NSF Grant OCA-1940270. The third author was funded by NSF
  CAREER Grant DMS-1654589. The fourth author was funded by NSF Grant
  DMS-1712822.  
\end{funding}

\appendix

\newpage

The rest of this paper consists of three sections: Appendix
\ref{minimax.proofs}, 
Appendix \ref{rd.proofs} and Appendix \ref{fd.proofs}. Appendix
\ref{minimax.proofs} contains proofs of all results in Section
\ref{minimax}. Appendix \ref{rd.proofs} contains proofs of all results
in Section \ref{sec.mores}. The proof sketch given in Subsection
\ref{pfskeblse} is followed and results quoted in that subsection are
also proved in Appendix \ref{rd.proofs}. Appendix \ref{fd.proofs}
contains proofs of all results in Section \ref{fdresults}. The proof
sketch given in Subsection \ref{pfskefdr} is followed and results
quoted in that subsection are also proved in Appendix
\ref{fd.proofs}.

\section{Proofs of Minimax Rates for Convex
  Regression}\label{minimax.proofs}
This section contains the proofs of Proposition
\ref{MinimaxLips}, Proposition \ref{minratelin} and Proposition
\ref{fixlobo} (these results were stated in Section \ref{minimax}).

For the proof of Proposition \ref{MinimaxLips} below, we need Lemma
\ref{TwoLosses} which allows switching between the two loss functions
$\ell_{\P_n}^2$ and $\ell_{\P}^2$. Lemma
\ref{TwoLosses} is stated in 
Section \ref{rd.proofs} because it is also crucial for the proof of 
Theorem \ref{implb.blse}.

\begin{proof}[Proof of Proposition \ref{MinimaxLips}]
  The lower bound for the minimax rate follows from the corresponding
  result for the smaller class $\C_L^L(\Omega)$. For the upper bound,
  let us first describe the estimator that we work with. Let $T_i :=
  Y_i - \bar{Y}$ (where $\bar{Y} := (Y_1 + \dots + Y_n)/n$) for each
  $i = 1, \dots, n$. For a finite subset $\finiteset$ of the function
  class $\C_{L}^{2L}(\Omega)$, let $\hat{h}_n^{\finiteset}$ denote any least
  squares 
  estimator over $\finiteset$ for the data $(X_1, T_1), \dots, (X_n,
  T_n)$. In other words, 
  \begin{equation*}
    \hat{h}_n^{\finiteset} \in \argmin_{h \in \finiteset} \sum_{i=1}^n \left(T_i
      - h(X_i) 
    \right)^2. 
  \end{equation*}
  Our estimator is given by the convex function: 
  \begin{equation}\label{estbarmean}
    x \mapsto \bar{Y} + \hat{h}_n^{\finiteset}(x)
  \end{equation}
  for an appropriately chosen covering subset $\finiteset$ of
  $\C_L^{2L}(\Omega)$. The proof below shows that this estimator
  achieves the rate $n^{-4/(d+4)}$ uniformly over $\C_L(\Omega)$. Fix
  a ``true'' function $f_0 \in \C_L(\Omega)$. We first bound the risk of
  $\hat{h}_n^{\finiteset}$. For any arbitrary nonnegative valued
  functional $H$, 
  the following inequality holds:
  \begin{equation*}
    H(\hat{h}_n^{\finiteset}) \leq \sum_{h \in \finiteset} H(h) \exp
    \left(\frac{1}{8 
        \sigma^2} \sum_{i=1}^n (T_i - \hat{h}_n^{\finiteset}(X_i))^2 -
      \frac{1}{8 
      \sigma^2} \sum_{i=1}^n (T_i - h(X_i))^2 \right) 
\end{equation*}
because of the presence of the term corresponding to $h = \hat{h}_n^{\finiteset}
\in \finiteset$ on the right hand side. Because
$\hat{h}_n^{\finiteset}$ minimizes sum of 
squares over $\finiteset$, we can replace $\hat{h}_n^{\finiteset}$ by
any other element $h' \in \finiteset$ leading to
  \begin{equation}\label{binq}
    H(\hat{h}_n^{\finiteset}) \leq \sum_{h \in \finiteset} H(h) \exp \left(\frac{1}{8
        \sigma^2} \sum_{i=1}^n (T_i -h'(X_i))^2 - \frac{1}{8
      \sigma^2} \sum_{i=1}^n (T_i - h(X_i))^2 \right)
\end{equation}
for every $h' \in \finiteset$. We now take the expectation on both
sides of the 
above inequality conditioned on $X_1, \dots, X_n$. The following
function $h_0 : \Omega \rightarrow \R$ will play a key role in the
sequel
\begin{equation*}
 h_0(x) := f_0(x) - \frac{f_0(X_1) + \dots + f_0(X_n)}{n}. 
\end{equation*}
$h_0$ is Lipschitz with Lipschitz constant $L$ and, moreover, $h_0$ is
uniformly bounded by $2L$ because:
\begin{align*}
  |h_0(x)| &\leq \frac{1}{n} \sum_{i=1}^n|f_0(x) - f_0(X_i)| \\ &\leq
  \frac{L}{n} \sum_{i=1}^n \|x - X_i\| \leq L \sup_{x,x' \in \Omega}
  \|x - x'\| \leq 2L. 
\end{align*}
In other words, $h_0 \in \C_L^{2L}(\Omega)$. In order to take the
expectation of both sides of \eqref{binq}, note that the
conditional distribution of $T_1, \dots, T_n$ given  $X_1, \dots, 
X_n$ is multivariate normal with mean vector $(h_0(X_1), \dots,
h_0(X_n))$ and covariance matrix $\sigma^2 \left(I - \frac{\one
    \one^T}{n}\right) \preceq \sigma^2 I$. Here $I$ is the identity
matrix and $\one$ is 
the $n \times 1$ vector of ones. By a straightforward calculation, we 
obtain
\begin{align*}
  &\E H\left(\hat{h}_n^{\finiteset} \mid X_1, \dots, X_n \right) \\
  &\leq \sum_{h \in \finiteset} H(h) \exp\left(\frac{n}{8 \sigma^2}
    \ell^2_{\P_n}\left(h_0, h' \right)  - \frac{n}{8 \sigma^2}
    \ell_{\P_n}^2\left(h_0, h \right) + \frac{n}{32 \sigma^2}
    \ell_{\P_n}^2\left(h, h' \right)  \right) \\
  &\leq \sum_{h \in \finiteset} H(h) \exp\left(\frac{3n}{16 \sigma^2}
    \ell^2_{\P_n}\left(h_0, h' \right)  - \frac{n}{16 \sigma^2}
    \ell_{\P_n}^2\left(h_0, h \right)  \right)  \\
  &=   \exp\left(\frac{3n}{16 \sigma^2}
    \ell^2_{\P_n}\left(h_0, h' \right)  \right) \sum_{h \in \finiteset} H(h)
    \exp\left(- \frac{n}{16 \sigma^2}
    \ell_{\P_n}^2\left(h_0, h \right)   \right) 
\end{align*}
where we used the standard fact
$\ell_{\P_n}^2(h, h') \leq 2 \ell_{\P_n}^2(h_0, h) + 2
\ell_{\P_n}^2(h_0, h')$ for the penultimate inequality. As $h' \in
\finiteset$ is arbitrary, we can take an infimum over $h' \in
\finiteset$ to obtain 
\begin{align*}
  &\E H\left(\hat{h}_n^{\finiteset} \mid X_1, \dots, X_n \right) \\
    &\leq   \exp\left(\frac{3n}{16 \sigma^2}
    \inf_{h' \in \finiteset}\ell^2_{\P_n}\left(h_0, h' \right)
      \right) \sum_{h \in \finiteset} H(h) 
    \exp\left(- \frac{n}{16 \sigma^2}
    \ell_{\P_n}^2\left(h_0, h \right)   \right) 
\end{align*}
Because $\ell_{\P_n}^2(h_0, h') \leq \|h_0 - h'\|_{\infty}^2$ and $h_0
\in \C_{L}^{2L}(\Omega)$,
\begin{align*}
  \inf_{h' \in \finiteset} \ell_{\P_n}^2\left(h_0, h' \right) \leq \inf_{h' \in
  \finiteset} \|h_0 - h'\|_{\infty}^2 \leq \sup_{g \in \C_L^{2L}(\Omega)} \inf_{h' \in
  \finiteset} \|g - h'\|_{\infty}^2.  
\end{align*}
We have thus proved
\begin{align*}
  &\E H\left(\hat{h}_n^{\finiteset} \mid X_1, \dots, X_n \right) \\
    &\leq   \exp\left(\frac{3n}{16 \sigma^2}
\sup_{g \in \C_L^{2L}(\Omega)} \inf_{h' \in
  \finiteset} \|g - h'\|_{\infty}^2 \right) \sum_{h \in \finiteset} H(h)
    \exp\left(- \frac{n}{16 \sigma^2}
    \ell_{\P_n}^2\left(h_0, h \right)   \right). 
\end{align*}
The choice
\begin{align*}
  H(h) := \exp\left(\frac{n}{16 \sigma^2}
    \ell_{\P_n}^2\left(h_0, h \right)   \right). 
\end{align*}
leads to
\begin{align*}
 & \E \left[ \exp \left(\frac{n}{16 \sigma^2} \ell_{\P_n}^2\left(h_0,
  \hat{h}_n^{\finiteset} \right) \right) \bigg| X_1, \dots, X_n \right] \\ &\leq \exp \left(\frac{3n}{16 \sigma^2}
\sup_{g \in \C_L^{2L}(\Omega)} \inf_{h' \in
  \finiteset} \|g - h'\|_{\infty}^2 + \log |\finiteset| \right) 
\end{align*}
where $|\finiteset|$ denotes the cardinality of the finite set
$\finiteset$. Jensen's 
inequality on the left hand side gives 
\begin{align*}
  \E \left(\ell_{\P_n}^2 \left(h_0, \hat{h}_n^{\finiteset} \right)
  \mid X_1, 
  \dots, X_n \right) \leq 3\sup_{g \in \C_L^{2L}(\Omega)} \inf_{h' \in
  \finiteset} \|g - h'\|_{\infty}^2 + \frac{16 \sigma^2}{n} \log
  |\finiteset|.  
\end{align*}
Taking expectations with respect to $X_1, \dots, X_n$ (the right hand
side above is actually nonrandom), we get the same bound
unconditionally:
\begin{align*}
  \E \ell_{\P_n}^2 \left(h_0, \hat{h}_n^{\finiteset} \right)  \leq
  3\sup_{g \in 
  \C_L^{2L}(\Omega)} \inf_{h' \in 
  \finiteset} \|g - h'\|_{\infty}^2 + \frac{16 \sigma^2}{n} \log
  |\finiteset|.  
\end{align*}
We now take $\finiteset$ to be an $\epsilon$-covering subset of
$\C_L^{2L}(\Omega)$ for an appropriate
$\epsilon$. By the classical metric entropy result of
\cite{Bronshtein76}, for each $\epsilon > 0$, we can find
$\finiteset$ such that
\begin{align*}
  \sup_{g \in \C_L^{2L}(\Omega)} \inf_{h' \in
  \finiteset} \|g - h'\|_{\infty}^2 \leq \epsilon^2 ~~ \text{ and } ~~
  \log |\finiteset| 
  \leq C_d \left(\frac{L}{\epsilon} \right)^{d/2}. 
\end{align*}
Taking $\epsilon$ to be of order $n^{-2/(d+4)}$ and $\finiteset$ to
the 
corresponding $\epsilon$-covering subset of $\C_L^{2L}(\Omega)$, we
deduce
\begin{align*}
  \E \ell_{\P_n}^2 \left(h_0, \hat{h}_n^{\finiteset} \right) \leq C_{d, L,
  \sigma} n^{-4/(d+4)}. 
\end{align*}
We convert this bound to $\ell_{\P}^2$ via Lemma \ref{TwoLosses}: 
\begin{align*}
  \E \ell_{\P}^2 \left(h_0, \hat{h}_n^{\finiteset} \right) &\leq 8   \E
                                                   \ell_{\P_n}^2
                                                   \left(h_0,
                                                   \hat{h}_n^{\finiteset} \right)
                                                       + 2 \E
                                                       \left(\ell_{\P}(h_0,
                                                       \hat{h}_n^{\finiteset})
                                                             - 
                                                       2 \ell_{\P_n}(h_0,
                                                       \hat{h}_n^{\finiteset})
                                                  \right)^2 \\
  &\leq 8   \E
                                                   \ell_{\P_n}^2
                                                   \left(h_0,
                                                   \hat{h}_n^{\finiteset}
    \right) 
                                                       + 2 \E \sup_{f,
    g \in \C_L^{2L}(\Omega)}
                                                       \left(\ell_{\P}(f,
                                                       g) -
                                                       2 \ell_{\P_n}(f,
                                                       g)
    \right)^2 \\
  &\leq 8   \E
                                                   \ell_{\P_n}^2
                                                   \left(h_0,
                                                   \hat{h}_n^{\finiteset} \right)
    + C_d n^{-4/(d+4)} L^{2d/(d+4)} + C \frac{L^2}{n} \\ &\leq C_{d, L,
    \sigma} n^{-4/(d+4)}. 
\end{align*}
By the elementary inequality $(a+b)^2 \geq a^2/2 - b^2$, we get
\begin{align*}
  \ell_{\P}^2\left(h_0, \hat{h}_n^{\finiteset} \right) \geq \frac{1}{2}
  \ell_{\P}^2\left(f_0, \mu + \hat{h}_n^{\finiteset} \right) -
  \left(\frac{f_0(X_1) + \dots + f_0(X_n)}{n} - \mu \right)^2
\end{align*}
where $\mu := \E f_0(X_1)$. Thus
\begin{align*}
  \E \ell_{\P}^2\left(f_0, \mu + \hat{h}_n^{\finiteset} \right) &\leq 2 \E
  \ell_{\P}^2\left(h_0, \hat{h}_n^{\finiteset} \right) + 2 \E
                                                       \left(\frac{f_0(X_1) + \dots + f_0(X_n)}{n} - \mu \right)^2 \\
  &\leq C_{d, L, \sigma} n^{-4/(d+4)} + \frac{1}{n} \E \left(f_0(X_1)
    - f_0(X_2) \right)^2 \\ &\leq C_{d, L, \sigma} n^{-4/(d+4)} + \frac{4
    L^2}{n} \leq C_{d, L, \sigma} n^{-4/(d+4)}. 
\end{align*}
Finally
\begin{align*}
  \E \ell_{\P}^2\left(f_0, \bar{Y} + \hat{h}_n^{\finiteset} \right) &\leq 2 \E
                                                           \ell_{\P}^2\left(f_0,
                                                           \mu +
                                                           \hat{h}_n^{\finiteset} 
                                                           \right) + 2
                                                           \E 
                                                           \left(\bar{Y}
                                                           - \mu
                                                           \right)^2
  \\
  &\leq C_{d, L, \sigma} n^{-4/(d+4)} + \frac{2 \sigma^2}{n} \leq
    C_{d, L, \sigma} n^{-4/(d+4)}. 
\end{align*} 
This completes the proof of Proposition \ref{MinimaxLips}. 
\end{proof}

We next turn to the proof of Proposition \ref{minratelin} which gives
an upper bound on the minimax rate for the class $\F^{\lin}(\Omega)$
in the fixed design setting. This proof is similar
to that of Proposition \ref{MinimaxLips} but somewhat simpler as,
in the fixed design setting, we do not need to switch between the two
losses $\ell^2_{\P_n}$ and $\ell^2_{\P}$. The metric entropy result
stated in Theorem \ref{polytope} is an important ingredient of the
following proof. 

\begin{proof}[Proof of Proposition \ref{minratelin}]
  Let us describe the estimator that we work with. Let the linear
  regression solution be given by the affine function $\hat{g}_n$: 
  \begin{equation*}
    \hat{g}_n \in \argmin_{g \in \A(\Omega)} \sum_{i=1}^n \left(Y_i -
      g(X_i) \right)^2. 
  \end{equation*}
 Let $T_i := Y_i - \hat{g}_n(X_i)$ for $i = 1, \dots, n$ be the
 residuals after linear regression. For a finite class $\finiteset$ of
 functions on $\Omega$, let $\hat{h}_n^{\finiteset}$ denote any least squares
 estimator over $\finiteset$ for the data $(X_1, T_1), \dots, (X_n, T_n)$:
 \begin{equation*}
   \hat{h}_n^{\finiteset} \in \argmin_{h \in \finiteset} \sum_{i=1}^n
   \left(T_i - h(X_i) 
   \right)^2 
 \end{equation*}
 We work with the estimator $\hat{g}_n + \hat{h}_n^{\finiteset}$ for an
 appropriate choice of $\finiteset$, and bound its supremum risk over the class
 $\F^{\lin}(\Omega)$. Fix a ``true'' function $f_0 \in
 \F^{\lin}(\Omega)$.  Let $g_o$ be the closest affine function to
 $f_0$ in the $\ell_{\P_n}$ metric:
 \begin{equation*}
   g_0 = \argmin_{g \in \A(\Omega)} \sum_{i=1}^n \left(f_0(X_i) -
     g(X_i) \right)^2. 
 \end{equation*}
 A key role in the sequel will be played by the convex function $h_0
 := f_0 - g_0$ (note that $h_0$ is convex because $f_0$ is convex and
 $g_0$ is affine). Fix a nonnegative valued functional $H$ and start
 with the inequality \eqref{binq}:
 \begin{equation*}
   H(\hat{h}_n^{\finiteset}) \leq \sum_{h \in \finiteset} H(h) \exp \left(\frac{1}{8
        \sigma^2} \sum_{i=1}^n (T_i -h'(X_i))^2 - \frac{1}{8
      \sigma^2} \sum_{i=1}^n (T_i - h(X_i))^2 \right)
\end{equation*}
for every $h' \in \finiteset$. Now take expectations on both sides above
(note that we are working 
in fixed design so $X_1, \dots, X_n$ are fixed grid points in
$\Omega$). To calculate the expectation of the right hand side,
observe that $(T_1, \dots, T_n)$ is multivariate normal with mean
vector $(h_0(X_1), \dots, h_0(X_n))$ and a covariance matrix $\Sigma$
which is dominated by $\sigma^2 I_n$ in the positive semi-definite
ordering. As in the proof of Proposition \ref{MinimaxLips}, we deduce
\begin{equation*}
  \E H(\hat{h}_n^{\finiteset}) \leq \exp\left(\frac{3n}{16 \sigma^2}
    \inf_{h' \in \finiteset}\ell^2_{\P_n}\left(h_0, h' \right)  \right) \sum_{h \in \finiteset} H(h)
    \exp\left(- \frac{n}{16 \sigma^2}
    \ell_{\P_n}^2\left(h_0, h \right)   \right) 
\end{equation*}
Taking
\begin{equation*}
  H(h) := \exp \left(\frac{n}{16 \sigma^2} \ell_{\P_n}^2(h_0, h) \right)
\end{equation*}
and using Jensen's inequality (as in the proof of Proposition
\ref{MinimaxLips}), we obtain
\begin{equation*}
  \E \ell_{\P_n}^2 \left(h_0, \hat{h}_n^{\finiteset} \right) \leq 3 \inf_{h' \in
    \finiteset} \ell_{\P_n}^2 \left(h_0, h' \right) + \frac{16 \sigma^2}{n}\log
  |\finiteset|. 
\end{equation*}
We now specify the finite function class $\finiteset$. For this note that
because $f_0 \in \F^{\lin}(\Omega)$ (i.e., $\inf_{g \in \A(\Omega)}
\ell_{\P_n}(f_0, g_0) \leq \lin$), the function $h_0 = f_0 - g_0$
belongs to the class
\begin{equation*}
  \mathcal{H} := \left\{h \in \C(\Omega) : \frac{1}{n} \sum_{i=1}^n h^2(X_i) \leq
    \lin^2\right\}. 
\end{equation*}
Theorem \ref{polytope} (with $p = 2$ and $t = \lin$) gives an upper
bound for the $\epsilon$-metric entropy of the above function class
under the $\ell_{\P_n}$ metric. This result implies the existence of a
finite set $\finiteset$ satisfying the twin properties:
\begin{equation*}
  \sup_{h_0 \in \mathcal{H}} \inf_{h' \in \finiteset} \ell_{\P_n}^2\left(h_0,
    h' \right) \leq \epsilon^2 ~~ \text{ and } ~~ \log |\finiteset| \leq
  \left(c_d \log n \right)^F \left(\frac{\lin}{\epsilon}
  \right)^{d/2}. 
\end{equation*}
As a result
\begin{equation*}
  \E \ell_{\P_n}^2 \left(h_0, \hat{h}_n^{\finiteset} \right) \leq 3 \epsilon^2 +
  \frac{16 \sigma^2}{n}   \left(c_d \log n \right)^F \left(\frac{\lin}{\epsilon}
  \right)^{d/2}. 
\end{equation*}
The choice $\epsilon = n^{-2/(d+4)}(c_d \log n)^{2F/(d+4)}$ gives
\begin{equation*}
  \E \ell_{\P_n}^2 \left(h_0, \hat{h}_n^{\finiteset} \right) \leq C_{\lin, \sigma}
  n^{-4/(d+4)} (c_d \log n)^{4F/(d+4)}. 
\end{equation*}
Finally
\begin{align*}
  \E \ell_{\P_n}^2 \left(f_0, \hat{g}_n + \hat{h}_n^{\finiteset} \right) &=
 \E \ell_{\P_n}^2 \left(g_0 + h_0, \hat{g}_n + \hat{h}_n^{\finiteset} \right)
  \\ &\leq 2  \E \ell_{\P_n}^2 \left(g_0, \hat{g}_n \right) +\E
       \ell_{\P_n}^2 \left(h_0, \hat{h}^{\finiteset}_n \right) \\
  &\leq \frac{C_d \sigma^2}{n} + C_{\lin, \sigma}
    n^{-4/(d+4)} (c_d \log n)^{4F/(d+4)} \\
  &\leq C_{\lin, \sigma}
    n^{-4/(d+4)} (c_d \log n)^{4F/(d+4)} 
\end{align*}
completing the proof of Proposition \ref{minratelin}. 
\end{proof}

Next we give the proof of Proposition \ref{fixlobo}. This proof is
based on the use of Assouad's lemma with a standard construction.  

\begin{proof}[Proof of Proposition \ref{fixlobo}]
Define the smooth function:
  \begin{equation*}
g(x_1,x_2,\ldots,x_d)=\twopiece{\sum_{i=1}^d
	\cos^3(\pi
        x_i)}{(x_1,x_2,\ldots,x_d)\in[-1/2,1/2]^d}{0}{(x_1,x_2,\ldots,x_d)\notin[-1/2,1/2]^d}. 
\end{equation*}
Note that $\frac{\partial^2 g}{\partial x_i \partial
  x_j} = 0$ for $i \neq j$ and
\begin{equation*}
  \left|\frac{\partial^2 g}{\partial x_i^2} (x_1, \dots, x_d) \right|
  \leq \frac{4 \sqrt{2}}{3} \pi^2
\end{equation*}
which implies that the Hessian of $g$ is dominated by $(4\sqrt{2}
\pi^2/3)$ times the identity matrix. It is also easy to check that the
Hessian of $g$ equals zero on the boundary of $[-0.5, 0.5]^d$.

Fix $\eta > 0$ and let $\grid^{\eta}$ be the $\eta$-grid consisting of all
points $(k_1 \eta, \dots, k_d \eta)$ for $k_1, \dots, k_d \in
\mathbb{Z}$. Let
\begin{equation*}
  T^{\eta} := \{(u_1, \dots, u_d) \in \grid^{\eta} : \prod_{j=1}^d \left[u_j -
      0.5 \eta, u_j + 0.5 \eta \right] \subseteq \Omega  \}. 
\end{equation*}
Assume that $\eta$ is small enough so that the cardinality
$m := |T^{\eta}|$ of $T^{\eta}$ is at least $c_d \eta^{-d}$. For every
point $u \in T^{\eta}$, let
\begin{equation*}
  g_u(x_1, \dots, x_d) := \eta^2 g \left(\frac{x_1 - u_1}{\eta},
    \dots, \frac{x_d - u_d}{\eta} \right). 
\end{equation*}
Clearly $g_u$ is supported on the cube $\prod_{j=1}^d [u_j - 0.5 \eta,
u_j + 0.5 \eta]$ and these cubes for different points $u \in T^{\eta}$
have disjoint interiors.

For each binary vector $\xi = (\xi_{u}, u \in T^{\eta}) \in \{0,
1\}^m$, consider the function  
\begin{equation}\label{gudef}
  G_{\xi}(x) = f_0(x) + \frac{3}{4 \sqrt{2} \pi^2} \sum_{u \in
    T^{\eta}} \xi_u g_u(x)
\end{equation}
where $f_0(x) := \|x\|^2$. It can be verified that $G_{\xi}$ is convex
because $f_0$ has constant Hessian equal to $2$ times the identity,
the Hessian of $g_{s}$ is 
bounded by $(4 \sqrt{2} \pi^2/3)$ and the supports of $g_{s}, s \in
\grid \cap \Omega$ have disjoint interiors. For a sufficiently small
constant $\eta$ and a sufficiently large constant $L$, it is also
clear that $G_{\xi} \in \C_L^L(\Omega)$. We use Assouad's lemma
with this collection $\{G_{\xi}, \xi \in \{0, 1\}^m\}$: 
\begin{equation*}
\mathfrak{R}_n^{\mathrm{fixed}}(\C^L_L(\Omega)) \geq \frac{m}{8}
\min_{\xi \neq \xi'} \frac{\ell_{\P_n}^2(G_{\xi},
  G_{\xi'})}{\Upsilon(\xi, \xi')} \min_{\Upsilon(\xi, \xi') = 1}
\left(1 - \|\P_{G_{\xi}} - \P_{G_{\xi'}}\|_{\text{TV}} \right) 
\end{equation*}
where $\Upsilon(\xi, \xi') := \sum_{u \in T^{\eta}} I\{\xi_u \neq
\xi'_{u}\}$ is the Hamming distance, and $\P_{f}$ is the
multivariate normal distribution with mean vector $(f(X_1), \dots,
f(X_n))$ and covariance matrix $\sigma^2 I_n$. We bound $\ell_{\P_n}^2(G_{\xi},
G_{\xi'})$ from below as
\begin{align*}
  \ell_{\P_n}^2(G_{\xi}, G_{\xi'}) &= \frac{1}{n} \sum_{i=1}^n
                                     \left(G_{\xi}(X_i) -
                                     G_{\xi'}(X_i) \right)^2 \\
  &= \frac{1}{n} \sum_{i=1}^n \left\{\frac{3}{4 \sqrt{2} \pi^2}
    \sum_{u \in T^{\eta}} \left(\xi_u - \xi'_u \right) g_u(X_i)
    \right\}^2  \\
  &= \frac{9}{32 \pi^4 n} \sum_{i=1}^n \sum_{u \in T^{\eta}} I\{\xi_u
    \neq \xi'_u\} g_u^2(X_i) \\
  &= \frac{9}{32 \pi^4} \sum_{u \in T^{\eta}}I\{\xi_u
    \neq \xi'_u\} \frac{1}{n} \sum_{i=1}^n  g_u^2(X_i) 
\end{align*}
To bound $\frac{1}{n} \sum_{i=1}^n  g_u^2(X_i)$ from below, note that
$g_u$ is supported on the cube $\prod_{j=1}^d [u_j - 0.5\eta, u_j +
0.5\eta]$ which has volume $\eta^d$, and also that the ``magnitude''
of $g_u$ on this cube is of order $\eta^2$ (equivalently, the squared
magnitude equals $\eta^4$). Thus when $\eta \geq c 
\delta$ for some $c$, it is straightforward to check
\begin{align*}
  \frac{1}{n} \sum_{i=1}^n  g_u^2(X_i) \geq C_d \eta^d \times \eta^4 =
  C_d \eta^{d+4}. 
\end{align*}
This gives
\begin{align*}
\min_{\xi \neq \xi'}  \frac{\ell_{\P_n}^2(G_{\xi}, 
  G_{\xi'})}{\Upsilon(\xi, \xi')}  \geq C_d \eta^{d+4} 
\end{align*}
For bounding $\|\P_{G_{\xi}} - \P_{G_{\xi'}}\|_{\text{TV}}$, we use
Pinsker's inequality to get
\begin{align*}
  \|\P_{G_{\xi}} - \P_{G_{\xi'}}\|^2_{\text{TV}} &\leq \frac{1}{2} D
  \left(\P_{G_{\xi}} \mathrel{\Vert} \P_{G_{\xi'}} \right) \\ &=
  \frac{n}{4 \sigma^2} \ell_{\P_n}^2\left(G_{\xi}, G_{\xi'}\right) \\&=
                                                                \frac{9n}{124
  \sigma^2 
                                                                \pi^4}
                                                                \sum_{u
                                                                \in
                                                                T^{\eta}}I\{\xi_u
                                                                \neq
                                                                \xi'_u\}
                                                                \frac{1}{n}
                                                                \sum_{i=1}^n
  g_u^2(X_i) \\
  &\leq \tilde{C}_d \frac{n \eta^{d+4}}{\sigma^2} \Upsilon(\xi,
    \xi'). 
\end{align*}
Assouad's lemma (with $m \geq c_d \eta^{-d}$) thus gives 
\begin{align*}
\mathfrak{R}_n^{\mathrm{fixed}}(\C^L_L(\Omega)) \geq c_d \eta^{-d} C_d
  \eta^{d+4} \left(1 - \sqrt{\tilde{C}_d \frac{n
  \eta^{d+4}}{\sigma^2}} \right). 
\end{align*}
The choice
\begin{align*}
  \eta = \left(\frac{\sigma^2}{4 n \tilde{C}_d} \right)^{1/(d+4)}
\end{align*}
leads to
\begin{align*}
  \mathfrak{R}_n^{\mathrm{fixed}}(\C^L_L(\Omega)) \geq c_{d, \sigma}
  n^{-4/(d+4)}. 
\end{align*}
\end{proof}

\section{Proofs of the random design LSE rate lower
  bounds}\label{rd.proofs}    
This section contains the proof of Theorem \ref{implb.blse}. We follow
the sketch given in Subsection \ref{pfskeblse} and results quoted in that
subsection are also proved here. The first three subsections below
contain the proofs of Lemma \ref{GrUB}, Lemma \ref{GrLB} and Lemma
\ref{tfktildelb} respectively. Using these three lemmas, the proof of
Theorem \ref{implb.blse} is completed in Subsection
\ref{comple.implb.blse}. An important role is played by
the main bracketing entropy bound stated in Theorem
\ref{Thm:Bddcube} and this theorem is proved in Subsection
\ref{Gao:subsection}. Finally, Subsection \ref{adproofs} contains the
proofs of Lemma \ref{ballapp} and Lemma \ref{lbf0.ran}. 

\subsection{Proof of Lemma \ref{GrUB}}
The most important ingredient for the proof of Lemma \ref{GrUB} is the
bracketing entropy bound given by Theorem \ref{Thm:Bddcube} which is
proved in Subsection \ref{Gao:subsection}. We also use the following
two standard results from the theory of empirical processes and
Gaussian processes respectively.

\begin{lemma}\label{stanbrack}\cite[Theorem 5.11]{vdG00} 
  Let $\Q$ be a probability on a set $\mathcal{Z}$ and let
  $\Q_n$ be the empirical distribution of $Z_1, \dots, Z_n
  \overset{\text{i.i.d}}{\sim} \Q$. Let $\mathcal{H}$ be a class of
  real-valued functions on $\mathcal{Z}$ and assume that each function
  in $\mathcal{H}$ is  uniformly bounded  by $\Gamma > 0$. Then 
    \begin{equation}\label{expbo}
    \begin{split}
   & \E \sup_{h \in \mathcal{H}} |\Q_nh  - \Q h| \\ &\leq C \inf \left\{a
      \geq \frac{\Gamma}{\sqrt{n}} : a \geq \frac{C}{\sqrt{n}}
      \int_a^{\Gamma} \sqrt{\log N_{[\,]}(u, \mathcal{H}, L_2(\Q))} du
    \right\}.
    \end{split}
  \end{equation}
\end{lemma}

The following result is Dudley's entropy integral bound
\cite{Dudley67}. 
  \begin{theorem}[Dudley]\label{dud}
 Let $\xi_1, \dots, \xi_n \overset{\text{i.i.d}}{\sim} N(0,
 \sigma^2)$. Then for every deterministic $X_1, \dots, X_n$, every
 class of functions $\F$, $f \in \F$ and $t \geq 0$:   
  \begin{equation*}
    \begin{split}
  &    \E \sup_{g \in \F: \ell_{\P_n}(f, g) \leq t} \frac{1}{n} \sum_{i=1}^n \xi_i
      \left(g(X_i) - f(X_i) \right) \\ &\leq \sigma \inf_{0 < \theta \leq t/2}
      \left(\frac{12}{\sqrt{n}} \int_{\theta}^{t/2} \sqrt{\log
          N(\epsilon, \left\{g \in \F: \ell_{\P_n}(f, g) \leq t \right\}, \ell_{\P_n})}
        d\epsilon + 2 \theta \right).
     \end{split} 
    \end{equation*}
  \end{theorem}

In the course of proving Lemma \ref{GrUB}, we will need to work with
both the loss functions $\ell_{\P}$ and $\ell_{\P_n}$. The next result
states that these loss functions  are, up to the factor 2,
sufficiently close (of order $n^{-2/(d+4)}$ which is much smaller than
$n^{-1/d}$ for $d \geq 5$)  allowing us to switch between them.   

\begin{lemma}\label{TwoLosses}
  Suppose $\Omega$ is a convex body satisfying \eqref{OmegaAss}.
Then there exist positive constants $C$ and $C_d$ such that  
  \begin{equation*}
      \begin{split}
&  \P \left\{\sup_{f, g \in \C^B_L(\Omega)} \left(\ell_{\P_n}(f, g)
                                                                        -
    2 \ell_{\P}(f, g) \right) \leq C_d
  n^{-2/(d+4)}(B+L)^{d/(d+4)} \right\} \\ & \geq 1 - C \exp
\left(-\frac{n^{d/(d+4)}}{C_dB^{8/(d+4)}} \right), 
\end{split}
\end{equation*}
and
  \begin{equation*}
      \begin{split}
&  \P \left\{\sup_{f, g \in \C^B_L(\Omega)} \left(\ell_{\P}(f, g)
                                                                        -
    2 \ell_{\P_n}(f, g) \right) \leq C_d
  n^{-2/(d+4)}(B+L)^{d/(d+4)} \right\} \\ & \geq 1 - C \exp
\left(-\frac{n^{d/(d+4)}}{C_d B^{8/(d+4)}} \right),
\end{split}
\end{equation*}
and
\begin{equation*}
  \E \sup_{f,
    g \in \C_L^{B}(\Omega)}
                                                       \left(\ell_{\P}(f,
                                                       g) -
                                                       2 \ell_{\P_n}(f,
                                                       g)
    \right)^2 \leq C_d n^{-4/(d+4)} L^{2d/(d+4)} + \frac{C L^2}{n}, 
  \end{equation*}
and
\begin{equation*}
  \E \sup_{f,
    g \in \C_L^{B}(\Omega)}
                                                       \left(\ell_{\P_n}(f,
                                                       g) -
                                                       2 \ell_{\P}(f,
                                                       g)
    \right)^2 \leq C_d n^{-4/(d+4)} L^{2d/(d+4)} + \frac{C L^2}{n}. 
  \end{equation*}
Additionally, when $\Omega$ is a polytope whose number of facets is
bounded by a constant depending on $d$ alone, all these inequalities
continue to hold if $\C^B_L(\Omega)$ is replaced by the 
larger class $\C^B(\Omega)$ with $(B+L)$ replaced by $B$.  
\end{lemma}

\begin{proof}[Proof of Lemma \ref{TwoLosses}]
  We use the following result which can be found in \cite[Proof of
  Lemma 5.16]{vdG00}: Suppose $\F$ is a class of functions on $\Omega$
  that are uniformly bounded by $\Gamma > 0$. Then there exists a
  positive constant $C$ such that 
  \begin{equation}\label{conemp.1}
    \P \left\{\sup_{f, g \in \F} \left(\ell_{\P_n}(f, g) - 2
        \ell_{\P}(f, g) \right) > C a \right\} \leq C \exp \left(-
      \frac{n a^2}{C \Gamma^2} \right)
  \end{equation}
  and
  \begin{equation}\label{conemp.2}
    \P \left\{\sup_{f, g \in \F} \left(\ell_{\P}(f, g) - 2
        \ell_{\P_n}(f, g) \right) > C a \right\} \leq C \exp \left(-
      \frac{n a^2}{C \Gamma^2} \right)
  \end{equation}
  provided
  \begin{equation}\label{conemp.req}
    n a^2 \geq C \log N_{[\,]}(a, \F, L_2(\P)).
  \end{equation}
  We first apply this result to $\F = \C_L^B(\Omega)$ with $\Gamma =
  B$. From \cite{Bronshtein76}, we get 
\begin{equation}\label{gaoweb}
  \log N_{[\,]}(\epsilon, \C_L^B(\Omega), \metdis) \leq C_d
  \left(\frac{B+L}{\epsilon} \right)^{d/2}
\end{equation}
The above bound implies that \eqref{conemp.req} is
satisfied when $a$ is 
$n^{-2/(d+4)} (B+L)^{d/(d+4)}$ multiplied by a large enough dimensional
constant. Inequalities \eqref{conemp.1} and \eqref{conemp.2} then
lead to the two stated inequalities in Lemma \ref{TwoLosses}. The
expectation bounds are obtained by integrating the probability
inequalities. 

To prove the bounds for $\C^B(\Omega)$ for polytopal $\Omega$, just
use, instead of the result from \cite{Bronshtein76}, the following
due to \cite[Theorem 1.5]{gao2017entropy}:
\begin{equation}\label{bronweb}
  \log N_{[\,]}(\epsilon, \C^B(\Omega), \metdis) \leq C_d
  \left(\frac{B}{\epsilon} \right)^{d/2}. 
\end{equation}
\end{proof}

We are now ready to give the proof of Lemma \ref{GrUB}.

\begin{proof}[Proof of Lemma \ref{GrUB}]
    Throughout we take $B$ and $L$ to be constants depending on
    dimension alone (so we can absorb them into a generic
    $C_d$). Lemma \ref{GrUB} has two parts: \eqref{GrUB.eq1} and
    \eqref{GrUB.eq2}. We prove \eqref{GrUB.eq2} first and then
    indicate the changes necessary for  \eqref{GrUB.eq1}. So assume
    first that $\Omega$ is a polytope with facets bounded by a
    constant depending on $d$ alone and take $\F = \C^B(\Omega)$. 

With $\ball_{\disdes}^{\F}(f, t) := \left\{g \in \F : 
  \ell_{\disdes}(f, g) \leq t \right\}$ and $\ball_{\empdes}^{\F}(f,
t) := \left\{g \in \F : \ell_{\empdes}(f, g) \leq t \right\}$, Lemma 
\ref{TwoLosses} gives  
\begin{equation*}
\ball^{\boucon}_{\empdes}(\tilde{f}_k, t) \subseteq
  \ball^{\boucon}_{\disdes}(\tilde{f}_k, 2t + C_d n^{-2/(d+4)})
\end{equation*}
with probability at least
\begin{equation}\label{pro1}
  1 - C \exp \left(-\frac{n^{d/(d+4)}}{C_d} \right).
\end{equation}
Thus for
\begin{equation}\label{tconr.1}
t \geq C_d n^{-2/(d+4)}, 
\end{equation}
we get
\begin{equation*}
  \ball^{\boucon}_{\empdes}(\tilde{f}_k, t) \subseteq
  \ball^{\boucon}_{\disdes}(\tilde{f}_k, 3t),
\end{equation*}
and consequently  $G_{\tilde{f}_k}(t, \C^B(\Omega)) \leq
\mathfrak{G}_{\tilde{f}_k}(3t)$ where 
\begin{equation*}
 \mathfrak{G}_{\tilde{f}_k}(t) := \E \left[ \sup_{g \in
       \ball^{\boucon}_{\disdes}(\tilde{f}_k, t)} \frac{1}{n}
  \sum_{i=1}^n \xi_i \left(g(X_i) - \tilde{f}_k(X_i) \right) \bigg| X_1, \dots,
  X_n \right]
\end{equation*}
holds with probability at least \eqref{pro1}. By concentration of
measure, the above conditional expectation will be close to the
corresponding unconditional expectation because
\begin{align*}
 & T(x_1, \dots, x_n) \\ &:= \E \left[ \sup_{g \in
       \ball^{\boucon}_{\disdes}(\tilde{f}_k, 3t)} \frac{1}{n}
  \sum_{i=1}^n \xi_i \left(g(X_i) - \tilde{f}_k(X_i) \right) \bigg|
  X_1 = x_1, \dots,
  X_n  = x_n\right]
\end{align*}
satisfies the bounded differences condition:
\begin{align*}
 \left| T(x_1, \dots, x_n) - T(x_1', \dots, x_n') \right| \leq
  \frac{2B \sigma}{n} \sum_{i=1}^n I \{x_i \neq x_i'\}
\end{align*}
and the bounded differences concentration inequality consequently
gives
\begin{align}\label{bdc1}
  \P\left\{\mathfrak{G}_{\tilde{f}_k}(3t) \leq \E \mathfrak{G}_{\tilde{f}_k}(3t) + x
  \right\} \geq 1 - \exp \left(\frac{-nx^2}{2B^2 \sigma^2} \right)
\end{align}
for every $x > 0$. We next control
\begin{equation*}
\E \mathfrak{G}_{\tilde{f}_k}(3t) = \E \left[ \sup_{g \in
       \ball^{\boucon}_{\disdes}(\tilde{f}_k, 3t)} \frac{1}{n}
  \sum_{i=1}^n \xi_i \left(g(X_i) - \tilde{f}_k(X_i) \right) \right]
\end{equation*}
where the expectation on the left hand side is with respect to $X_1,
\dots, X_n$  while the expectation on the right hand side is with
respect to all variables $\xi_1, \dots, \xi_n, X_1, \dots, X_n$. Clearly
\begin{equation}\label{cov1}
  \E \mathfrak{G}_{\tilde{f}_k}(3t) = \E \sup_{h \in \mathcal{H}} \left(\mathbb{Q}_n h
  -\mathbb{Q} h \right)
\end{equation}
where $\mathcal{H}$ consists of all functions of the form $(\xi, x)
\mapsto \xi \left(g(x) - \tilde{f}_k(x) \right)$ as $g$ varies over
$\ball^{\boucon}_{\disdes}(\tilde{f}_k, 3t)$, $\mathbb{Q}_n$
is the empirical measure corresponding to $(\xi_i, X_i), i = 1, \dots,
n$, and $\mathbb{Q}$ is the distribution of $(\xi, X)$ where $\xi$ and
$X$ are independent with $\xi \sim N(0, \sigma^2)$ and $X \sim \P$.

We now use the bound \eqref{expbo} which requires us to control
$N_{[\,]}(\epsilon, \mathcal{H}, L_2(\Q))$. This is done by Theorem
\ref{Thm:Bddcube} which states that
\begin{equation}\label{metbob}
\log N_{[\,]}(\epsilon, \ball^{\boucon}_{\disdes}(\tilde{f}_k, t), L_2(\P))
\leq C_d k \left(\log \frac{C_d B}{\epsilon} \right)^{d+1}
\left(\frac{t}{\epsilon} \right)^{d/2}.
\end{equation}
Theorem \ref{Thm:Bddcube} is stated under the unnormalized
integral constraint $\int_{\Omega} (f - \tilde{f}_k)^2 \leq t^2$ and
for bracketing numbers under the unnormalized Lebesgue measure but
this implies \eqref{metbob}  as the volume of $\Omega$ is assumed to
be bounded on both sides by dimensional constants. We now claim that
\begin{equation}\label{qpbra}
  N_{[\,]}(\epsilon, \mathcal{H}, L_2(\Q)) \leq N_{[\,]}(\epsilon \sigma^{-1},
  \ball^{\boucon}_{\disdes}(\tilde{f}_k, 3t), L_2(\P)).
\end{equation}
Inequality \eqref{qpbra} is true because of the following. Let $\{[g_L,
g_U], g \in G\}$ be a set of covering brackets for the set
$\ball^{\boucon}_{\disdes}(\tilde{f}_k, 3t)$. For each bracket $[g_L,
g_U]$, we associate a corresponding bracket $[h_L, h_U]$ for
$\mathcal{H}$ as follows:
\begin{equation*}
  h_L(\xi, x) := \xi \left(g_L(x) - \tilde{f}_k(x) \right) I\{\xi \geq
  0\} + \xi \left(g_U(x) - \tilde{f}_k(x) \right) I\{\xi < 0\}
\end{equation*}
and
\begin{equation*}
  h_U(\xi, x) := \xi \left(g_U(x) - \tilde{f}_k(x)\right) I\{\xi \geq
  0\} + \xi \left(g_L(x) - \tilde{f}_k(x) \right)I\{\xi < 0\}.
\end{equation*}
It is now easy to check that whenever $g_L \leq g \leq g_U$, we have
$h_L \leq h_g \leq h_U$ where $h_g(\xi, x) = \xi \left(g(x) -
  \tilde{f}_k(x) \right)$. Further, $h_U - h_L = |\xi| \left(g_U - g_L
\right)$ and thus $\Q \left(h_U - h_L \right)^2 = \sigma^2 \P
\left(g_U - g_L \right)^2$ which proves \eqref{qpbra}. Inequality
\eqref{metbob} then gives that for every $a \geq B/\sqrt{n}$, we have
\begin{align*}
& \int_a^{B} \sqrt{\log N_{[\,]}(u, \mathcal{H}, L_2(\Q))} du \\ &\leq  C_d \sqrt{k}
                                                         \int_a^B
                                                         \left(\log
                                                         \frac{C_d B
                                                         \sigma}{u} \right)^{(d+1)/2}
                                                         \left(\frac{t
                                                         \sigma}{u}
                                                         \right)^{d/4}
                                                         du \\
  &\leq C_d \sqrt{k} (t \sigma)^{d/4} \left(\log
                                                         \frac{C_d B
                                                         \sigma}{a}
    \right)^{(d+1)/2} \int_a^{\infty} u^{-d/4} du \\
  &\leq C_d \sqrt{k} (t \sigma)^{d/4} \left(\log
                                                         \frac{C_d B
                                                         \sigma}{a}
    \right)^{(d+1)/2} a^{1-(d/4)} \\
  &\leq C_d \sqrt{k} (t \sigma)^{d/4} \left(\log (C_d \sigma
    \sqrt{n})\right)^{(d+1)/2} a^{1-(d/4)}
\end{align*}
where, in the last inequality, we used $a \geq B/\sqrt{n}$. The
inequality $$a \geq Cn^{-1/2} \int_a^B \sqrt{\log N_{[\,]}(u,
  \mathcal{H}, L_2(\Q))} du$$ will therefore be satisfied for
\begin{align*}
  a \geq C_d t \sigma \left(\frac{k}{n} \right)^{2/d} \left(\log (C_d \sigma
    \sqrt{n})\right)^{2(d+1)/d}
\end{align*}
for an appropriate constant $C_d$. The bound \eqref{expbo} then gives
\begin{align*}
   \E \mathfrak{G}_{\tilde{f}_k}(3t) \leq C_d t \sigma
  \left(\frac{k}{n} \right)^{2/d} \left(\log (C_d \sigma
  \sqrt{n})\right)^{2(d+1)/d}  + \frac{C_d}{\sqrt{n}}.
\end{align*}
Combining the above steps, we deduce that for every $x > 0$,
the inequality
\begin{align*}
  G_{\tilde{f}_k}(t, \C^B(\Omega)) \leq                                            C_d t \sigma
                                                 \left(\frac{k}{n}
                                                 \right)^{2/d}
                                                 \left(\log (C_d
                                                 \sigma
                                                 \sqrt{n})\right)^{2(d+1)/d}
                                                 +
                                                 \frac{C_d}{\sqrt{n}}
                                                 + x 
\end{align*}
holds with probability at least
\begin{align*}
  1 - C \exp \left(-\frac{n^{d/(d+4)}}{C_d} \right) - \exp
  \left(\frac{-nx^2}{C_d \sigma^2} \right)
\end{align*}
for every fixed $t$ satisfying \eqref{tconr.1}. Inequality \eqref{GrUB.eq2} is now deduced by taking
\begin{equation*}
  x = C_d t \sigma \left(\frac{k}{n} \right)^{2/d} \left(\log (C_d \sigma
    \sqrt{n})\right)^{2(d+1)/d}
\end{equation*}

Let us now get to \eqref{GrUB.eq1}. Here $\Omega$ is not necessarily a
polytope and $\F$ is either $\C_L^L(\Omega)$ or $\C_L(\Omega)$. Let us
first argue that it is enough to prove \eqref{GrUB.eq1} when $\F =
\C_L^{4L}(\Omega)$. This will obviously imply the same inequality for
the smaller class $\C_L^L(\Omega)$. It also implies the same
inequality for the larger class $\C_L(\Omega)$ because of the
following claim:
\begin{align}\label{lipbo}
  f \in \C_L(\Omega), f \notin \C_L^{4L}(\Omega) \implies \min
  \left(\ell_{\P_n}(f, \tilde{f}_k), \ell_{\P}(f, \tilde{f}_k)
  \right) > L.
\end{align}
The above claim immediately implies that
\begin{equation*}
  \ball_{\P_n}^{\C_{L}(\Omega)}(\tilde{f}_k, t) =
  \ball_{\P_n}^{\C^{4L}_{L}(\Omega)}(\tilde{f}_k, t) \qt{for all $t \leq
    L$}.
\end{equation*}
which leads to
\begin{equation*}
  G_{\tilde{f}_k}(t, \C_L^{4L}(\Omega)) =   G_{\tilde{f}_k}(t,
  \C_L(\Omega))  \qt{for all $t \le L$}. 
\end{equation*}
To see \eqref{lipbo}, note that assumptions $f \in \C_L(\Omega)$ and 
$f \notin \C_L^{4L}(\Omega)$ together imply that $|f(x)| > 4L$ for some
$x \in \Omega$. By the Lipschitz property of $f$, the fact that
$\Omega$ has diameter $\leq 2$ and the fact that $\tilde{f}_k$ is
bounded by $L$, we have
\begin{align*}
  |f(y) - \tilde{f}_k(y)| &= |f(y) - f(x) + f(x) - \tilde{f}_k(y)| \\
                          &\geq |f(x)| - |f(y) - f(x)| - |\tilde{f}_k(y)| \\
  &> 4 L - L \|x - y\| - L \geq L
\end{align*}
for every $y \in \Omega$. This clearly implies that both
$\ell_{\P_n}(f, \tilde{f}_k)$ and $\ell_{\P}(f, \tilde{f}_k)$  are
larger than $L$ which proves \eqref{lipbo}.

In the rest of the proof, we therefore assume that $\F =
\C_L^{4L}(\Omega)$. We write
\begin{equation*}
  G_{\tilde{f}_k}(t, \F) \leq G^{I}_{\tilde{f}_k}(t, \F)  +
  G^{II}_{\tilde{f}_k}(t, \F)
\end{equation*}
where 
\begin{equation*}
     G^I_{\tilde{f}_k}(t, \F) := \E \left[ \sup_{g \in
         \ball^{\F}_{\empdes}(\tilde{f}_k, t)}
       \frac{1}{n} 
  \sum_{i: X_i \in \cup_{i=1}^m \Delta_i}  \xi_i \left(g(X_i) -
    \tilde{f}_k(X_i) \right) \bigg| X_1, \dots, X_n \right]
\end{equation*}
and
\begin{equation*}
  G^{II}_{\tilde{f}_k}(t, \F) := \E \left[ \sup_{g \in
       \ball^{\F}_{\empdes}(\tilde{f}_k, t)} \frac{1}{n}
  \sum_{i : X_i \notin \cup_{i=1}^m \Delta_i} \xi_i \left(g(X_i) -
    \tilde{f}_k(X_i) \right) \bigg| X_1, \dots,
  X_n \right].
\end{equation*}
Here $\Delta_1, \dots, \Delta_m$ are the $d$-simplices given by Lemma
\ref{ballapp}. We now bound $G^{I}_{\tilde{f}_k}(t, \F)$ and
$G^{II}_{\tilde{f}_k}(t, \F)$ separately. The first term
$G^{I}_{\tilde{f}_k}(t, \F)$ can be shown to satisfy the same bound in
\eqref{GrUB.eq1} using almost the same argument used for
\eqref{GrUB.eq1}. The only difference is that, instead of
\eqref{metbob}, we use: 
\begin{align*}
& N_{[\,]}(\epsilon,
\left\{x\mapsto g(x)I\{x \in \cup_{i=1}^m \Delta_i\} : g \in
  \ball^{\C_L^{4L}(\Omega)}_{\disdes}(\tilde{f}_k, 3t) \right\}, L_2(\P))
\\ &\leq C_d k \left(\log \frac{C_d L}{\epsilon} \right)^{d+1}
\left(\frac{t}{\epsilon} \right)^{d/2}, 
\end{align*}
the above bound also following from Theorem \ref{Thm:Bddcube}. For
$G^{II}_{\tilde{f}_k}(t)$, let $\tilde{n} := \sum_{i=1}^n I\{X_i
\notin \cup_{i=1}^m \Delta_i\}$ and use Dudley's bound (Theorem
\ref{dud}) and \eqref{bronweb} to write 
\begin{align*}
  G^{II}_{\tilde{f}_k}(t, \F) &= \frac{\tilde n}{n}\E \left[ \sup_{g \in
       \ball^{\F}_{\empdes}(\tilde{f}_k, t)}
                            \frac{1}{\tilde n}
  \sum_{i : X_i \notin \cup_{i=1}^m \Delta_i} \xi_i \left(g(X_i) - \tilde{f}_k(X_i) \right) \bigg| X_1, \dots,
                            X_n \right] \\
  &\leq \frac{\sigma \tilde{n}}{n} \inf_{\delta > 0}
    \left(\frac{12}{\sqrt{\tilde{n}}} \int_{\delta}^{\infty}
    \sqrt{\log N(\epsilon, \C_L^{4L}(\Omega), \ell_{\infty})}
    d\epsilon + 2 \delta \right) \\
  &\leq C_d \frac{\sigma \tilde{n}}{n} \inf_{\delta > 0}
    \left(\frac{1}{\sqrt{\tilde{n}}} \int_{\delta}^{\infty}
    \left(\frac{L}{\epsilon} \right)^{d/4}
    d\epsilon +  \delta \right).
\end{align*}
The choice $\delta = L (\tilde n)^{-2/d}$ leads to
\begin{align}\label{gf2}
  G^{II}_{\tilde{f}_k}(t, \F) \leq C_d L \frac{\sigma}{n} (\tilde
  n)^{1-(2/d)}.
\end{align}
Note that $\tilde{n}$ is binomially distributed with parameters
$n$ and $\tilde{p} := \text{Vol}(\Omega \setminus
(\cup_{i=1}^m \Delta_i))/\text{Vol}(\Omega)$. Because $(1 - C_d
k^{-2/d}) \Omega \subseteq
\cup_{i=1}^m \Delta_i \subseteq \Omega$ (from Lemma \ref{ballapp}), we have
\begin{align*}
  \tilde{p} \leq 1 - (1 - C_d k^{-1/d})^d \leq d C_d k^{-1/d}
\end{align*}
where we used $(1- u)^d \geq 1- du$. Hoeffding's inequality:
\begin{align*}
  \P \left\{\text{Bin}(n, p) \le np + u \right\} \geq 1 - \exp
  \left(\frac{-u^2}{2n} \right) \qt{for every $u \geq 0$}
\end{align*}
gives (below $C_d$ is such that $\tilde p \leq C_d k^{-1/d}$)
\begin{align*}
  \P \left\{\tilde n \leq 2 C_d n k^{-1/d} \right\} \geq \P
  \left\{\tilde{n} - n \tilde{p} \leq C_d n k^{-1/d} \right\} \geq 1 -
  \exp \left(-\frac{C_d^2}{2} n k^{-2/d} \right).
\end{align*}
Combining the above with \eqref{gf2}, we obtain
\begin{align*}
  G^{II}_{\tilde{f}_k}(t, \F) \leq C_d L \sigma n^{-2/d} k^{-1/d} k^{2/d^2}
\end{align*}
with probability at least $1 - \exp(-C_dn k^{-2/d})$. The proof of
\eqref{GrUB.eq1} is now completed by combining the obtained bounds for
$G^{I}_{\tilde{f}_k}(t, \F)$ and $G^{II}_{\tilde{f}_k}(t, \F)$ 
\end{proof}

\subsection{Proof of Lemma \ref{GrLB}}
Lemma \ref{GrLB} is proved below using Lemma \ref{lbf0.ran} (proved in
Subsection \ref{adproofs}) and the
following standard result (Sudakov Minoration) from the theory of
Gaussian processes (see e.g., \cite[Theorem
3.18]{LedouxTalagrand91book}):

\begin{lemma}[Sudakov minoration]\label{sud}
  In our random design setting (with $\xi_1, \dots, \xi_n
  \overset{\text{i.i.d}}{\sim} N(0, \sigma^2)$), the following is true
  for every class $\F$ and $t \geq 0$:  
   \begin{equation}\label{sudr}
  \begin{split}
&   \E \left[ \sup_{\{g \in \F : \ell_{\P_n}(g, f) \leq t \}} \frac{1}{n} \sum_{i=1}^n \xi_i
      \left(g(X_i) - f(X_i) \right) \bigg| X_1, \dots, X_n \right] \\ &\geq  \frac{\beta
        \sigma}{\sqrt{n}} \sup_{\epsilon > 0} \left\{\epsilon
        \sqrt{\log N(\epsilon, \{g \in \F : \ell_{\P_n}(g, f) \leq t
                                         \}, \ell_{\P_n})}
      \right\},
      \end{split}
    \end{equation}
    where $\beta$ is a universal positive constant. 
\end{lemma}

\begin{proof}[Proof of Lemma \ref{GrLB}]
As in the proof of Lemma \ref{GrUB}, we use the notation:
$\B_{\P_n}^{\F}(f, t) = \{g \in \F : 
\ell_{\P_n}(g, f) \leq t\}$. By triangle inequality,   
\begin{equation*}
   \ball^{\F}_{\empdes}(\tilde{f}_k, t) \supseteq
   \ball^{\F}_{\empdes}(f_0, t/2)  \qt{whenever $\metemp(f_0,
     \tilde{f}_k) \leq t/2$}.
 \end{equation*}
From Lemma \ref{ballapp},  $\metemp(f_0, \tilde{f}_k) \leq \sup_{x \in
  \Omega} |f_0(x) - \tilde{f}_k(x)| \leq C_d k^{-2/d}$, and so
$\metemp(f_0,  \tilde{f}_k) \leq t/2$  will be satisfied whenever $t
\geq 2 C_d k^{-2/d}$. For such $t$, by Sudakov minoration (Lemma
\ref{sud}): 
\begin{align*}
G_{\tilde{f}_k}(t, \F) &\geq  \E \left[ \sup_{g \in
       \ball^{\F}_{\empdes}(f_0, t/2)} \frac{1}{n}
  \sum_{i=1}^n \xi_i \left(g(X_i) - \tilde{f}_k(X_i) \right) \bigg| X_1, \dots,
                     X_n \right]  \\
  &= \E \left[ \sup_{g \in
       \ball^{\F}_{\empdes}(f_0, t/2)} \frac{1}{n}
  \sum_{i=1}^n \xi_i \left(g(X_i) - f_0(X_i) \right) \bigg| X_1, \dots,
                     X_n \right] \\
&\geq \frac{\beta \sigma}{\sqrt{n}} \sup_{\epsilon
  > 0} \left\{\epsilon \sqrt{\log N(\epsilon,
                       \ball^{\F}_{\empdes}(f_0, t/2), \empdes)}
  \right\}. 
\end{align*}
We now use Lemma \ref{lbf0.ran} which applies to the function class 
$\C_L^L(\Omega)$ and, consequently, also to the larger classes
$\C_L(\Omega)$, $\C^B_L(\Omega)$ and $\C^B(\Omega)$ for $B \geq
L$. Applying Lemma \ref{lbf0.ran} with  $\epsilon = c_3 n^{-2/d}$, we
get that for each fixed $t \geq 16 c_3 n^{-2/d}$, the inequality
\begin{equation*}
  G_{\tilde{f}_k}(t, \F) \geq \beta \sqrt{c_1} \sigma c_3^{1-(d/4)} n^{-2/d}
\end{equation*}
holds with probability at least $1 - \exp(-c_2 n)$. This completes the
proof of Lemma \ref{GrLB} (note that, if the constant $C_d$ is
enlarged suitably, the earlier condition $t \geq 2 C_d k^{-2/d}$
implies $t \geq 16 c_3 n^{-2/d}$ because $k \leq n$). 
\end{proof}

\subsection{Proof of Lemma \ref{tfktildelb}} \label{tfloboran}
Lemma \ref{tfktildelb} is proved below using Lemmas \ref{GrUB} and
\ref{GrLB}.

\begin{proof}[Proof of Lemma \ref{tfktildelb}]
Taking $t = t_0 := \sqrt{c_d} n^{-1/d} \sqrt{\sigma}$ in Lemma
\ref{GrLB}, we obtain
\begin{equation}\label{flb}
  \P \left\{H_{\tilde{f}_k}(t_0, \F) \geq \frac{c_d}{2} \sigma
    n^{-2/d}  \right\} \geq 1 - \exp(-c_d n). 
\end{equation}
The condition $t_0 \geq C_d k^{-2/d}$ required for the application of
Lemma \ref{GrLB} places the following restriction on $k$:
\begin{equation}\label{resk1}
  k \geq \left(\frac{C_d}{\sqrt{c_d}} \right)^{d/2} \sigma^{-d/4}
  \sqrt{n}. 
\end{equation}
Now consider the upper bound on $H_{\tilde{f}_k}(t, \F) \leq
G_{\tilde{f}_k}(t, \F)$ given in Lemma \ref{GrUB}. The leading term in
this upper bound is the first term: 
\begin{equation}\label{ublead}
  C_d t \sigma \left(\frac{k}{n} \right)^{2/d} (\log(C_d \sigma
  \sqrt{n}))^{2(d+1)/d} 
\end{equation}
as long as
\begin{equation}\label{tlboo}
  t \geq \max \left(n^{(4-d)/(2d)} k^{-2/d} \sigma^{-1},
    k^{(2-3d)/d^2} \right). 
\end{equation}
We now choose $t$ so that \eqref{ublead} matches the lower bound on
$H_{\tilde{f}_k}(t_0, \F)$ from \eqref{flb}:
\begin{equation*}
  C_d t \sigma \left(\frac{k}{n} \right)^{2/d} (\log(C_d \sigma
  \sqrt{n}))^{2(d+1)/d} = \frac{c_d}{2} \sigma n^{-2/d}
\end{equation*}
leading to
\begin{equation*}
  t = t_1 := \frac{c_d}{2 C_d} k^{-2/d} (\log(C_d \sigma
  \sqrt{n}))^{-2(d+1)/d}.
\end{equation*}
Check that this choice of $t = t_1$ satisfies \eqref{tlboo} provided
\begin{equation}\label{resn}
  n \geq \left(\frac{c_d \sigma}{2C_d} \right)^{2d/(d-4)} (\log(C_d
  \sigma \sqrt{n}))^{4(d+1)/(d-4)}
\end{equation}
and
\begin{equation}\label{resk2}
  k \geq \left(\frac{2C_d}{c_d} \right)^{d^2/(d-2)} (\log(C_d \sigma
  \sqrt{n}))^{2d(d+1)/(d-2)}. 
\end{equation}
We have thus proved $H_{\tilde{f}_k}(t_1, \F) \leq H_{\tilde{f}_k}(t_0,
\F)$. Check that $t_1 < t_0$ provided 
\begin{equation}\label{resk3}
  k > (2 \sqrt{c_d} C_d)^{d/2}(\log(C_d \sigma
  \sqrt{n}))^{-(d+1)}  \sigma^{-d/4} \sqrt{n}. 
\end{equation}
As can now be directly checked, the choice $k = \gamma_d \sqrt{n}
\sigma^{-d/4}$ for an appropriate $\gamma_d$ and the condition $n \geq
N_{d, \sigma}$ for an appropriate $N_{d, \sigma}$ satisfy all the four
conditions \eqref{resk1}, \eqref{resk2}, \eqref{resk3} and
\eqref{resn}. Further the probability with which all the above bounds
hold is bounded from below by $1 - C_d
\exp(-n^{(d-4)/d})/C_d^2$. Inequality \eqref{tflb} in Theorem
\ref{chat} now implies that $t_{\tilde{f}_k}(\F) \geq t_1$. The
logarithmic term $\log(C_d \sigma \sqrt{n})$ in $t_1$ can be further
simplified to $\log n$ because of $n \geq N_{d, \sigma}$. This
completes the proof of Lemma \ref{tfktildelb}.  
\end{proof}

\subsection{Completion of the proof of Theorem
  \ref{implb.blse}}\label{comple.implb.blse} 
We now complete the proof of Theorem \ref{implb.blse} from Lemma
\ref{tfktildelb}.
\begin{proof}[Proof of Theorem \ref{implb.blse}]
We shall lower bound $\E_{\tilde{f}_k} \ell^2_{\P}(\hat{f}_n(\F),
\tilde{f}_k)$. Let $$\rho_n := c_d n^{-1/d} \sqrt{\sigma} (\log
    n)^{-2(d+1)/d}$$ be the lower bound on $t_{\tilde{f}_k}(\F)$ given
    by Lemma \ref{tfktildelb}. As a result,
\begin{align*}
&  \P \left\{ \ell^2_{\P_n}(\hat{f}_n(\F), \tilde{f}_k) \geq
    \frac{1}{2} \rho_n^2 \right\} \\ &\geq   \P \left\{
                                       \ell^2_{\P_n}(\hat{f}_n(\F),
                                       \tilde{f}_k) \geq 
                                       \frac{1}{2} t^2_{\tilde{f}_k}(\F), t_{\tilde{f}_k}(
                                       \F) \geq \rho_n
                                    \right\} \\
  &= \E \left[I \left\{t_{\tilde{f}_k}(\F) \geq \rho_n \right\} \P
    \left\{ \ell^2_{\P_n}(\hat{f}_n(\F), \tilde{f}_k) \geq 
    \frac{1}{2} t^2_{\tilde{f}_k}(\F) \bigg| X_1, \dots, X_n \right\}
    \right] \\
  &\geq \E \left[I \left\{t_{\tilde{f}_k}(\F) \geq \rho_n \right\} \left(
    1 - 6 \exp \left(\frac{-c n t^2_{\tilde{f}_k}(\F)}{\sigma^2} \right)
    \right) \right] \\
  &\geq \left(
    1 - 6 \exp \left(\frac{-c n \rho_n^2}{\sigma^2} \right) \right) \P
    \left\{t_{\tilde{f}_k}(\F) \geq \rho_n \right\}.
\end{align*}
where we used \eqref{charan} in the penultimate inequality. Lemma
\ref{tfktildelb} now gives
\begin{align*}
&   \P \left\{ \ell^2_{\P_n}(\hat{f}_n(\F), \tilde{f}_k) \geq
    \frac{1}{2} \rho_n^2 \right\} \\ &\geq \left(
    1 - 6 \exp \left(\frac{-c n \rho_n^2}{\sigma^2} \right) \right)
  \left(1 - C_d \exp
                                    \left(\frac{-n^{(d-4)/d}}{C_d^2} \right) \right) \\
  &\geq 1 - 6 \exp \left(\frac{-c n \rho_n^2}{\sigma^2} \right)  -  C_d \exp
                                    \left(\frac{-n^{(d-4)/d}}{C_d^2}
    \right).
\end{align*}
Clearly if $N_{d, \sigma}$ is chosen appropriately, then, for $n \geq
N_{d, \sigma}$,
\begin{equation*}
  \frac{n \rho_n^2}{\sigma^2} = \frac{c_d^2}{\sigma} n^{(d-2)/d} (\log
  n)^{-4(d+1)/d}
\end{equation*}
will be larger than any constant multiple of $n^{(d-4)/d}$ which gives
\begin{equation}\label{chap.b}
  \P \left\{ \ell^2_{\P_n}(\hat{f}_n(\F), \tilde{f}_k) \geq
    \frac{1}{2} \rho_n^2 \right\}  \geq 1 - C_d \exp
                                    \left(\frac{-n^{(d-4)/d}}{C_d^2}\right).
\end{equation}
We now argue that a similar inequality also holds for
$\ell^2_{\P}(\hat{f}_n(\F), \tilde{f}_k)$. Here it is easiest to break
the argument into the different choices of $\F$. First assume $\F =
\C^{B}(\Omega)$. Combining the above inequality with Lemma
\ref{TwoLosses}, we obtain
\begin{align*}
 & \P \left\{\ell_{\P}(\hat{f}_n(\boucon), \tilde{f}_k) \geq \frac{1}{2}
  \frac{\rho_n}{\sqrt{2}} - \frac{1}{2} C_d n^{-2/(d+4)} B^{d/(d+4)}
  \right\} \\ &\geq 1 - C_d \exp
                                    \left(\frac{-n^{(d-4)/d}}{C_d^2}\right)
                - C \exp \left(-\frac{n^{d/(d+4)}}{C_d B^{8/(d+4)}}
                \right) 
\end{align*}
Because $n^{-2/(d+4)}$ is of a smaller order than $n^{-1/d}$ and
$n^{d/(d+4)}$ is of a larger order than $n^{(d-4)/d}$ (and $B$ is a
dimensional constant), we obtain
\begin{equation*}
  \begin{split}
&  \P \left\{\ell_{\P}(\hat{f}_n(\boucon), \tilde{f}_k) \geq
  c_d n^{-1/d} \sqrt{\sigma} (\log n)^{-2(d+1)/d} \right\} \\ &\geq 1 - C_d \exp
                                    \left(\frac{-n^{(d-4)/d}}{C_d^2}\right).
\end{split}
\end{equation*}
provided $n \geq N_{d, \sigma}$ where $N_{d, \sigma}$ is a constant
depending on $d$ and $\sigma$ alone. Using this, we can further adjust
$N_{d, \sigma}$ so that 
\begin{align*}
  \P \left\{\ell_{\P}(\hat{f}_n(\boucon), \tilde{f}_k) \geq
  c_d n^{-1/d} \sqrt{\sigma} (\log n)^{-2(d+1)/d} \right\} \geq \frac{1}{2}
\end{align*}
which immediately gives that
\begin{align}\label{ana2}
  \E_{\tilde{f}_k} \ell^2_{\P}(\hat{f}_n(\boucon), \tilde{f}_k) \geq \frac{c_d^2}{4} \sigma n^{-2/d}  (\log
  n)^{-4(d+1)/d}
\end{align}
completing the proof of inequality \eqref{implb.blse.eq}. The same
argument also yields \eqref{implb.clse.eq} for $\F =
\C_L^B(\Omega)$. Now take $\F = \C_L(\Omega)$. Let us take $N_{d,
  \sigma}$ large enough 
so that $\rho_n \leq L$ for $n \geq N_{d, \sigma}$. The fact
\eqref{lipbo} then implies that both
$\ell_{\P_n}(\hat{f}(\C_L(\Omega)), \tilde{f}_k)$ and
$\ell_{\P}(\hat{f}(\C_L(\Omega)), \tilde{f}_k)$ trivially exceed
$\rho_n$ if $\hat{f}_n(\C_L(\Omega))$ does not belong to
$\C_L^{4L}(\Omega)$. Therefore while lower bounding these losses, we
can assume that  $\hat{f}_n(\C_L(\Omega)) \in \C_L^{4L}(\Omega)$. This
allows application of Lemma \ref{TwoLosses} again and yields
\eqref{implb.clse.eq}  for $\F = \C_{L}(\Omega)$. The proof of Theorem
\ref{implb.blse} is complete. 
\end{proof}

\subsection{Proof of Theorem \ref{Thm:Bddcube}}\label{Gao:subsection}
It is enough to prove Theorem
\ref{Thm:Bddcube} when $\Omega$ is exactly equal to $\cup_{i=1}^k
\Delta_i$. So in the proof of Theorem \ref{Thm:Bddcube}, we shall
assume that $\Omega = \cup_{i=1}^k \Delta_i$. The main step in  the
proof is to prove it 
in the special case when $\Omega$ is itself a $d$-simplex and
$\tilde{f}$ is identically zero on $\Omega$. This special case is
stated in the next result (note that simplices can be written in the
form \eqref{omsim} so the next result is applicable when $\Omega$ is a
simplex). 

\begin{lemma}\label{simplex}
  Let $\Omega$ be a convex body contained in the unit ball, and
  of the form 
  \begin{equation}\label{omsim} 
    \Omega=\{x\in \R^d: a_i\le v_i^Tx\le b_i, 1\le i\le d+1\}
  \end{equation}
where $v_i$ are fixed unit vectors. For a fixed $1 \leq p < \infty$
and $t > 0$, let
\begin{equation*}
  B_p^{\Gamma}(0, t, \Omega) := \left\{f \in \C^{\Gamma}(\Omega) :
    \int_{\Omega} |f(x)|^p dx \leq t^p \right\}. 
\end{equation*}
Then for every $0 < \epsilon < \Gamma$, we have
\begin{equation*}
  \log N_{[\,]}(\epsilon, B_p^{\Gamma}(0, t, \Omega), \|\cdot\|_{p,
    \Omega}) \leq C_{d, p}  \left(\log
    \frac{\Gamma}{\epsilon} \right)^{d+1} \left(\frac{t}{\epsilon}
  \right)^{d/2} 
\end{equation*}
\end{lemma}

Below, we first provide the proof of Theorem \ref{Thm:Bddcube}
using Lemma \ref{simplex}, and then provide the proof of Lemma
\ref{simplex}.

\begin{proof}[Proof of Theorem \ref{Thm:Bddcube}]
Without loss of generality, we assume $\Omega =\cup_{i=1}^k
\Delta_i$. For each $f\in B_p^\Gamma(\tilde f, t, \Omega)$ and $1 \leq 
i \leq k$, we define
$t_i(f)$ as the smallest positive integer $t_i$ such
that
\begin{equation}\label{tif1}
  \int_{\Delta_i}|f(x)-\tilde{f}(x)|^pdx\le t_i^pt^p|\Delta_i|
\end{equation}
where $|\Delta_i|$ denotes the volume of $\Delta_i$. Let
$\mathfrak{T}$ denote the collection of all sequences $T := (t_1(f),
\dots, t_k(f))$ as $f$ ranges over $B_p^{\Gamma}(\tilde f, t,
\Omega)$. Because $|f(x) -
f_0(x)|\le 2\Gamma$ for all $x \in \Omega$, we have 
$t_i(f)\le \lceil 2\Gamma/t \rceil$ (here $\lceil x \rceil$ is the
smallest positive integer larger than or equal to $x$). Thus, the
cardinality of $\mathfrak{T}$ is at most $\lceil2\Gamma/t
\rceil^k$. Further as $t_i(f)$ is the smallest positive integer $t_i$
satisying \eqref{tif1}, the inequality will be violated for $t_i - 1$
so that 
\begin{equation*}
(t_i - 1)^p t^p |\Delta_i| \leq  \int_{\Delta_i}|f(x)-\tilde{f}(x)|^pdx
\end{equation*}
for each $i = 1, \dots, k$. Summing these for $i = 1,
\dots, k$, we get
\begin{align*}
  \sum_{i=1}^k (t_i - 1)^p t^p |\Delta_i| &\leq
                                            \sum_{i=1}^k
                                            \int_{\Delta_i}|f(x)-f_0(x)|^pdx
                                            \\ &\le
                                            \int_{\Omega} |f(x) -
                                            f_0(x)|^p dx \leq
                                            t^p,
\end{align*}
which is equivalent to $\sum_{i=1}^k (t_i - 1)^p |\Delta_i| \leq
1$. As a result
\begin{equation}\label{sumti1}
\sum_{i=1}^m t_i^p|\Delta_i|\le 2^{p-1} \sum_{i=1}^m
[(t_i-1)^p+1]|\Delta_i|\le 2^p.  
\end{equation}
For each sequence $T=(t_1, t_2,
\ldots, t_k) \in \mathfrak{T}$, let $\F_T$ be the
collection of all functions $f \in B_p^\Gamma(\tilde f, t, \Omega)$
satisfying
\begin{equation*}
(t_i-1)^pt^p|\Delta_i|\le \int_{\Delta_i}|f(x)-f_0(x)|^pdx\le
  t_i^pt^p|\Delta_i| 
\end{equation*}
for each $i = 1, \dots, k$. For $f \in \F_T$ and $1 \leq i
\leq k$, the restriction of $f- \tilde f$ to $\Delta_i$ belongs to $ B_p^{2
\Gamma}(0, t_it|\Delta_i|^{1/p}, \Delta_i)$ (since $\tilde{f}$ is
linear on each $\Delta_i$). Applying Lemma~\ref{simplex}, we deduce,
for every $0 < \epsilon_i < 2\Gamma$, the existence of a set $\cG_i$ consisting
of no more than $$\exp(C_{d, p}[\log(2\Gamma/\epsilon_i)]^{d+1} 
t_i^{d/2}|\Delta_i|^{d/2p} t^{d/2}\epsilon_i^{-d/2})$$ brackets, such that
for each $f\in \F_T$, there exists a bracket $[g_i, h_i]\in \cG_i$
such that $g_i(x)+f_0(x)\le f(x)\le h_i(x)+f_0(x)$ for all $x\in
\Delta_i$, and $$\int_{\Delta_i}|h_i(x)-g_i(x)|^pdx\le \epsilon_i^p.$$
We now define a bracket $[g, h]$ on $\Omega$ via $g(x)=g_i(x)$ and
$h(x)=h_i(x)$ for $x\in \Delta_i$, $1\le i\le k$. Then,
we clearly have $g(x)\le f(x)\le h(x)$ for all $x\in \Omega$, and
$$\int_{\Omega}|h(x)-g(x)|^pdx=\sum_{i=1}^k
\int_{\Delta_i}|h_i(x)-g_i(x)|^pdx\le
\sum_{i=1}^k\epsilon_i^p.$$ 
We choose
$$\epsilon_i=\max \left(2^{-1-2/p}t_i|\Delta_i|^{1/p} \epsilon,(4k)^{-1/p}
\epsilon \right),$$
so that 
 $$\sum_{i=1}^k\epsilon_i^p\le
 \left(\frac{\epsilon^p}{2^{p+2}}\sum_{i=1}^k
   t_i^p|\Delta_i|+\frac14\epsilon^p\right)\le \frac{\epsilon^p}{4},$$
 where we used \eqref{sumti1}. Thus, $[g,h]$ is an $\epsilon$-bracket.

Note that for each fixed $T$, the total number of brackets $[g, h]$ is at most
\begin{align*}
N:=&\prod_{i=1}^k\exp\left(C_{d, p}[\log(\Gamma/\epsilon_i)]^{d+1}
     t_i^{d/2}|\Delta_i|^{d/2p} t^{d/2}\epsilon_i^{-d/2}\right)\\ 
\le & \prod_{i=1}^k\exp\left(C_{d, p}\left[\frac1p\log (4k)+\log
      \Gamma+\log(1/\epsilon) \right]^{d+1}
      \left(\frac{2^{1+2/p}t}{\epsilon}\right)^{d/2}\right)\\ 
\le &\exp\left(C'_{d, p} k[\log k+\log
      \Gamma+\log(1/\epsilon)]^{d+1}\left( \frac{t}{\epsilon} \right)^{d/2}\right)
\end{align*}
Combining with the number of choices $\leq \lceil 2\Gamma/t \rceil^k$
of the sequences $T = (t_1, \dots, t_k)$, the number of realizations
of the brackets $[g,h]$ is at most $$\lceil 2\Gamma/t \rceil^k\cdot
N\le \exp\left(C''_{d, p}k[\log k+\log
  \Gamma+\log(1/\epsilon)]^{d+1}(t/\epsilon)^{d/2}\right),$$ which can
be shown to yield inequality \eqref{Thm:Bddcube.eq}.  
\end{proof}

We now prove Lemma \ref{simplex}. The main ingredient in this proof is
the result below whose proof directly follows from arguments in
\cite{gao2017entropy}. 

\begin{lemma}\label{Corr:shrink}
  Suppose $\Omega$ has volume at most 1 and is of the form:
 \begin{equation*}
    \Omega=\{x\in \R^d: a_i\le v_i^Tx\le b_i, 1\le i\le d+1\}
  \end{equation*}
  where $v_i$ are fixed unit vectors. For a fixed $0 < \eta < 1/5$,
  let
  \begin{equation*}
    \Omega_0:=\{x\in \R^d: a_i+\eta(b_i-a_i)\le v_i^Tx\le
    b_i-\eta(b_i-a_i), 1\le i\le d+1\}. 
  \end{equation*}
 For $1 \leq p < \infty$ and $t >
  0$, let 
    \begin{equation*}
      C_p(\Omega, t) := \left\{f \in \C(\Omega) : \int_{\Omega}
        |f(x)|^p dx \leq t^p  \right\}. 
    \end{equation*}
    Then for every $\epsilon > 0$, we have
    \begin{equation*}
      \log N_{[\,]}(\epsilon, C_{p}(\Omega, t), \|\cdot\|_{p,
        \Omega_0}) \leq C_{d, p, \eta} \left(\frac{t}{\epsilon}
      \right)^{d/2}.  
    \end{equation*}  
  \end{lemma}
The main idea behind Lemma \ref{Corr:shrink} is that on $\Omega_0$,
each function in $C_p(\Omega, t)$ is uniformly bounded by a
constant multiple of $t$. The proof of Lemma \ref{Corr:shrink} follows
then from the bracketing entropy result for uniformly bounded convex
functions (details can be found in \cite{gao2017entropy}). 

In the proof of Lemma \ref{simplex}, we use the following
notation. For $\Omega := \{x\in \R^d: a_i\le 
v_i^Tx\le b_i, 1\le i\le d+1\}$, where $v_i$ are fixed unit vectors
and $0 \leq r \leq d+1$, let
\begin{equation*}
  \begin{split}
T_r(\Omega)&=\left\{x\in \Omega : a_{i}\le v_i^T x\le b_i \ {\rm
  for  }\  1 \le i\le r, {\rm and } \right. \\ & \left. a_j+\eta (b_j-a_j) \le v_j^T x\le b_j-\eta(b_j-a_j) \
{\rm   for } \ r<j\le d+1\right\}.    
  \end{split}
\end{equation*}
Observe that for $r = 0$, the set $T_0(\Omega)$
coincides with $\Omega_0$ in Lemma \ref{Corr:shrink}. Further
$T_{d+1}(\Omega) = \Omega$. 

\begin{proof}[Proof of Lemma \ref{simplex}]
Fix $0<\eta<1/5$. We shall prove the following by induction on $r$:
There exist two constants $C_1(d,p)$ and $C_2(d,p)$ such that
\begin{equation}\label{indsim}
\log N_{[\,]}(\epsilon, B^{\Gamma}_p(0, t, \Omega),
        \|\cdot\|_{p, T_r(\Omega)})\le C_1[C_2\log(\Gamma/\epsilon)]^r
        t^{d/2}\eps^{-d/2}  
\end{equation}
for every $r=0, 1,\ldots,d+1$, and for every $\Omega$ of the form
\eqref{omsim}. As mentioned above, for $r = 0$, the 
set $T_0(\Omega)$ equals the set $\Omega_0$ in Lemma \ref{Corr:shrink}
and, as a result, \eqref{indsim} for $r = 0$ follows directly from
Lemma \ref{Corr:shrink}. Let us now assume that \eqref{indsim} is true
for $r=k-1$ and proceed to prove it for $r = k$. Define
	$$K_0=T_{k-1}(\Omega)=\{x\in T_k(\Omega): a_k+\eta(b_k-a_k)\le
        v_k^Tx\le b_k-\eta(b_k-a_k)\}.$$
For a positive integer $m$ (to be determined later) and for
$s=0, 1,2,\ldots, m$, define 
	$$K_{2s+1}=\{x\in T_k(\Omega): a_k+2^{-s-1}\eta(b_k-a_k)\le
        v_k^Tx< a_k+ 2^{-s}\eta(b_k-a_k)\}, $$
	$$K_{2s+2}=\{x\in T_k(\Omega): b_k-2^{-s}\eta(b_k-a_k)<
        v_k^Tx\le b_k-2^{-s-1}\eta(b_k-a_k)\}, $$
	Furthermore, define
	$$K_L=\{x\in T_k(\Omega): a_k\le  v_k^Tx<
        a_k+2^{-m-1}\eta(b_k-a_k)\}, $$
	$$K_R=\{x\in T_k(\Omega): b_k-2^{-m-1}\eta(b_k-a_k)< v_k^Tx\le
        b_k\}. $$
        Then, the sets $K_1, K_2, K_3, K_4, \dots, K_{2m+1}, K_{2m+2}$ together
        with $K_0, K_L, K_R$ form  a partition of $T_k(\Omega)$. We
        now aim to use the induction step. For this purpose we denote
        the inflated sets
	$$\widehat{K}_{2s+1}=\{x\in \Omega:
        a_k+2^{-s-2}\eta(b_k-a_k)\le v_k^Tx< a_k+3\cdot
        2^{-s-1}\eta(b_k-a_k)\}, $$
	$$\widehat{K}_{2s+2}=\{x\in \Omega: b_k-3\cdot
        2^{-s-1}\eta(b_k-a_k)< v_k^Tx\le
        b_k-2^{-s-2}\eta(b_k-a_k)\}. $$
The key observation now is that $T_{k-1}(\widehat{K}_{2s+1}) \supset
K_{2s+1}$ and $T_{k-1}(\widehat{K}_{2s+2})\supset K_{2s+2}$. To prove
$T_{k-1}(\widehat{K}_{2s+1}) \supset K_{2s+1}$, observe that
$T_{k-1}(\widehat{K}_{2s+1})$ equals  
	\begin{align*}
&\left\{x\in \widehat{K}_{2s+1}: a_i\le
                                      v_i^T x\le b_i \ {\rm  for  }\
                                      i\le k-1, \right.\\
	& \ \ \ \ \ \ a_j+\eta (b_j-a_j)\leq v_j^T x\le
   b_j-\eta(b_j-a_j) \ {\rm   for } \ k+1\le j\le d+1,\\
	& \ \ \ \ \ \  a_k+2^{-s-2}\eta(b_k-a_k)+\eta^2(5\cdot
   2^{-s-2} (b_k-a_k))\le v_k^Tx \\
	& \left.\ \ \ \ \ \ \ \ \ \ \ \ \ \ <a_k+3\cdot
   2^{-s-1}\eta(b_k-a_k)-\eta^2(5\cdot 2^{-s-2} (b_k-a_k))\right\}. 
	\end{align*}
As a result, $T_{k-1}(\widehat{K}_{2s+1}) \supset
K_{2s+1}$ if and only if
\begin{equation*}
  2^{-s-2} \eta(b_k - a_k) + \eta^2 5 \cdot 2^{-s-2} (b_k - a_k) \leq
  2^{-s-1} \eta (b_k - a_k)
\end{equation*}
which is equivalent to $\eta < 1/5$. The claim
$T_{k-1}(\widehat{K}_{2s+2})\supset K_{2s+2}$ follows similarly.

For every $f \in B^{\Gamma}_p(0, t, \Omega)$ and $1 \leq i \leq 2m+2$, 
let $t_i := t_i(f)$ be the smallest positive integer satisfying
\begin{equation}
  \label{tif2}
\int_{\widehat{K}_i}|f(x)|^pdx\le |\widehat{K}_i|t_i^pt^{p} 
\end{equation}
where $|\widehat{K}_i|$ denotes the volume of $\widehat{K}_i$. Because
$|f(x)| \leq \Gamma$ for all $x \in \Omega$, we have $t_i(f) \leq
\lceil \Gamma/t \rceil$. Let $\mathfrak{T}$ denote the collection of
all sequences $T := (t_1(f), \dots, t_{2m+2}(f))$ as $f$ ranges over
$B^{\Gamma}_p(0, t, \Omega)$. Because each $t_i(f) \leq \lceil
\Gamma/t \rceil$, the cardinality of $\mathfrak{T}$ is at most $\lceil
\Gamma/t \rceil^{2m+2}$. Further, because $t_i(f)$ is the smallest
        positive integer $t_i$ satisfying \eqref{tif2}, the inequality
will be violated for $t_i - 1$ so that 
\begin{equation*}
|\widehat{K}_i|(t_i-1)^pt^{p} \leq \int_{\widehat{K}_i}|f(x)|^pdx           
\end{equation*}
for each $i = 1, \dots, 2m+2$. Summing these over $i$, we obtain
\begin{align*}
\sum_{i=1}^{2m+2}  |\widehat{K}_i|(t_i-1)^pt^{p} \leq
  \sum_{i=1}^{2m+2}  \int_{\widehat{K}_i}|f(x)|^pdx         
\end{align*}
We now observe that every point in $\Omega$ is contained in
$\widehat{K}_i$ for at most three different $i$ so that
$\sum_{i=1}^{2m+2} I\{x \in \widehat{K}_i\} \leq 3I\{x \in
\Omega\}$. This gives
\begin{align*}
\sum_{i=1}^{2m+2}  |\widehat{K}_i|(t_i-1)^pt^{p} \leq
  \sum_{i=1}^{2m+2}  \int_{\widehat{K}_i}|f(x)|^pdx \leq 3
  \int_{\Omega} |f(x)|^p dx \leq 3 t^p. 
\end{align*}
In other words, we have proved that every sequence $(t_1, \dots,
t_{2m+2}) \in \mathfrak{T}$ satisfies:
\begin{align}\label{cond}
  \sum_{i=1}^{2m+2}  |\widehat{K}_i|(t_i-1)^p \leq 3. 
\end{align}
One consequence of the above is that
\begin{align}\label{oneconse}
  \sum_{i=1}^{2m+2} t_i^p |\widehat{K}_i| \leq 
        2^{p-1}\sum_{i=0}^{2m+2}[(t_i-1)^p+1]|\widehat{K}_i|\le 3\cdot
  2^p,
\end{align}
where we used convexity ($(a_i + b_i)^p \leq 2^{p-1} (a_i^p + b_i^p)$
with $a_i = t_i-1$ and $b_i = 1$) and the fact that
$\sum_{i=0}^{2m+2}|\widehat{K}_i| \leq 3|\Omega| \leq 3$.

For each $(t_1, t_2, \ldots,
t_{2m+2}) \in \mathfrak{T}$, let $\F_T$ be the class of all functions
$f \in B_p^{\Gamma}(0, t, \Omega)$ additionally satisfying
\begin{align}\label{ftom}
(t_i-1)^pt^p|\widehat{K}_i| \leq\int_{\widehat{K}_i}|f(x)|^pdx\le
        t_i^pt^p|\widehat{K}_i| 
\end{align}
for every $i = 1, \dots, 2m+2$. It is then clear that $B_p^{\Gamma}(0,
t, \Omega) \subseteq \cup_{T \in \mathfrak{T}} \F_T$ which implies
(below $|\mathfrak{T}|$ denotes the cardinality of the finite set
$\mathfrak{T}$) 
\begin{align}
 & \log N_{[\,]}(\epsilon, B_p^\Gamma(0, t, \Omega), \|\cdot\|_{p,
    T_k(\Omega)}) \nonumber \\ &\leq \log \sum_{T \in \mathfrak{T}}
                    N_{[\,]}(\epsilon, \F_T,
                    \|\cdot\|_{p, 
                    T_k(\Omega)}) \nonumber \\
  &\leq \log |\mathfrak{T}| + \max_{T \in \mathfrak{T}} \log
    N_{[\,]}(\epsilon, \F_T, \|\cdot\|_{p,
    T_k(\Omega)}) \nonumber \\
  &\leq (2m + 2) \log \lceil \Gamma/t \rceil + \max_{T \in \mathfrak{T}} \log
    N_{[\,]}(\epsilon, \F_T, \|\cdot\|_{p,
    T_k(\Omega)}). \label{axu1}
\end{align}
We shall now fix $T \in \mathfrak{T}$ and bound $\log
N_{[\,]}(\epsilon, \F_T, \|\cdot\|_{p, T_k(\Omega)})$.  Because $K_0$,
$K_1$, $\dots$, $K_{2m+2}$, $K_L$, $K_R$ form a partition of $T_k(\Omega)$, we
have
\begin{equation}\label{axu2}
  \begin{split}
 &\log N_{[\,]}(\epsilon, \F_T, \|\cdot\|_{p, T_k(\Omega)}) \\
  &\leq \log N_{[\,]}(\epsilon_0, \F_T, \|\cdot\|_{p, K_0}) +
  + \sum_{i=1}^{2m+2} \log N_{[\,]}(\epsilon_i, \F_T, \|\cdot\|_{p,
    K_i}) \\
  &+ \log N_{[\,]}(\epsilon_L, \F_T, \|\cdot\|_{p,
    K_L}) +  \log N_{[\,]}(\epsilon_R, \F_T, \|\cdot\|_{p,
    K_R})
    \end{split}
\end{equation}
provided $\epsilon_0, \epsilon_1, \dots, \epsilon_{2m+2}, \epsilon_L,
\epsilon_R > 0$ satisfy
\begin{align}\label{epscondp}
  \epsilon_0^p + \sum_{i=1}^{2m+2} \epsilon_i^p + \epsilon_L^p +
  \epsilon_R^p \leq \epsilon^p. 
\end{align}
To bound $\log N_{[\,]}(\epsilon_i, \F_T, \|\cdot\|_{p,
  K_i})$ for a fixed $1 \leq i \leq 2m+2$, note first that $K_i
\subseteq T_{k-1}(\widehat{K}_i)$ so that 
\begin{align*}
  \log N_{[\,]}(\epsilon_i, \F_T, \|\cdot\|_{p,
  K_i}) \leq   \log N_{[\,]}(\epsilon_i, \F_T, \|\cdot\|_{p,
  T_{k-1}(\widehat{K}_i)}). 
\end{align*}
The induction hypothesis will be used to control the right hand side
above.  Because of the right side inequality in \eqref{ftom}, the
restriction of each $f \in \F_T$ to the set $\widehat{K}_i$ 
belongs to $B_p^{\Gamma}(0, t_i t |\widehat{K}_i|^{1/p},
\widehat{K}_i)$
and so, by the induction hypothesis, we have 
\begin{align}
  &\log N_{[\,]}(\epsilon_i, \F_T, \|\cdot\|_{p,
  T_{k-1}(\widehat{K}_i)})  \nonumber\\ &\leq   \log N_{[\,]}(\epsilon_i, B_p^{\Gamma}(0, t_i t |\widehat{K}_i|^{1/p},
\widehat{K}_i) , \|\cdot\|_{p, 
                                T_{k-1}(\widehat{K}_i)}) \nonumber \\
  &\leq C_1 \left(C_2 \log \frac{\Gamma}{\epsilon_i} \right)^{k-1}
    |\widehat{K}_i|^{d/(2p)} t^{d/2} \left(\frac{t_i}{\epsilon_i}
    \right)^{d/2}. \label{crubo}
\end{align}
For $K_0 = T_{k-1}(\Omega)$, we use the induction hypothesis again to
obtain
\begin{align}
 & \log N_{[\,]}(\epsilon_0, \F_T, \|\cdot\|_{p, K_0}) \nonumber \\ &\leq  \log
                                                            N_{[\,]}(\epsilon_0,
                                                            B_p^{\Gamma}(0,
                                                            t, \Omega),
                                                            \|\cdot\|_{p,
                                                            T_{k-1}(\Omega)})
\nonumber  \\
  &\leq C_1 \left[C_2 \log \frac{\Gamma}{\epsilon_0} \right]^{k-1}
    \left(\frac{t}{\epsilon_0} \right)^{d/2}. \label{crubo1}
\end{align}
On the sets $K_L$ and $K_R$, we set $\log N_{[\,]}(\epsilon_L, \F_T, \|\cdot\|_{p, K_L})$
and  $\log N_{[\,]}(\epsilon_R, \F_T, \|\cdot\|_{p, K_R})$ to be equal
to zero, by taking the trivial bracket $[-\Gamma, \Gamma]$. For this
to be valid, $\epsilon_L$ and $\epsilon_R$ have to satisfy
\begin{align}\label{mchoi}
  \Gamma |K_L|^{1/p} \leq \epsilon_L ~~ \text{ and } ~~   \Gamma
  |K_R|^{1/p} \leq \epsilon_R. 
\end{align}
We now choose $m$, $\epsilon_i, 0 \leq i \leq 2m+2, \epsilon_L,
\epsilon_R$ so as to satisfy \eqref{epscondp}.  First of all, we take
\begin{align}\label{fcho}
  \epsilon_0 = \epsilon_L = \epsilon_R = \frac{\epsilon}{6^{1/p}}. 
\end{align}
To satisfy the conditions \eqref{mchoi}, note that $|K_L|$ and
$|K_R|$ are bounded by $C_d 2^{-m} \eta$ (note also that $\eta$ is a
constant). Thus \eqref{mchoi} can be ensured for $\epsilon_L =
\epsilon_R = \epsilon/6^{1/p}$ for $m \leq C_{d, p}
\log(\Gamma/\epsilon)$ for some $C_{d, p}$.

We also take
\begin{align}\label{schoi}
  \epsilon_i = \max \left(\frac{1}{2 \cdot 12^{1/p}} t_i
  |\widehat{K}_i|^{1/p} \epsilon, \frac{\epsilon}{4^{1/p}
  (2m+2)^{1/p}} \right).
\end{align}
It is then easy to check (using \eqref{oneconse}) that
\begin{align*}
  \sum_{i=1}^{2m+2} \epsilon_i^p &\leq \sum_{i=1}^{2m+2} \left(\frac{1}{2 \cdot 12^{1/p}} t_i
  |\widehat{K}_i|^{1/p} \epsilon \right)^p + \sum_{i=1}^{2m+2} \left(\frac{\epsilon}{4^{1/p}
  (2m+2)^{1/p}}  \right)^p \\ &\leq \frac{\epsilon^p}{4} +
  \frac{\epsilon^p}{4} = \frac{\epsilon^p}{2}. 
\end{align*}
Thus our choices satisfy \eqref{epscondp}. The proof of Lemma
\ref{simplex} is now completed by plugging these choices in
\eqref{crubo} and \eqref{crubo1}, and combining the resulting
bounds with \eqref{axu1} and \eqref{axu2}. 
\end{proof}

\subsection{Proofs of Lemma \ref{ballapp} and Lemma
  \ref{lbf0.ran}}\label{adproofs}

\begin{proof}[Proof of Lemma \ref{ballapp}]
   Let us first assume that $\Omega$ is a polytope (satisfying
   \eqref{OmegaAss}) whose number of facets is bounded by a constant
   depending on $d$ alone. For a fixed $\eta > 0$, let
   $\mathfrak{C}_{\eta}$ be the collection of all cubes of the form 
  \begin{equation}\label{allcubes}
    [k_1 \eta, (k_1 + 1) \eta] \times \dots \times [k_d \eta, (k_d +
    1) \eta]
  \end{equation}
  for $(k_1, \dots, k_d) \in \mathbb{Z}^d$ which intersect
  $\Omega$. Because $\Omega$ is contained in the unit ball, there
  exists a dimensional constant $c_d$ such that the cardinality of
  $\mathfrak{C}_{\eta}$ is at most $C_d \eta^{-d}$ for $\eta \leq
  c_d$.

   For each $B \in \mathfrak{C}_{\eta}$, the set $B \cap \Omega$
  is a polytope whose number of facets is bounded from above by a
  constant depending on $d$ alone. This polytope can therefore be
  triangulated into at most $C_d$ number of $d$-simplices. Let
  $\Delta_1, \dots, \Delta_m$ be the collection obtained by the taking
  the all of the aforementioned simplices as $B$ varies over
  $\mathfrak{C}_{\eta}$. These simplices clearly satisfy $\Omega =
  \cup_{i=1}^m \Delta_i$ and the second requirement of Lemma
  \ref{ballapp}. Moreover
  \begin{equation*}
    m \leq C_d \eta^{-d}
  \end{equation*}
  and the diameter of each simplex $\Delta_i$ is at most $C_d
  \eta$. Now define $\tilde{f}_{\eta}$ to be a piecewise affine convex
  function that agrees with $f_0(x) = \|x\|^2$ for each vertex of each
  simplex $\Delta_i$ and is defined by linearity everywhere else on
  $\Omega$. This function is clearly affine on each $\Delta_i$,
  belongs to $\C_{C_d}^{C_d}(\Omega)$ for a sufficiently large $C_d$
  and it satisfies
  \begin{equation*}
   \sup_{x \in \Delta_i} |f_0(x) - \tilde{f}_{\eta}(x)| \leq C_d
   \left(\text{diameter}(\Delta_i) \right)^2 \leq C_d \eta^2.
 \end{equation*}
 Now given $k \geq 1$, let $\eta = c_d k^{-1/d}$ for a sufficiently
 small dimensional constant $c_d$ and let $\tilde{f}_k$ to be the
 function $\tilde{f}_{\eta}$ for this $\eta$. The number of simplices
 is now $m \leq C_d k$ and
 \begin{equation*}
   \sup_{x \in \Delta_i} |f_0(x) - \tilde{f}_{\eta}(x)| \leq C_d k^{-2/d}
 \end{equation*}
 which completes the proof of Lemma \ref{ballapp} when $\Omega$ is a
 polytope. 

Now assume that $\Omega$ is a generic convex body (not necessarily a
polytope) satisfying \eqref{OmegaAss}. Here the only difference in the
proof is that we take $\mathfrak{D}_{\eta}$ be the collection
of all cubes of the form \eqref{allcubes} for $(k_1, \dots, k_d) \in
\mathbb{Z}^d$ which are contained in the 
  interior of $\Omega$.  Because each of these cubes has diameter
  $\eta \sqrt{d}$, it follows that
  \begin{equation*}
\left(1 - C_d \eta \right) \Omega \subseteq  \cup_{B \in
  \mathfrak{D}_{\eta}} B \subseteq \Omega.
\end{equation*}
The rest of the argument is the same as in the polytopal case. 
\end{proof}

\begin{proof}[Proof of Lemma \ref{lbf0.ran}]
By a standard argument involving local perturbations to the quadratic
function $f_0(x) = \|x\|^2$, we can prove the following for three
constants $c_1, c_2, C$ depending on $d$ alone: for every
$\epsilon \leq c_1$ and $L \geq C$, there exist an integer $N$ with 
\begin{equation*}
  c_2 \epsilon^{-d/2} \leq \log N \leq 2c_2 \epsilon^{-d/2}
\end{equation*}
and functions $f_1, \dots, f_N \in \C_L^L(\Omega)$  such that
\begin{equation*}
  \min_{1 \leq i \neq j \leq N} \ell_{\P}(f_i, f_j) \geq \sqrt{2}\epsilon ~~
  \text{ and } ~~ \max_{1 \leq i \leq N} \sup_{x \in \Omega} |f_i(x) -
  f_0(x)| \leq 4 \epsilon.
\end{equation*}
One explicit construction of such perturbed functions is given in the
proof of Lemma \ref{lbf0}. Lemma \ref{lbf0.ran} follows from the above
claim and the Hoeffding inequality. Indeed, Hoeffding's inequality
applied to the random 
variables $(f_j(X_i) - f_k(X_i))^2$ (which are bounded by $64
\epsilon^2$) followed by a union bound allows us to deduce that, for
every $t > 0$,  
\begin{equation*}
  \P \left\{\ell_{\P_n}^2(f_j, f_k) - \ell_{\P}^2(f_j, f_k) \geq -t n^{-1/2}
    \text{ for all } j , k \right\} \geq 1 - N^2 \exp
  \left(\frac{-t^2}{\Gamma \epsilon^4} \right).
\end{equation*}
for a universal constant $\Gamma$. Taking $t = \epsilon^2 \sqrt{n}$,
we get
\begin{align*}
  \P \left\{\ell_{\P_n}(f_j, f_k) \geq \epsilon   \text{ for all } j ,
    k \right\}  &\geq 1 - N^2 \exp \left(\frac{-n}{\Gamma} \right) \\ &\geq
  1 - \exp \left(4 c_2 \epsilon^{-d/2} - \frac{n}{\Gamma} \right).
\end{align*}
Assuming now that $\epsilon \geq n^{-2/d} (8 c_2 \Gamma)^{2/d}$, we
get 
\begin{equation*}
  \P \left\{\ell_{\P_n}(f_j, f_k) \geq \epsilon   \text{ for all } j ,
    k \right\}  \geq
  1 - \exp \left(- \frac{n}{2\Gamma} \right).
\end{equation*}
As each $f_j$ satisfies $\ell_{\P_n}(f_j, f_0) \leq \sup_{x} |f_j(x) -
f_0(x)| \leq 4\epsilon$, the above completes the proof of Lemma
\ref{lbf0.ran}. 
\end{proof}

\section{Proofs of fixed-design results for the unrestricted convex
  LSE} \label{fd.proofs}
This section contains the proofs of Theorem \ref{wcad}, Theorem
\ref{minsub}, Theorem \ref{adrate} and Theorem \ref{implb}. We follow
the sketch given in Subsection \ref{pfskefdr} and results quoted in
that subsection are also proved here. Subsection \ref{subs1} below
contains the proof of Theorem \ref{wcad} and that of the main
ingredient in its proof, Lemma \ref{GtUBwcFD}.  Subsection \ref{subs2}
contains the proof of Theorem \ref{adrate} and that of the main
ingredient in its proof, Lemma \ref{GtUBadFD}. Subsection
\ref{minsubproof} contains the proof of Theorem
\ref{minsub}. Subsection \ref{subs4} contains the proof of Theorem
\ref{implb}. Lemma \ref{allkt.lem} and Lemma \ref{lbf0}, which are both
crucial for the proof of Theorem \ref{implb}, are proved in Subsection
\ref{allktproof}. The metric entropy results are proved in Subsections
\ref{fdsimproofs} and \ref{discmentproofs}. Specifically, the main
metric entropy result (Theorem \ref{polytope}) is proved in Subsection
\ref{discmentproofs}. The two important corollaries of Theorem
\ref{polytope}: inequality \eqref{simmet} and inequality
\eqref{simmet1}, are proved in Subsection \ref{fdsimproofs}.

\subsection{Proof of Lemma \ref{GtUBwcFD} and Theorem \ref{wcad}}\label{subs1}
The proof of Lemma \ref{GtUBwcFD} makes crucial use of the metric entropy bound
\eqref{simmet}.

\begin{proof}[Proof of Lemma \ref{GtUBwcFD}]
  The metric entropy bound \eqref{simmet}, together with Dudley's
bound (Theorem \ref{dud}), give
\begin{align}
  &\frac{G_{f_0}(t, \C(\Omega))}{\sigma} \nonumber \\ &\leq \frac{C_d (\log n)^{F/2}}{\sqrt{n}}
  \int_{\theta}^{t/2} \left(\frac{t + \lin}{\epsilon} \right)^{d/4}
                        d\epsilon + 2 \theta \nonumber \\
  &\leq \frac{(C_d2^{d/4}) (\log n)^{F/2}}{\sqrt{n}}
 \left\{ \int_{\theta}^{t/2} \left(\frac{t}{\epsilon} \right)^{d/4}
    d\epsilon + \int_{\theta}^{t/2} \left(\frac{\lin}{\epsilon} \right)^{d/4}
    d\epsilon  \right\} + 2 \theta \label{dint}
\end{align}
for every $0 < \theta \leq t/2$. We replace $C_d 2^{d/4}$ by just
$C_d$ (in general, the value of $C_d$ can change from place to
place). We now split into the three 
cases $d \leq 3, d = 4$ and $d \geq 5$. For $d \leq 3$, we take
$\theta = 0$ to get
\begin{equation*}
  G_{f_0}(t, \C(\Omega)) \le \frac{C_d\sigma}{\sqrt{n}} (\log n)^{F/2} \left(t +
\lin^{d/4} t^{1-d/4} \right).
\end{equation*}
For $d = 4$, \eqref{dint} leads to
\begin{equation*}
  G_{f_0}(t, \C(\Omega)) \leq \frac{C_d \sigma}{\sqrt{n}} (\log n)^{F/2} (t + \lin) \log
  \frac{t}{2\theta} + 2 \sigma \theta.
\end{equation*}
Choosing $\theta = t/(2\sqrt{n})$, we obtain
\begin{equation*}
  G_{f_0}(t, \C(\Omega)) \leq \frac{C_d \sigma}{\sqrt{n}} \left(t + \lin \right) (\log n)^{1+(F/2)}
\end{equation*}
Finally, for $d \geq 5$, \eqref{dint} leads to the bound
\begin{align*}
  G_{f_0}(t, \C(\Omega)) &\leq \frac{C_d \sigma (\log n)^{F/2}}{\sqrt{n}}
 \left\{ \int_{\theta}^{\infty} \left(\frac{t}{\epsilon} \right)^{d/4}
    d\epsilon + \int_{\theta}^{\infty} \left(\frac{\lin}{\epsilon} \right)^{d/4}
         d\epsilon  \right\} + 2 \sigma \theta \\
  & \leq C_d \frac{\sigma (\log n)^{F/2}}{\sqrt{n}} (t + \lin)^{d/4}
  \theta^{1-(d/4)} + 2 \sigma \theta
\end{align*}
for every $\theta > 0$. The choice
\begin{equation*}
  \theta = \left(\frac{C_d (\log n)^{F/2}}{\sqrt{n}} \right)^{4/d} (t + \lin)
\end{equation*}
gives
\begin{equation*}
  G_{f_0}(t, \C(\Omega)) \le C_d\sigma \left(\frac{(\log n)^{F/2}}{\sqrt{n}}
  \right)^{4/d} \left(t + \lin \right). 
\end{equation*}
\end{proof}

Theorem \ref{wcad}, which follows from Lemma \ref{GtUBwcFD} and
Theorem \ref{chat}, is proved next.
\begin{proof}[Proof of Theorem \ref{wcad}]
  By Theorem \ref{chat} (specifically the upper bound in inequality
  \eqref{expecb}, and \eqref{tfub}), the risk is bounded from above by
  $2 t^2 + C\sigma^2/n$ for every $t$ satisfying $H_{f_0}(t,
  \C(\Omega)) \leq 0$ or, equivalently, $G_{f_0}(t, \C(\Omega)) \leq
  t^2/2$. We shall use the bound on $G_{f_0}(t, \C(\Omega))$ given in
  Lemma \ref{GtUBwcFD} to get $t$ such that $G_{f_0}(t, \C(\Omega)) \leq
  t^2/2$. The risk will be dominated by such $t$ because this $t$ will
  be of larger order than $\sigma^2/n$. For $d \leq 3$, Lemma
  \ref{GtUBwcFD} gives 
\begin{equation*}
  G_{f_0}(t, \C(\Omega)) \le \frac{C_d\sigma}{\sqrt{n}} (\log n)^{F/2} \left(t +
\lin^{d/4} t^{1-d/4} \right).
\end{equation*}
Because
\begin{equation*}
\frac{C_d \sigma}{\sqrt{n}} (\log n)^{F/2} t \le \frac{t^2}{4} ~~
  \text{ if and only if } ~~ t \geq \frac{4C_d \sigma}{\sqrt{n}} (\log
  n)^{F/2}
\end{equation*}
and
\begin{equation*}
  \frac{C_d \sigma}{\sqrt{n}} (\log n)^{F/2}
  \lin^{d/4} t^{1-d/4} \leq
  \frac{t^2}{4} ~~ \text{ iff } ~~ t \geq (4 C_d)^{\frac{4}{d+4}}
  \left(\frac{\sigma (\log n)^{F/2}}{\sqrt{n}} \right)^{\frac{4}{d+4}}
  \lin^{\frac{d}{d+4}},
\end{equation*}
we deduce that $G_{f_0}(t, f_0) \leq t^2/2$ for 
\begin{equation*}
t \geq C_d \max
    \left(\left(\frac{\sigma (\log n)^{F/2}}{\sqrt{n}} \right)^{\frac{4}{d+4}}
\lin^{\frac{d}{d+4}}, \frac{\sigma}{\sqrt{n}} (\log
  n)^{F/2}   \right).
\end{equation*}
This proves, for $d \leq 3$,
\begin{equation*}
   \E_{f_0} \ell_{\empdes}^2(\hat f_n, f_0) \leq    C_d\max \left\{ \lin^{\frac{2d}{4 + d}} \left(\frac{\sigma^2}{n} (\log n)^F
     \right)^{\frac{4}{d+4}}, \frac{\sigma^2}{n} (\log n)^F \right\}
 \end{equation*}
The leading term in the right hand side above is the first
term inside the maximum which proves Theorem \ref{wcad} for $d \leq
3$. 

For $d = 4$, Lemma \ref{GtUBwcFD} gives
\begin{equation*}
  G_{f_0}(t, \C(\Omega)) \leq \frac{C_d \sigma}{\sqrt{n}} \left(t + \lin \right) (\log n)^{1+(F/2)}
\end{equation*}
from which we can deduce (as in the case $d \leq 3$) that $G(t) \leq t^2/2$ for 
\begin{equation*}
t \geq C_d \max
   \left( \frac{\sqrt{\sigma \lin} (\log n)^{(F/4) + (1/2)}}{n^{1/4}}
   \frac{\sigma (\log n)^{1 + (F/2)}}{\sqrt{n}} \right). 
\end{equation*}
This proves, for $d = 4$,
\begin{equation*}
   \E_{f_0} \ell_{\empdes}^2(\hat f_n, f_0) \leq    C_4  \max \left\{\frac{\sigma
     \lin}{\sqrt{n}}  (\log n)^{1 + \frac{F}{2}},
     \frac{\sigma^2 }{n} (\log n)^{2 + F} \right\}. 
\end{equation*}
The leading term in the right hand side above is the first
term inside the maximum which proves Theorem \ref{wcad} for $d = 4$.

Finally, for $d \geq 5$, Lemma \ref{GtUBwcFD} gives
\begin{equation*}
  G_{f_0}(t, \C(\Omega)) \le 2\sigma \left(\frac{C_d (\log n)^{F/2}}{\sqrt{n}}
  \right)^{4/d} \left(t + \lin \right)
\end{equation*}
from which it follows that $G(t) \leq t^2/2$ for
\begin{equation*}
t \geq C_d \max
    \left(\sqrt{\sigma \lin} \left(\frac{(\log n)^{F/2}}{\sqrt{n}}
      \right)^{2/d}, \sigma \left(\frac{(\log n)^{F/2}}{\sqrt{n}}
      \right)^{4/d} \right). 
\end{equation*}
This proves, for $d \geq 5$, 
\begin{equation*}
   \E_{f_0} \ell_{\empdes}^2(\hat f_n, f_0) \leq    C_d \max \left\{ \sigma \lin \left(\frac{(\log n)^{F}}{n} \right)^{\frac{2}{d}},
   \sigma^2 \left(\frac{(\log n)^{F}}{n}
   \right)^{\frac{4}{d}}\right\}. 
\end{equation*}
The leading term in the right hand side above is the first
term inside the maximum which proves Theorem \ref{wcad} for $d \geq 5$.
\end{proof}

\subsection{Proof of Lemma \ref{GtUBadFD} and Theorem
  \ref{adrate}}\label{subs2} 
First we prove Lemma \ref{GtUBadFD} which crucially uses the metric
entropy bound \eqref{simmet1}. 

\begin{proof}[Proof of Lemma \ref{GtUBadFD}]
Combining the entropy bound \eqref{simmet1} and Dudley's result
(Theorem \ref{dud}), we get  
\begin{equation}\label{adint}
  G_{f_0}(t, \C(\Omega)) \leq \sigma \sqrt{\frac{k}{n}} (c_d \log
  n)^{h/2} \int_{\theta}^{t/2} \left(\frac{t}{\epsilon} \right)^{d/4}
  d\epsilon + 2 \sigma \theta
\end{equation}
for every $0 < \theta \leq t/2$. We now split into the three cases $d
\le 3$, $d = 4$ and $d \geq 5$. When $d \le 3$, we can take $\theta =
0$ to obtain 
\begin{equation*}
  G_{f_0}(t, \C(\Omega)) \leq t \sigma \sqrt{\frac{k}{n}} (c_d\log
  n)^{h/2}.
\end{equation*}
For $d = 4$, we get
\begin{equation*}
  G_{f_0}(t, \C(\Omega)) \leq \sigma \sqrt{\frac{k}{n}} (c_d\log n)^{h/2} t \log
  \frac{t}{2\theta} + 2 \sigma \theta.
\end{equation*}
Choosing $\theta := t/(2\sqrt{n})$, we get
\begin{equation*}
  G_{f_0}(t,\C(\Omega))\leq \sigma t \sqrt{\frac{k}{n}} (c_d\log n)^{1+h/2}.
\end{equation*}
For $d \geq 5$,  we get
\begin{align*}
  G_{f_0}(t, \C(\Omega)) &\le \sigma \sqrt{\frac{k}{n}} (c_d \log
  n)^{h/2} \int_{\theta}^{\infty} \left(\frac{t}{\epsilon} \right)^{d/4}
         d\epsilon + 2 \sigma \theta  \\
  &\leq C_d \sigma \sqrt{\frac{k}{n}} (c_d\log
  n)^{h/2} t^{d/4} \theta^{1-(d/4)} + 2 \sigma \theta.
\end{align*}
Take $\theta = t \left( (c_d\log n)^{h/2} \sqrt{\frac{k}{n}}
\right)^{4/d}$ to get
\begin{equation}
  \label{gtkp.eq}
G_{f_0}(t, \C(\Omega)) \leq \sigma t \left((c_d\log n)^{h/2}
    \sqrt{\frac{k}{n}} \right)^{4/d}.
\end{equation}
which completes the proof of Lemma \ref{GtUBadFD}.
\end{proof}

We next provide the proof of Theorem \ref{adrate}.

\begin{proof}[Proof of Theorem \ref{adrate}]
From Lemma \ref{GtUBadFD}, it is straightforward to see that
$G_{f_0}(t, \C(\Omega)) \leq t^2/2$ provided
  \begin{equation*}
t \geq
  \begin{cases}
 2 \sigma \sqrt{\frac{k}{n}} (c_d\log n)^{h/2} ~~~~~~~\text{for } d  \le 3 \\
2 \sigma \sqrt{\frac{k}{n}} (c_d\log n)^{1 + (h/2)}   ~~~~~~~~~ \text{for } d = 4 \\
2 \sigma  \left((c_d\log n)^{h/2} \sqrt{\frac{k}{n}} \right)^{4/d}    ~~~~~~  
 \text{for } d \geq 5. 
  \end{cases}
\end{equation*}  
Theorem \ref{GtUBadFD} then follows by an application of Theorem
\ref{chat}.  
\end{proof}

\subsection{Proof of Theorem \ref{minsub}}\label{minsubproof}
This basically follows from Theorem \ref{implb}. Let $c_d$ and
$N_d$ be as given by Theorem \ref{implb}. Letting $k = \sqrt{n}
\sigma^{-d/4}$ and assuming that $n \geq \max(N_d,
(c_d^{-2})\sigma^{-d/2}$), we obtain from Theorem \ref{implb} that
\begin{equation*}
  \sup_{f_0 \in \C^{C_d}_{C_d}(\Omega)}
    \E_{f_0} \ell_{\empdes}^2(\hat{f}_n, f_0)\geq c_d \sigma n^{-2/d}
    (\log n)^{-4(d+1)/d}.
  \end{equation*}
where $C_d$ is such that $\tilde{f}_k \in \C^{C_d}_{C_d}(\Omega)$
(existence of such a $C_d$ is guaranteed by Lemma \ref{ballapp}). The
required lower bound \eqref{minsub.eq} on the class $\C_L^L(\Omega)$
for an arbitrary $L > 0$ can now be obtained by an elementary scaling
argument.

\subsection{Proof of Theorem \ref{implb}}\label{subs4} 
We prove Theorem \ref{implb} below using Lemma \ref{allkt.lem}. The
proof of Lemma \ref{allkt.lem} is given in Subsection
\ref{allktproof}. 

\begin{proof}[Proof of Theorem \ref{implb}]
  Let $\tilde{G}(t) :=
  G_{\tilde{f}_k}(t, \C(\Omega))$ and $t_{\tilde{f}_k} :=
  t_{\tilde{f}_k}(\C(\Omega))$ for notational convenience. By the
  lower bound given by Lemma 
  \ref{allkt.lem},   
  \begin{equation*}
    \sup_{t > 0} \left(\tilde{G}(t) - \frac{t^2}{2}
    \right) \geq 
    \sup_{t \leq c_d k^{-2/d}} \left(c_d \sigma t
      \left(\frac{k}{n} \right)^{2/d} - \frac{t^2}{2} \right).
  \end{equation*}
  Taking $t = c_d \sigma (k/n)^{2/d}$  and noting that
  \begin{equation*}
    t = c_d \sigma \left(\frac{k}{n} \right)^{2/d} \leq c_d
    k^{-2/d} \qt{if and only if $k \leq \sqrt{n} \sigma^{-d/4}$},
  \end{equation*}
  we get that
  \begin{equation*}
    \sup_{t > 0} \left(\tilde{G}(t) - \frac{t^2}{2}
    \right) \geq 
    \frac{c_d^2 \sigma^2}{2} \left(\frac{k}{n} \right)^{4/d}.
  \end{equation*}
  The above inequality, combined with the upper bound \eqref{upbcl}
  and the fact that $t_{\tilde{f}_k}$ maximizes
  $\tilde{G}(t)  - t^2/2$ over all $t > 0$, yield 
  \begin{align*}
    \frac{c_d^2 \sigma^2}{2} \left(\frac{k}{n} \right)^{4/d}
    &\leq \sup_{t > 0} \left(\tilde{G}(t) -
      \frac{t^2}{2} \right) \\ &
                                 =
                                 \tilde{G}(t_{\tilde{f}_k}) -
                                 \frac{t^2_{\tilde{f}_k}}{2}
                                 \\ &\leq
     \tilde{G}(t_{\tilde{f}_k}) \leq C_d \sigma
      t_{\tilde{f}_k} (\log n)^{2(d+1)/d} \left(\frac{k}{n} \right)^{2/d}.
  \end{align*}
  This implies
  \begin{equation*}
    t_{\tilde{f}_k} \geq \frac{c_d^2}{2C_d} \sigma
    \left(\frac{k}{n} \right)^{2/d} (\log n)^{-2(d+1)/d}.
  \end{equation*}
  Theorem \ref{chat} then gives
  \begin{equation*}
    \E_{\tilde{f}_k}\ell_{\P_n}^2 \left(\hat{f}_n(\C(\Omega)), \tilde{f}_k \right) \geq
    \frac{c_d^4}{8 C_d^2} \sigma^2 \left(\frac{k}{n}
    \right)^{4/d} (\log n)^{-4(d+1)/d} - \frac{C \sigma^2}{n}.
  \end{equation*}
  The first term on the right hand side above
  dominates the second term when $n$ is larger than a constant
  depending on $d$ alone. This completes the proof of Theorem
  \ref{implb}.
\end{proof}

\subsection{Proofs of Lemma \ref{allkt.lem} and Lemma
  \ref{lbf0}}\label{allktproof}
We first prove Lemma \ref{allkt.lem} below assuming the validity of
Lemma \ref{lbf0}. Lemma \ref{lbf0} is proved later in this
subsection.

\begin{proof}[Proof of Lemma \ref{allkt.lem}]
    By Lemma \ref{ballapp}, $\tilde{f}_k$ satisfies
  \begin{equation}\label{app.con}
    \ell_{\empdes}(f_0, \tilde{f}_k) \leq \sup_{x \in \Omega} \left|f_0(x) -
      \tilde{f}_k(x) \right| \leq C_d k^{-2/d}.
  \end{equation}
  where $f_0(x) := \|x\|^2$. Assume first that $t =
  2C_d k^{-2/d}$ where $C_d$ is the constant from \eqref{app.con}. For
  this choice of $t$, it follows from the triangle inequality (and
  \eqref{app.con}) that
  \begin{equation*}
    \left\{f \in \C(\Omega) : \ell_{\P_n}(f, \tilde{f}_k) \leq t
    \right\} \supseteq \left\{f \in \C(\Omega) : \ell_{\P_n}(f,
      f_0) \leq C_d k^{-2/d} \right\}. 
  \end{equation*}
  This immediately implies
  \begin{equation*}
    G_{\tilde{f}_k}(t, \C(\Omega)) \geq G_{f_0}(C_d k^{-2/d},
    \C(\Omega)). 
  \end{equation*}
 Let $\tilde{G}(t) := G_{\tilde{f}_k}(t, \C(\Omega))$ and $G_0(t) :=
 G_{f_0}(t, \C(\Omega))$ in the rest of this proof (this is for
 notational convenience). We just proved $\tilde{G}(t) \geq G_0(C_d
 k^{-2/d})$ for $t = 2C_d k^{-2/d}$. By Sudakov minoration (Lemma
 \ref{sud}), $G_0(C_d k^{-2/d})$ is bounded from below by 
  \begin{align*}
\frac{\beta
      \sigma}{\sqrt{n}} \sup_{\epsilon > 0} \left\{\epsilon \sqrt{\log
                                        N(\epsilon,\left\{g \in \C(\Omega) : \ell_{\P_n}(f_0,
      g) \leq C_d k^{-2/d} \right\}, \ell_{\empdes})}
                                        \right\}.
  \end{align*}
Lemma \ref{lbf0} gives a lower bound on the metric entropy appearing
above. Applying this for $\epsilon = c_1 n^{-2/d}$, we get 
  \begin{equation*}
    \tilde{G}(t) \geq \frac{\beta \sigma}{\sqrt{n}} (c_1 n^{-2/d})
    \sqrt{\frac{n}{8}} = \frac{\beta c_1}{\sqrt{8}} \sigma n^{-2/d}
  \end{equation*}
  provided
  \begin{equation}\label{kcon1}
    k \leq \left(\frac{C_d}{c_2} \right)^{d/2} n.
  \end{equation}
  The condition above is necessary for the inequality $c_2 n^{-2/d}
  \leq C_d k^{-2/d}$ which is required for the application of Lemma
  \ref{lbf0}. This gives
  \begin{equation*}
    \tilde{G}(t) \geq \frac{\beta c_1}{\sqrt{8}} \sigma n^{-2/d}
    \qt{for $t = 2 C_d k^{-2/d}$}.
  \end{equation*}
  Now for $t \leq 2 C_d k^{-2/d}$, we use the fact that $x \mapsto
  \tilde{G}(x)$ is concave on $[0, \infty)$ (and that $\tilde{G}(0) =
  0$) to deduce that
  \begin{equation*}
    \frac{\tilde{G}(t)}{t} \geq \frac{\tilde{G}(2 C_d k^{-2/d})}{2C_d
      k^{-2/d}} \geq
    \sigma \left(\frac{k}{n} \right)^{2/d} \frac{\beta c_1}{2 \sqrt{8}
      C_d} \qt{for all $t \leq 2C_d k^{-2/d}$}.
  \end{equation*}
  This proves \eqref{allkt} for
  \begin{equation*}
    c_d = \min \left(\frac{\beta c_1}{2 \sqrt{8}
      C_d}, 2C_d, \left(\frac{C_d}{c_2} \right)^{d/2} \right), 
\end{equation*}
  completing the proof of Lemma \ref{allkt.lem}. 
\end{proof}

We next prove Lemma \ref{lbf0}. 

\begin{proof}[Proof of Lemma \ref{lbf0}]
  We shall use a perturbation result that is similar to the one
  mentioned at the beginning of the proof of Lemma
  \ref{lbf0.ran}. Because we are dealing with the discrete metric
  $\ell_{\P_n}$ here, this argument might be nonstandard so we provide
  an explicit construction (some aspects of this construction were
  also used in the proof of Proposition \ref{fixlobo}). Let 
\begin{equation*}
g(x_1,x_2,\ldots,x_d)=\twopiece{\sum_{i=1}^d
	\cos^3(\pi  x_i)}{(x_1,x_2,\ldots,x_d)\in[-1/2,1/2]^d}{0}{(x_1,x_2,\ldots,x_d)\notin[-1/2,1/2]^d}.
\end{equation*}
Note that $g$ is smooth, $\frac{\partial^2 g}{\partial x_i \partial
  x_j} = 0$ for $i \neq j$ and
\begin{equation*}
  \left|\frac{\partial^2 g}{\partial x_i^2} (x_1, \dots, x_d) \right|
  \leq \frac{4 \sqrt{2}}{3} \pi^2
\end{equation*}
which implies that the Hessian of $g$ is dominated by $(4\sqrt{2}
\pi^2/3)$ times the identity matrix. It is also easy to check that the
Hessian of $g$ equals zero on the boundary of $[-0.5, 0.5]^d$.

Now for every grid point $s := (k_1 \delta, \dots, k_d \delta)$ in
$\grid \cap \Omega$, consider the function
\begin{equation*}
  g_s(x_1, \dots, x_d) := \delta^2 g \left(\frac{x_1 - k_1
      \delta}{\delta}, \dots, \frac{x_d - k_d
      \delta}{\delta} \right).
\end{equation*}
Clearly $g_s$ is supported on the cube
\begin{equation*}
\prod_{j=1}^d  [(k_j - 1/2) \delta, (k_j + 1/2) \delta] 
\end{equation*}
and these cubes for different grid points have disjoint
interiors.

We now consider binary vectors in $\{0, 1\}^n$. We shall index each
$\xi \in \{0, 1\}^n$ by $\xi_s, s \in \grid \cap \Omega$. For every
$\xi = (\xi_s, s \in \grid \cap \Omega) \in \{0, 1\}^n$, consider the
function
\begin{equation}\label{gsdef}
  G_{\xi}(x) = f_0(x) + \frac{3}{4 \sqrt{2} \pi^2} \sum_{s \in
  \grid \cap \Omega} \xi_s g_s(x).
\end{equation}
It can be verified that $G_{\xi}$ is convex because $f_0$ has constant
Hessian equal to $2$ times the identity, the Hessian of $g_{s}$ is
bounded by $(4 \sqrt{2} \pi^2/3)$ and the supports of $g_{s}, s \in
\grid \cap \Omega$ have disjoint interiors. Note further that for
$\xi, \xi' \in \{0, 1\}^n$ and $s \in \grid \cap \Omega$,
\begin{equation*}
  G_{\xi}(s) - G_{\xi'}(s) = \frac{3 d \delta^2}{4 \sqrt{2} \pi^2}
  \left(\xi_s - \xi'_s \right).
\end{equation*}
This implies that
\begin{equation*}
  \ell_{\empdes}(G_{\xi}, G_{\xi'}) = \frac{3 d \delta^2}{4 \sqrt{2}
    \pi^2} \sqrt{\frac{\Upsilon(\xi, \xi')}{n}}
\end{equation*}
where $\Upsilon(\xi, \xi') := \sum_{s \in \grid \cap \Omega} I\{\xi_s
\neq \xi'_s\}$ is the Hamming distance between $\xi$ and $\xi'$. The
Varshamov-Gilbert lemma (see e.g., \cite[Lemma 4.7]{Massart03Flour})
asserts the existence of a subset $W$ of $\{0, 1\}^n$ with cardinality
$|W| \geq \exp(n/8)$ such that $\Upsilon(\xi, \xi') \geq n/4$ for all
$\xi, \xi' \in W$ with $\xi \neq \xi'$. We then have
\begin{equation*}
  \ell_{\empdes}(G_{\xi}, G_{\xi'}) \geq \frac{3 d \delta^2}{8 \sqrt{2}
    \pi^2} \qt{for all $\xi, \xi' \in W$ with $\xi \neq \xi'$}.
\end{equation*}
Inequality \eqref{nlob} then gives
\begin{equation*}
  \ell_{\empdes}(G_{\xi}, G_{\xi'}) \geq c_1 n^{-2/d} \qt{for all $\xi, \xi' \in W$ with $\xi \neq \xi'$}.
\end{equation*}
for a dimensional constant $c_1$. One can also
check that
\begin{equation*}
  \ell_{\empdes}(G_{\xi}, f_0) \leq \frac{3 d \delta^2}{4 \sqrt{2}
    \pi^2} \leq c_2 n^{-2/d}
\end{equation*}
for another dimensional constant $c_2$ completing the
proof of Lemma \ref{lbf0}.
\end{proof}

\subsection{Proofs  of inequality \eqref{simmet} and inequality
  \eqref{simmet1}}\label{fdsimproofs}
In this subsection, we provide proofs of the discrete metric entropy
results given by inequalities \eqref{simmet} and \eqref{simmet1}. We
shall assume and use Theorem \ref{polytope} for these proofs. The
proof of Theorem \ref{polytope} is given in the next subsection
(Subsection \ref{discmentproofs}).

\begin{proof}[Proof of Inequality \eqref{simmet}]
  By specializing to the case $p = 2$ and writing $\ell_{\P_n}$ for
  $\ell_{\grid}(\cdot, \Omega, 2)$, we can rephrase the conclusion of
  Theorem \ref{polytope} as:
  \begin{equation*}
    \log N\left(\epsilon, \{g \in \C(\Omega): \ell_{\P_n}(g, 0) \leq t \}, \ell_{\P_n}
    \right) \leq \left[c_d \log(1/\delta) \right]^F
    \left(\frac{t}{\epsilon} \right)^{d/2}. 
  \end{equation*}
  Here the $0$ in $\ell_{\P_n}(g, 0) \leq t$ refers to the function
  that is identically equal to zero. Because in the representation
  \eqref{parallel}, the number $F$ is assumed to be bounded from above
  by a constant depending on $d$ alone, we have $\left[c_d
    \log(1/\delta) \right]^F \leq C_d \log(1/\delta)$. We then use
  \eqref{nlob} to bound $\log (1/\delta)$ by a constant multiple of
  $\log n$. These give
  \begin{equation*}
    \log N\left(\epsilon, \{f \in \C(\Omega): \ell_{\P_n}(f, 0) \leq t \}, \ell_{\P_n}
    \right) \leq C_d (\log n)^F
    \left(\frac{t}{\epsilon} \right)^{d/2}. 
  \end{equation*}
  Because $f \in \C(\Omega)$ if and only if $f - g_0 \in \C(\Omega)$
  for every affine function $g_0$ on $\Omega$ (i.e., $g_0 \in
  \A(\Omega)$), we deduce
    \begin{equation*}
    \log N\left(\epsilon, \{f \in \C(\Omega): \ell_{\P_n}(f, g_0) \leq t \}, \ell_{\P_n}
    \right) \leq C_d (\log n)^F
    \left(\frac{t}{\epsilon} \right)^{d/2}
  \end{equation*}
  for every $g_0 \in \A(\Omega)$. By triangle inequality,
  \begin{equation*}
    \{f \in \C(\Omega): \ell_{\P_n}(f, f_0) \leq t\} \subseteq \{f \in
    \C(\Omega): \ell_{\P_n}(f, g_0) \leq t + \ell_{\P_n}(f_0, g_0)\}
  \end{equation*}
  for every $g_0 \in \A(\Omega)$. Thus
    \begin{equation*}
    \log N\left(\epsilon, \{f \in \C(\Omega): \ell_{\P_n}(f, f_0) \leq t \}, \ell_{\P_n}
    \right) \leq C_d (\log n)^F
    \left(\frac{t + \ell_{\P_n}(f_0, g_0)}{\epsilon} \right)^{d/2}
  \end{equation*}
  Because $g_0 \in \A(\Omega)$ is arbitrary in the right hand side
  above, we can take the infimum over $g_0 \in \A(\Omega)$. This
  allows us to replace $\ell_{\P_n}(f_0, g_0)$ on the right hand side
  above by $\inf_{g_0 \in \A(\Omega)} \ell_{\P_n}(f_0, g_0) \leq \lin$
  leading to the required inequality \eqref{simmet}.
\end{proof}

\begin{proof}[Proof of Inequality \eqref{simmet1}]
    Fix $f_0 \in \mathfrak{C}_{k, h}(\Omega)$. By definition of
    $\mathfrak{C}_{k, h}(\Omega)$, there exists $k$ subsets $\Omega_1,
    \dots, \Omega_k$ satisfying the following properties:
    \begin{enumerate}
\item $f_0$ is affine on each $\Omega_i$,
\item each $\Omega_i$ can be written as an intersection of at most $h$
  slabs (i.e., as in \eqref{parallel} with $F = h$), 
  and  
\item $\Omega_1 \cap \grid, \dots, \Omega_k \cap \grid$ are disjoint 
  with $\cup_{i=1}^k (\Omega_i \cap \grid) = \Omega \cap \grid$
\end{enumerate}
Note that $n$ is the cardinality of $\Omega \cap \grid$. Let $n_i$
   denote the cardinality of $\Omega_i \cap \grid$ for each $i =
   1, \dots, k$. We can assume that each $n_i > 0$ for otherwise we
   can simply drop that $\Omega_i$. For each $f \in \C(\Omega)$ such
   that $\ell_{\P_n}(f_0, f) \leq t$ and $1 \leq i \leq k$, let
   $\sigma_i(f)$ be the smallest positive integer for which
   \begin{equation}\label{abfre}
     \sum_{s \in \Omega_i \cap \grid} \left(f(s) - f_0(s) \right)^2 \leq n_i
     \sigma_i(f) t^2.
   \end{equation}
   Because $\ell_{\P_n}(f_0, f) \leq t$, we have $1 \le \sigma_i(f)
   \leq n$ for each $i$. Also because $\sigma_i(f)$ is the smallest
   integer satisfying \eqref{abfre}, we have
   \begin{equation*}
     \sum_{s \in \Omega_i \cap \grid} \left(f(s) - f_0(s) \right)^2 \geq
     n_i \left(\sigma_i(f) - 1 \right) t^2
   \end{equation*}
   which implies that
   \begin{align*}
     \sum_{i=1}^k n_i \left(\sigma_i(f) - 1 \right) t^2 &\leq
     \sum_{i=1}^k \sum_{s \in \Omega_i \cap \grid} \left(f(s) - f_0(s)
     \right)^2 \\ &= \sum_{s \in \Omega \cap \grid} \left(f(s) - f_0(s)
     \right)^2 \leq nt^2,
   \end{align*}
   leading to
   \begin{equation*}
     \sum_{i=1}^k n_i \sigma_i(f) \leq n + \sum_{i=1}^k n_i= 2n. 
   \end{equation*}
   Let
   \begin{equation*}
     \Sigma := \left\{(\sigma_1(f), \dots, \sigma_k(f)) : f \in
       \C(\Omega), \ell_{\P_n}(f, f_0) \leq t \right\}, 
   \end{equation*}
   and note that the cardinality of $\Sigma$ is at most $n^k$ as $1 \leq \sigma_i(f)
   \leq n$  for each $i$. For each $(\sigma_1, \dots, \sigma_k) \in
   \Sigma$, let
   \begin{equation*}
     \F_{\sigma_1, \dots, \sigma_k} = \left\{ f \in \C(\Omega):
       \ell_{\P_n}(f, f_0) \leq t \text{ and }  \sigma_i(f) =
       \sigma_i \text{ for each } i = 1, \dots, k
     \right\}.
   \end{equation*}
   By construction, we have
   \begin{equation*}
     \left\{f \in \C(\Omega) : \ell_{\P_n}(f, f_0) \leq t \right\} =
     \bigcup_{(\sigma_1, \dots, \sigma_k) \in \Sigma} \F_{\sigma_1,
       \dots, \sigma_k}
   \end{equation*}
   so that
   \begin{align}
     &\log N\left(\epsilon, \{f \in \C(\Omega) : \ell_{\P_n}(f, f_0)
       \leq t \}, \ell_{\P_n} \right) \nonumber \\
     &\leq \log \left(\sum_{(\sigma_1, \dots, \sigma_k) \in \Sigma}
       N\left(\epsilon, \F_{\sigma_1, \dots, \sigma_k}, \ell_{\P_n}
       \right)\right) \nonumber \\
     &\leq \max_{(\sigma_1, \dots, \sigma_k) \in \Sigma} \log
       N(\epsilon, \F_{\sigma_1, \dots, \sigma_k}, \ell_{\P_n}) + \log
       |\Sigma| \nonumber \\
     &\leq \max_{(\sigma_1, \dots, \sigma_k) \in \Sigma} \log
       N(\epsilon, \F_{\sigma_1, \dots, \sigma_k}, \ell_{\P_n}) + k
       \log n \label{jjcc2}
   \end{align}
   where $|\Sigma|$ denotes the cardinality of the finite set
   $\Sigma$, and we used $|\Sigma| \leq n^k$. We shall now upper-bound $\log
       N(\epsilon, \F_{\sigma_1, \dots, \sigma_k}, \ell_{\P_n})$ for a
       fixed $(\sigma_1, \dots, \sigma_k) \in \Sigma$. The idea is
       that this will follow from Theorem \ref{polytope} applied to
       each $\Omega_i, i = 1, \dots, k$. First observe that Theorem
       \ref{polytope} applied to $\Omega_i$ and $f$ replaced by $f -
       f_0$ gives
       \begin{equation}\label{jjcc}
         \begin{split}
        & \log N\left(\epsilon, \{f \in \C(\Omega_i): \ell_{\grid}(f -
         f_0, \Omega_i, 2) \leq t\}, \ell_{\grid}(\cdot, \Omega_i, 2)
          \right) \\ &\leq \left(c_d \log(1/\delta) \right)^h
                       \left(\frac{t}{\epsilon}\right)^{d/2}
        \end{split}                       
       \end{equation}
      for every $\epsilon, t > 0$. Here it is crucial that $f_0$ is
      affine on $\Omega_i$ (which ensures $f \in \C(\Omega) \iff f -
      f_0 \in \C(\Omega)$); also note that $F$ in \eqref{polytope.eq}
      is replaced by $h$ here because it is assumed that each
      $\Omega_i$ is of the form \eqref{parallel} with $F = h$ (this is
      part of the assumption $f_0 \in \mathfrak{C}_{k,
        h}(\Omega)$). Observe now that if
      \begin{align*}
        \frac{1}{n_i} \sum_{s \in \Omega_i \cap \grid} \left(f(s)  -
        f_0(s)\right)^2 \leq \frac{\epsilon^2}{2} \sigma_i
      \end{align*}
      for each $i = 1, \dots, k$, then
      \begin{align*}
        \frac{1}{n} \sum_{s \in \Omega \cap \grid} \left(f(s) - f_0(s)
        \right)^2 \leq \frac{\epsilon^2}{2n} \sum_{i=1}^k n_i
        \sigma_i(f) \leq \epsilon^2. 
      \end{align*}
      As a result
      \begin{align*}
        &\log N\left(\epsilon, \F_{\sigma_1, \dots, \sigma_k}, \ell_{\P_n} \right) \\
        &\leq \sum_{i=1}^k \log N\left(\epsilon \sqrt{\sigma_i/2}, \{f
          \in \C(\Omega_i) : \ell_{\grid}(f-f_0, \Omega_i, 2) \leq t
          \sqrt{\sigma_i}\}, \ell_{\grid}(\cdot, \Omega_i, 2) \right)  
      \end{align*}
      Inequality \eqref{jjcc} then gives
      \begin{align*}
        \log N\left(\epsilon, \F_{\sigma_1, \dots, \sigma_k},
        \ell_{\P_n} \right) \leq k \left(c_d \log(1/\delta) \right)^h 
                       \left(\frac{t}{\epsilon}\right)^{d/2}
      \end{align*}
      The proof of \eqref{simmet1} is now completed by combining the
      above inequality with \eqref{jjcc2} (note that $\log n \leq C_d
      \log(1/\delta)$ because of  \eqref{nlob}). 
\end{proof}

\subsection{Proof of Theorem \ref{polytope}} \label{discmentproofs}
This proof has two main steps. In the first step (which is the focus
of the next subsection), we stay away from the boundary and prove the
entropy bound with a modified 
metric that is defined only in the interior of the domain $\Omega$. In
the second step (which is the focus of Subsection \ref{reachbound}),
we extend this result to reach the boundary of $\Omega$.

Throughout this proof, we use the setting described in Subsection
\ref{minfd}. In particular, $\grid$ is the regular
$d$-dimensional $\delta$-grid \eqref{frg} and $\Omega$ is of the form
\eqref{parallel} and satisfies \eqref{OmegaAss}. $X_1, \dots,
X_n$ are an enumeration of the points in  $\Omega \cap \grid$ with $n$
denoting the cardinality of $\Omega \cap \grid$. Also $\ell_{\grid}(f,
\Omega, p)$ is defined as in \eqref{ellSdef}. 

\subsubsection{Away from the boundary}
The goal of this subsection is to prove the following proposition,
which is the equivalent version 
of Lemma \ref{Corr:shrink} in this discrete setting. Let $\omega_0$
denote the center of the John ellipsoid of $\Omega$ (recall that the
John ellipsoid of $\Omega$ is the unique ellipsoid of maximum volume
contained in $\Omega$). For $\lambda>0$, let
\begin{align}\label{omegasub}
 \Omega_\lambda = \omega_0+\lambda(\Omega-\omega_0) 
\end{align}
and note that
$|\Omega_\lambda|=\lambda^d|\Omega|$, where $|\Omega_{\lambda}|$ and
$|\Omega|$ denote volumes of $\Omega_{\lambda}$ and $\Omega$.    

\begin{proposition}\label{prop}
Suppose $\delta \leq r_d/(400 d^{3/2})$. Fix $t>0$ and
$0<\epsilon<1$. Then there exists a set $\cN$ consisting of
$\exp(\gamma_d\cdot(t/\epsilon)^{d/2})$ functions such that for 
        every $f\in \C(\Omega)$ with $\ell_{\grid}(f, \Omega, p) \leq
        t$, there exists $g\in \cN$ satisfying $|f(x)-g(x)|<\epsilon$
        for all $x\in \Omega_{0.9}$, where $\Omega_{0.9}$ is as
        defined in \eqref{omegasub} and $\gamma_d$ is a constant
        depending only on $d$. 
\end{proposition}

Proposition \ref{prop} states that, under the supremum
    metric on the interior set $\Omega_{0.9}$,  the metric entropy of
    $\left\{f \in \C(\Omega) : 
      \ell_{\grid}(f, \Omega, p) \leq t \right\}$  is atmost $\gamma_d
    (t/\epsilon)^{d/2}$. To prove Proposition \ref{prop}, we need some
    preliminary results. The next result shows that the number of grid
    points contained in $\Omega$ is of order $\delta^{-d}$ (this
    provides a proof of the claim \eqref{nlob}).  

\begin{lemma}\label{points}
      Suppose $\delta \leq r_d/(10 d^{3/2})$ where $r_d$ is the
      quantity from \eqref{OmegaAss}. Then
      \begin{align*}
        \frac{9}{10}|\Omega|\delta^{-d}\le n \le
        \frac{11}{10}|\Omega|\delta^{-d}
      \end{align*}
      where $|\Omega|$ is the volume of $\Omega$. 
\end{lemma}

\begin{proof} 
First observe that
	$$\bigcup_{i=1}^n(X_i+[-\delta/2, \delta/2]^d)\subseteq
        \Omega+[-\delta/2, \delta/2]^d\subseteq
        \Omega+\frac{\sqrt{d}\delta}{2}B_d.$$ 
	Because of \eqref{OmegaAss}, $\Omega$ contains the ball of
        radius $r_d$ centered at the origin which implies
	$$|\Omega+\frac{\sqrt{d}\delta}{2}B_d|\le
        \left(1+\frac{\sqrt{d}\delta}{2r_d}\right)^d|\Omega|\le
        \left(1+\frac{1}{20d}\right)^{1/d}|\Omega|\le
        \frac{11}{10}|\Omega|.$$ 
	Volume comparison with the last pair of displayed relations
        gives us $$n\le \frac{11}{10}|\Omega|\delta^{-d}.$$ 
	
	On the other hand, let $U$ be the union of the cubes
        $X_i+[-\delta/2,\delta/2]^d$. The volume of $U$ is
        $n\delta^d$. Since the union of $X_i+[-\delta,\delta]^d$
        covers $\Omega$, we have $U+[-\delta/2,\delta/2]^d\supseteq
        \Omega$. In particular, $U$ contains the set  
	$$\{x\in \Omega: {\rm dist}(x,\partial \Omega)\ge
        \sqrt{d}\delta/2\}.$$ 
	Since $\Omega$ contains the ball of radius $r_0$ centered at
        the origin, if we let
	$$\widehat{\Omega}=\left(1-\frac{\sqrt{d}\delta}{2r_d}\right)\Omega,$$
	then the distance between any $x\in \widehat{\Omega}$ and $\partial \Omega$ is at least $\sqrt{d}\delta/2$. Hence $U\supset \widehat{\Omega}$. Consequently
	$$n=|U|\delta^{-d}\ge
        |\widehat{\Omega}|\delta^{-d}=\left(1-\frac{\sqrt{d}\delta}{2r_d}\right)^d
        |\Omega|\delta^{-d}.$$
        Using $\delta \leq r_d/(10 d^{3/2})$ and $(1 - 1/(20d))^d \geq
        9/10$, the above implies that $n \geq (9/10) |\Omega|
        \delta^{-d}$ completing the proof of Lemma \ref{points}. 
\end{proof}

\begin{lemma}\label{L-bound}
  Suppose $\delta \leq r_d/(10 d^{3/2})$. Then $\inf_{x \in \Omega}
  f(x) \geq -20 d t$ for every $f \in \C(\Omega)$ with
  $\ell_{\grid}(f, \Omega, p) \leq t$. 
\end{lemma}
      
\begin{proof}
Let $x_0$ be the minimizer of $f$ on $\Omega$. If $f(x_0)\ge 0$, then 
there is nothing to prove; otherwise, the set $K:=\set{x\in
  \Omega}{f(x)\le 0}$ is a closed convex set containing $x_0$. 
Denote $K(t)=x_0+t(K-x_0)$, and let $\widehat{K}=K({1+\zeta})\setminus
K({1-\zeta})$, where $\zeta :=(10d)^{-1}$. We show that for all $x\in
\Omega\setminus \widehat{K}$, $|f(x)|\ge \zeta|f(x_0)|$. Indeed, if
we define a function $g$ on $\Omega$ so that $g(x_0)=f(x_0)$,
$g(\gamma)=f(\gamma)$ for all $\gamma\in \partial K$, and $g$ is
linear on $L_\gamma:=\set{x=x_0+t(\gamma-x_0)\in \Omega}{t\ge
  0}$. Then, by the convexity of $f$ on each $L_\gamma$, we have
$|f(x)|\ge |g(x)|$ on $\Omega$. Thus, for all $x\in \Omega\setminus
\widehat{K}$, 
$$|f(x)|\ge
|g(x)|=|g(\gamma)|+\frac{\|x-\gamma\|}{\|x_0-\gamma\|}|f(x_0)|\ge
\zeta|f(x_0)|.$$ 

Next, we show that most of the grid points in $\Omega$ are outside
$\widehat{K}$. Indeed, If $s$ is a grid point in $\widehat{K}$, then
$s+[-\delta/2,\delta/2]^d\subset K({1+\zeta})\cap \Omega
+[-\delta/2,\delta/2]^d$ and at least one half of the cube
$s+[-\delta/2,\delta/2]^d$ lies outside $K({1-\zeta})$. 
Thus, the number of grid points in $\widehat{K}$ is bounded
by $$2|(K({1+\zeta})\cap \Omega+[-\delta/2,\delta/2]^d)\setminus
K({1-\zeta})|\delta^{-d}.$$ 
Since by the Minkowski-Steiner formula (see e.g., \cite{Schneider}),
$|(A+B)\setminus A|$ can be expressed as a sum of mixed volumes of $A$
and $B$, and because the mixed volumes are monotone, we have 
$$|(A+B)\setminus B|\le |(C+D)\setminus D|$$ for all
convex sets $C\supset A$ and $D\supset B$. This gives 
\begin{align*}
&|(K({1+\zeta})\cap \Omega+[-\delta/2,\delta/2]^d)\setminus
K({1+\zeta})\cap \Omega| \\ &\le
|([-1,1]^d+[-\delta/2,\delta/2]^d)\setminus
[-1,1]^d| =(2+\delta)^d-2^d. 
\end{align*}
Also
$|K({1+\zeta})\cap \Omega \setminus K({1-\zeta})|\le
\left[1-\left(\frac{1-\zeta}{1+\zeta}\right)^d\right]|K({1+\zeta})\cap\Omega|$
so that
\begin{align*}
&|(K({1+\zeta})\cap \Omega+[-\delta/2,\delta/2]^d)\setminus
  K({1-\zeta})| \\ &\le
  \left[1-\left(\frac{1-\zeta}{1+\zeta}\right)^d\right]|K({1+\zeta})\cap\Omega|+(2+\delta)^d-2^d\le 
  3d\zeta|\Omega|.   
\end{align*}
Thus, the number of grid points in $\widehat{K}$ is bounded by
$$6d\zeta|\Omega|\delta^{-d}\le
6d\zeta\cdot\frac{10}{9}n\le 7d\zeta
n.$$ 
Hence,
$$n t^p\ge \sum_{s\in \grid\cap (\Omega\setminus
  \widehat{K})}|f(s)|^p\ge  (1-7d\zeta) n\cdot
(\zeta |f(x_0)|)^p,$$ 
which implies that $f(x_0)\ge -2^{1/p}\zeta^{-1}t\ge -20dt$ by using
$\zeta=(10d)^{-1}$ and $2^{1/p} \leq 2$ for $p \in [1, \infty)$.  
\end{proof}

\begin{lemma}\label{volume}
	Suppose $\delta \leq r_d/(400 d^{3/2})$. Then, at any point
        $P$ on the boundary of $\Omega_{0.95}$, any hyperplane passing through
        $P$ cuts $\Omega$ into two parts. The part that does not
        contain the center of John ellipsoid of $\Omega$ as its
        interior point contains at least $(20d)^{-d-1}\cdot n$ grid
        points. 
\end{lemma}
\begin{proof}
	Since $P$ is on the boundary of $\Omega_{0.95}$, any
        hyperplane passing through $P$ cuts $\Omega$ into two
        parts. Suppose $L$ is a part that does not contain the center
        of John ellipsoid of $\Omega$ as its interior. We prove that
        $|L|\ge \frac{1}{2(20d)^d}|\Omega|$. Because the ratio
        $|L|/|\Omega|$ is invariant under affine transform, we
        estimate $|TL|/|T\Omega|$, where $T$ is an affine transform so
        that the John ellipsoid of $T\Omega$ is the  unit ball
        $B_d$. Then, it is known that $T\Omega$ is contained in a ball
        of radius $d$ (see e.g., \cite[Lecture
        3]{ball1997elementary}). Because the 
        distance from $(T\Omega)_{0.95}$ to  
        the boundary of $T\Omega$ is at least $\frac{1}{20}$, $TL$
        contain half of the ball with center at $TP$ and radius
        $\frac{1}{20}$. Thus, $TL$ has volume at least $\frac12
        20^{-d}|B_d|$. Since $T\Omega$ is contained in the ball of
        radius $d$, we have $|T\Omega|\le d^d|B_d|$. This implies that
        $|TL|\ge \frac1{2d}(20d)^{-d}|T\Omega|$. Hence $|L|\ge
        \frac1{2d}(20d)^{-d}|\Omega|$. 
	
	Because the John ellipsoid of $\Omega$ contains a ball of
        radius at least $400d^{3/2}\delta$, the distance from
        $\Omega_{0.95}$ to the boundary of $\Omega$ is at least
        $20d^{3/2}\delta$. Thus, $L$ contains a ball of radius at
        least $10d^{3/2}\delta$. By Lemma~\ref{points}, the number of
        grid points in it is at least $\frac{9}{10}|L|\delta^{-d}\ge
        \frac9{20d}(20d)^{-d}|\Omega| \delta^{-d}$. The statement of
        Lemma~\ref{volume} then follows by using Lemma~\ref{points}
        one more time. 
\end{proof}

\begin{lemma}\label{U-bound}
  Suppose $\delta \leq r_d/(400 d^{3/2})$. Then $\sup_{x \in
    \Omega_{0.95}} f(x) \leq (20d)^{\frac{d+1}{p}} t$ for every $f \in
  \C(\Omega)$ with $\ell_{\grid}(f, \Omega, p) \leq t$.  
\end{lemma}
\begin{proof}
	Let $z$ be the maximizer of $f$ on $\Omega_{0.95}$. By
        convexity of $f$, the point $z$ must be on the boundary of 
        $\Omega_{0.95}$. If $f(z)\le 0$, there is nothing to prove. So
        we assume $f(z)>0$. The convexity of $f$ implies that $z$ lies
        on the boundary of the convex set $K:\{x\in \Omega: f(x)\le
        f(z)\}\supset \Omega_\eta$. There exists a hyperplane  $z$ so
        that the convex set $\set{x}{f(x)\le f(z)}$ lies entirely on
        one side of the hyperplane. Let $L$ be the portion of $\Omega$
        that lies on the other side of the hyperplane that support $K$
        at $z$. This hyperplane cuts $\Omega$ into two parts. Let $L$
        be the part that does not contain $K$. Then, $f(x)\ge f(z)$
        for all $x\in L$. Let $m$ denote the cardinality of $L \cap
        \grid$. By Lemma~\ref{volume}, we have
        \begin{align*}
          m \ge (20d)^{-d-1} n.
        \end{align*}
	Since $f(x)\ge f(z)>0$ for all $x\in L$, we have 
	$$m f(z)^p\le \sum_{s\in L\cap \grid}|f(s)|^p\le \sum_{s\in
          \Omega\cap \grid }|f(s)|^p\le n t^p.$$  
	We obtain
        $f(z)\le (20d)^{\frac{d+1}{p}} t$ by combining the above two 
        displayed inequalities, and this proves Lemma \ref{U-bound}.  
\end{proof}

We are now ready to prove Proposition \ref{prop}.

\begin{proof}[Proof of Proposition \ref{prop}]
Fix $f \in \C(\Omega)$ with $\ell_{\grid}(f, \Omega, p) \leq t$. By
Lemma~\ref{L-bound} and 
Lemma~\ref{U-bound}, we have $-20d t\le f(x)\le (20d)^{d+1}t$ for all
$x\in \Omega_{0.95}$. Let $T$ be an affine transformation so that the
John ellipsoid of $T\Omega_{0.95}$ equals the unit ball $\balld$. Because
$\Omega_{0.9}\subseteq (\Omega_{0.95})_{0.95}$, by the proof of
Lemma~\ref{volume}, the distance between the boundary of
$T(\Omega_{0.95})$ and the boundary of $T(\Omega_{0.9})$ is at least
$\frac1{20}$. If we define convex function $\widetilde{f}$ on
$T(\Omega)$ by $\widetilde{f}(y)=f(T^{-1}(y))$. Then, $-20dt\le
\widetilde{f}(y)\le (20d)^{d+1} t$ for all $y\in
T(\Omega_{0.95})$.

For $u,v\in T(\Omega_{0.9})$, assume without loss of
generality that $\widetilde{f}(u)\le \widetilde{f}(v)$ and consider
the half-line starting from $u$ and passing through $v$. Suppose the
half-line intersects the boundary of $T(\Omega_{0.9})$ at the point $a$ and the
boundary of $T(\Omega_{0.95})$ at the point $b$. By the convexity of 
$\widetilde{f}$ on this half-line, we have
$$0\le \frac{\widetilde{f}(v)-\widetilde{f}(u)}{\|v-u\|}\le
\frac{|\widetilde{f}(b)-\widetilde{f}(a)|}{\|b-a\|}\le
20[(20d)^{d+1}t+20dt]:=M.$$ 
This implies that $\widetilde{f}$ is a convex function on
$T(\Omega_{0.9})$ that has a Lipschitz constant $M$. Of course $f$ is
also bounded by $M$ on $T(\Omega_{0.9})$. Thus by the classical result
of \cite{Bronshtein76} on the metric entropy (in the supremum metric)
of bounded Lipschitz convex functions, there exists a finite set of
functions $\cG$ 
consisting of $\leq \exp(\beta\cdot (M/\epsilon)^{d/2})$ functions 
such that for every $f \in \C(\Omega)$ with $\ell_{\grid}(f, \Omega,
p) \leq t$, there exists  $g\in
\cG$, such that 
$\sup_{y \in T(\Omega_{0.9})}|\widetilde{f}(y)-g(y)|<\epsilon$.
This implies $\sup_{x \in \Omega_{.9}} |f(x)-g(T(x))|<\epsilon$. Thus,
by setting $\cN=\set{g\circ T}{g\in \cG}$ the lemma follows with
$\gamma_d\ge \beta (M/t)^{d/2}$ (note that $M$ is a multiple of $t$ so
that $M/t$ is a constant depending on $d$ alone). 
\end{proof}

\subsubsection{Reaching the Boundary}\label{reachbound} 
Now, we try to reach closer to the boundary of $\Omega$. More
precisely, we will extend Proposition~\ref{prop} from $\Omega_{0.9}$
to the set $\Omega_0$ defined below. For this section, it will be
convenient to rewrite $\Omega$ as:
\begin{align*}
  \Omega := \{x\in \R^d: -\mathfrak{a}_i \le v_i^T(x-\omega_0) \le
  \mathfrak{b}_i, 1\le i\le F\} 
\end{align*}
where $\omega_0$ is the center of the John ellipsoid of $\Omega$, and
$\mathfrak{a}_i, \mathfrak{b}_i > 0$. As in \eqref{parallel}, $v_1,
\dots, v_F$ are unit vectors.   

Let $m_i$ and $n_i$ be the smallest integers such that
$2^{-m_i} \mathfrak{a}_i\le\delta$ and $2^{-n_i} \mathfrak{b}_i\le
\delta$. Let $$\Omega_0= 
\{x\in \R^d: -(1-2^{-m_i}) \mathfrak{a}_i\le v_i^T(x-\omega_0)\le
(1-2^{-n_i}) \mathfrak{b}_i, 1\le 
i\le F\}.$$ Note that $\Omega_0$ is quite close to $\Omega$ because
the Hausdorff distance between them is at most $\delta$.

The following proposition suggests that to achieve our goal of
bounding the metric entropy of $\{f \in \C(\Omega): \ell_{\grid}(f,
\Omega, p) \leq t\}$ on $\Omega_0$, we only
need to properly decompose $\Omega_0$. 
\begin{proposition}\label{prop2}
	Suppose $D_i$, $1\le i\le m$ is a sequence of convex subsets of
        $\Omega$ such that no point in $\Omega$ is contained in more
        than $M$ subsets in the sequence. Further suppose that
        $\Omega_0\subset\cup_{i=1}^m (D_i)_{0.9}$. Then 
	$$\log N(\epsilon, \{f \in \C(\Omega): \ell_{\grid}(f,
\Omega, p) \leq t\},
        \ell_{\grid}(\cdot,\Omega_0, p))\le
        cmM^{\frac{d}{2p}}\left(\frac{t}{\epsilon} \right)^{d/2}.$$  
\end{proposition}
\begin{proof}
	Let $G_i$ be the set of grid points in $D_i$, and $\grid_i$ be the
        grid points in $(D_i)_{0.9}\setminus
        \cup_{j<i}(D_j)_{0.9}$. For every $f\in \C(\Omega)$ such that
        $\ell_{\grid}(f, \Omega, p) \leq t$,
        define $t_i=t_i(f)$ to be the smallest positive integer, such
        that 
	$$\sum_{x\in G_i}|f(x)|^p\le |G_i|t_i t^p.$$
	Because $t_i$ is the smallest positive integer satisfying the
        above, the inequality will be reversed for $t_i - 1$ allowing
        us to deduce 
	$$\sum_{i=1}^m |G_i|(t_i-1) t^p\le \sum_{i=1}^m \sum_{x\in
          G_i} |f(x)|^p\le Mnt^p,$$
        which is equivalent to $\sum_{i=1}^m |G_i| (t_i - 1) \leq M
        n$. The  number of possible values for each $t_i$ is clearly
        at most $Mn + m$ which implies that there are no more than
        ${{Mn+m}\choose{m}}$ possible values of $(t_1, t_2, \ldots,
        t_m)$.
        
	Let $$\cK=\{(k_1,k_2,\ldots, k_m)\in \mathbb{N}^m:
        \sum_{i=1}^m |G_i|(k_i-1)\le Mn\}.$$ 
	For each $K=(k_1,k_2,\ldots, k_m)\in \cK$, define
        \begin{align*}
        \F_K=\{f\in \C(\Omega): \ell_{\grid}(f, \Omega, p) \leq t
          \text{ and }  t_i(f)=k_i, 1\le i\le
        m\}. 
        \end{align*}
  Then $\F_K\subset \{f \in \C(D_i) : \ell_{\grid}(f, D_i, p) \leq
  k_i^{1/p} t\}$ for each $i$. Applying Proposition~\ref{prop} for
  $D_i$, there exists a set $\G_i$ consisting of $\exp(\gamma
  [k_i^{1/p}t]^{d/2}\epsilon_i^{-d/2})$ functions, such that for every
  $f \in F_K$, there exists $g_i\in \G_i$ satisfying 
	$$\sum_{x\in \grid_i}|f(x)-g_i(x)|^p\le |G_i|\epsilon_i^p.$$  
	If we define $g(x)=g_i(x)$ for $x\in \grid_i$, then we have
	$$\sum_{x\in \grid\cap \Omega_0}|f(x)-g(x)|^p\le
        \sum_{i=1}^m|G_i|\epsilon_i^p=n\epsilon^p,$$ 
	where the last inequality holds if we let
	$$\epsilon_i=\frac{t_i^{\frac{d}{p(d+2p)}}|G_i|^{-\frac{2}{d+4p}}}{\left(\sum_{i=1}^m(|G_i|k_i)^{\frac{d}{d+2p}}\right)^{1/p}}n^{1/p}\epsilon.$$
	The total number of possible choices for $g$ is therefore
        \begin{align*}
&\exp\left(\gamma\sum_{i=1}^m
          [k_i^{1/p}t]^{d/2}\epsilon_i^{-d/2}\right) \\
&=\exp\left\{\left(\sum_{i=1}^m(|G_i|k_i)^{\frac{d}{d+2p}}\right)^{\frac{d+4p}{2p}}\cdot  
          n^{-\frac{d}{2p}}(t/\epsilon)^{d/2}\right\}.
        \end{align*}
	Using the inequalities
	$$\sum_{i=1}^m(|G_i|k_i)^{\frac{d}{d+2p}}\le \left(\sum_{i=1}^m|G_i|k_i\right)^{\frac{d}{d+2p}}m^{\frac{2p}{d+4p}}$$
	and
	$$\sum_{i=1}^m|G_i|k_i=\sum_{i=1}^m|G_i|(k_i-1)+\sum_{i=1}^m|G_i|\le Mn+Mn=2Mn,$$
	we can bound the total number of realizations of $g$ by
	$$\exp\left(\gamma(2M)^{\frac{d}{2p}}m(t/\epsilon)^{d/2}\right).$$
	Consequently, we have
        \begin{align*}
	&\log N(\epsilon, \{f \in \C(\Omega): \ell_{\grid}(f, \Omega, p)
          \leq t\},
        \ell_{\grid}(\cdot,\Omega_0, p)) \\ &\le \log
        {{Mn+m}\choose{m}}+\gamma(2M)^{\frac{d}{2p}}m(t/\epsilon)^{d/2}\le
        cmM^{\frac{d}{2p}}(t/\epsilon)^{d/2}.          
        \end{align*}
      \end{proof}

In the next result, we decompose $\Omega_0$ according to the requirement
of Lemma \ref{prop2}. Recall that $m_i$ and $n_i$ are of order
$\log(1/\delta)$ and they have been defined at in the beginning of
this subsection (Subsection \ref{reachbound}).  

\begin{lemma}\label{indlemma}
  Let $N := \prod_{i=1}^F(m_i + n_i)$. There exists convex sets
  $\widehat{D}_i$, $1\le i\le N$ contained in $\Omega$, such that no
  point in $\Omega$ is contained in more than $4^F$ of these sets, and 
	$$\Omega_0\subset\cup_{i=1}^N (\widehat{D}_i)_{0.8}.$$
\end{lemma}
\begin{proof}
	Let $$\cK=\{(k_1, k_2, \ldots, k_F): -m_i\le k_i\le n_i-1,
        1\le i\le F\}.$$ There are $\prod_{i=1}^F(m_i+n_i)$ elements
        in $\cK$. 
	For each $K=(k_1, k_2, \ldots, k_F)\in \cK$, define
	$$D_K=\{x\in \R^d: \alpha_i(k_i)\le v_i^T(x-\omega_0)\le
        \alpha_i(k_i+1)\},$$
        where $$\alpha_i(t)=\twopiece{-(1-2^{t})\mathfrak{a}_i}{t\le
          0}{(1-2^{-t})\mathfrak{b}_i}{t>0}.$$ 
	$D_K$ is a convex set.
	The union of all $D_K$, $K\in \cK$ is the set $$\{x\in \R^d:
        -(1-2^{-m_i})\mathfrak{a}_i\le v_i^T(x-\omega_0)\le
        (1-2^{-n_i})\mathfrak{b}_i\}.$$ 
	Similarly, we define
	$$\widehat{D}_K=\{x\in \R^d: \beta_i(k_i)\le v_i^T(x-c)\le \gamma_i(k_i)\},$$
	where
        \begin{align*}
   &
     \beta_i(k_i)=\alpha_i(k_i)-\frac14[\alpha_i(k_i+1)-\alpha_i(k_i)],
     \text{ and }
          \\
&\gamma_i(k_i)=\alpha_i(k_i+1)+\frac14[\alpha_i(k_i+1)-\alpha_i(k_i)].
        \end{align*}
	Let $\omega_0(K)$ be the center of John ellipsoid of
        $\widehat{D}_K$. Because $\widehat{D}_K$ equals
        \begin{align*}
&\left\{x\in \R^d: \beta_i(k_i)-v_i^T(\omega_0(K)-\omega_0)\le
        v_i^T(x-\omega_0(K)) \right. \\ &\left.\le
          \gamma_i(k_i)-v_i^T(\omega_0(K)-\omega_0)\right\},           
        \end{align*}
we have
	\begin{align*}
&	(\widehat{D}_K)_{0.8}\\ =&\left\{x:
                                   0.8[\beta_i(k_i)-v_i^T(\omega_0(K)-\omega_0)]\le 
                                   v_i^T(x-\omega_0(K)) \right. \\ &\le \left.  0.8[\gamma_i(k_i)-v_i^T(\omega_0(K)-\omega_0)]\right\}\\
	=&\left\{x: 0.8\beta_i(k_i)+0.2 v_i^T(\omega_0(K)-\omega_0)\le
           v_i^T(x-\omega_0) \right. \\ & \left.\le 0.8\gamma_i(k_i)+0.2v_i^T(\omega_0(K)-\omega_0)\right\}\\
	\supseteq&\{x: 0.8\beta_i(k_i)+0.2\alpha_i(k_i+1)\le
                   v_i^T(x-\omega_0(K))\le
                   0.8\gamma_i(k_i)+0.2\alpha_i(k_i)\}\\ 
	=&\{x: \alpha_i(k_i)\le v_i^T(x-\omega_0(K))\le \alpha_i(k_i+1)\} =D_K,
	\end{align*}
	where in the second to the last equality we used the fact that $0.8\beta_i(k_i)+0.2\alpha_i(k_i+1)=\alpha_i(k_i)$ and $0.8\gamma_i(k_i)+0.2\alpha_i(k_i)=\alpha_i(k_i+1)$.
	
	It can be checked that when the integer $k_i\notin \{-1, 0\}$,
        the intervals $(\beta_i(k_i), \gamma_i(k_i))$ and
        $(\beta_i(j_i), \gamma_i(j_i))$ intersect only when
        $|k_i-j_i|\le 1$ or when $j_i=0$ or when
        $j_i=-1$. Hence where are at most four possibilities which
        implies that  no point can be contained in more than $4^F$
        different sets $\widehat{D_K}$. Lemma \ref{indlemma} follows
        by renaming these sets as $\widehat{D}_i$, $1\le i\le N$. 
\end{proof}

We are finally ready to complete the proof of Theorem \ref{polytope}.

\begin{proof}[Proof of Theorem \ref{polytope}]
By Lemma \ref{prop2} and Lemma \ref{indlemma}, we have 
$$\log N(\epsilon, \{f \in \C(\Omega) : \ell_{\grid}(f, \Omega, p)
\leq t\}, \ell_{\grid}(\cdot,\Omega_0, p))\le
c2^{\frac{dF}{2p}}\left(\log \frac{1}{\delta}\right)^{F}\left(\frac{t}{\epsilon}
\right)^{d/2}.$$ 

Because the distance between the boundary of $\Omega$
  and the boundary of $\Omega_0$ is no larger than $\delta$, the set 
$\Omega\setminus 
\Omega_0$ can be decomposed into at most $2F$ pieces of width
$\delta$. By Khinchine's flatness theorem, the grid points in
$\Omega\setminus \Omega_0$ are contained in $cF$ hyperplanes for some
constant $c$.  The intersection of $\Omega$ and each of these 
hyperplanes is a $(d-1)$ dimensional convex
polytope. This enables us 
to obtain covering number estimates on $\Omega\setminus \Omega_0$
using lower dimensional estimates. Because the desired covering number
estimate is known to be true for $d=1$, the result follows from
mathematical induction on dimension. This concludes the proof of Theorem
\ref{polytope}.  
\end{proof}

\bibliographystyle{imsart-number}
\bibliography{AG}

\end{document}